\documentclass[a4paper, 10pt]{amsproc}
\usepackage{defs}
\usepackage{subcaption}

\title[$\quito$ $\vertwo$: transcription, localization, and refinement]{\quito{} $\vertwo$: Trajectory Optimization with Uniform Error Guarantees under Path Constraints}
\thanks{S. Ganguly is with \faGroup Applied Mathematics and Physics, Graduate School of Informatics, \faUniversity Kyoto University, Japan. R.\ A.\ D'Silva is with \faGroup\ Department of Aerospace Engineering, \faUniversity\ University of Illinois at Urbana-Champaign, USA. D. Chatterjee is with \faGroup\ Systems and Control Engineering, \faUniversity\ Indian Institute of Technology, Mumbai, 400076, India.\\
  \noindent (SG): \faEnvelope\ \texttt{ganguly.siddhartha.7p@kyoto-u.ac.jp}, \faHome\ \url{https://sites.google.com/view/siddhartha-ganguly}.\\
	(RAD): \faEnvelope\ \texttt{dslvaaron24@gmail.com}, \faHome\ \url{https://sites.google.com/view/rihandsilva}.\\
	(DC): \faEnvelope\ \texttt{dchatter@iitb.ac.in}, \faHome\ \url{http://www.sc.iitb.ac.in/~chatterjee}.
}

\author[S. Ganguly]{Siddhartha Ganguly\,\orcidlink{0000-0003-2046-2061}}
\author[R. A. D'Silva]{Rihan Aaron D'Silva\, \orcidlink{0000-0002-0273-5364}}
\author[D. Chatterjee]{Debasish Chatterjee\,\orcidlink{0000-0002-1718-653X}}
\begin{document}

\maketitle

\begin{abstract}

This article introduces a new transcription, change point localization, and mesh refinement scheme for direct optimization-based solutions and for uniform approximation of optimal control trajectories associated with a class of nonlinear constrained optimal control problems (OCPs). The base transcription algorithm for which we establish the refinement algorithm is a \emph{direct multiple shooting technique} --- \(\quito\) \(\vertwo\) (Quasi-Interpolation based Trajectory Optimization). The mesh refinement technique consists of two steps ---  localization of certain irregular regions in an optimal control trajectory via wavelets, followed by a targeted \(h\)-refinement approach around such regions of irregularity. Theoretical approximation guarantees on uniform grids are presented for optimal controls with certain regularity properties along with guarantees of localization of change points by wavelet transform. Numerical illustrations are provided for control profiles involving discontinuities to show the effectiveness of the localization and refinement strategy. We also announce, and make freely available, a new software package based on \(\quito\) \(\vertwo\) along with all its functionalities for completeness. The package is available at: \url{https://github.com/chatterjee-d/QuITOv2.git}.
\end{abstract}

\section{Introduction}\label{sec:intro}

One of the key challenges in the field of optimal control and dynamic optimization is to design efficient and tractable numerical algorithms for producing optimal state-action trajectories under \emph{path} constraints --- constraints on the state-action trajectory at each instant of time. There are, broadly, two approaches to numerical optimal control: One is the \emph{indirect method} that proceeds via the Pontryagin maximum principle (PMP), leveraging necessary conditions for optimality. Tools based on the indirect method are well-known to be accurate, but it is difficult to include path constraints on the states \cite[\S 2.2.3]{ref:trelat2023control} in this framework because the resulting boundary value problem involves infinite-dimensional objects that are difficult to encode into algorithms. The second approach is the \emph{direct method}, which follows the \emph{discretize-then-optimize} paradigm: it transcribes the original (infinite-dimensional) OCP as a nonlinear finite-dimensional optimization problem post discretization and employs an optimization routine to solve it numerically.

\par Direct methods have historically been more successful, broadly speaking, at furnishing optimal state-action trajectories of constrained OCPs with relative ease because they bypass most of the difficulties involved with the indirect method, although they tend to compromise accuracy in certain cases; see the enlightening survey articles \cite{ref:JZ:ET:MC-2017, ref:JBC:RF:EL:HZ:OCP_algo:23}. The chief reason for the imprecision here is the \emph{approximation scheme}, which passes on the exploration overload for a search of admissible trajectories in appropriate infinite-dimensional spaces into some finite-dimensional vector space by means of discretization. The indicated transfer leads to a nonlinear program, for which one can employ off-the-shelf numerical solvers having the desired accuracy guarantees.

Amid the whole gamut of direct optimization methods, the \emph{direct shooting algorithms} \cite{ref:betts-book,ref:SG:NR:DC:RB-22} constitute perhaps the simplest and the most well-known direct trajectory optimization technique. Direct shooting algorithms leverage control trajectory parameterization \cite{ref:DasGanAnjCha23CDC} over the span of finitely many suitably chosen generating functions to discretize the original OCP along with the constraints, and solve the corresponding constrained nonlinear program (NLP). The \textit{pseudospectral} or the \textit{orthogonal collocation} methods \cite{ref:ross-PSMreview-2012} are another popular class of direct method that discretizes both states and control by orthogonal polynomials and Lagrange interpolation via placing collocation points in a non-uniform manner over the time horizon \cite{ref:Rao-10}. Conventional collocation schemes are primarily of two types --- on the one hand, we have the \(h\)-method, where the time horizon is divided into a (possibly nonuniform) mesh, and the state trajectory is approximated via fixed degree polynomials in each partition of the time horizon; on the other hand, in \(p\)-methods the degree of the interpolating polynomial is varied keeping the partition of the time horizon fixed. Another class of methods that uses properties of both \(h-\) and \(p-\)methods are the \(hp\)-collocation techniques \cite{ref:hpadaptive_coll}. More recently, an integrated residue minimization (IRM) \cite{ref:YN:ECK:residue} and penalty barrier function \cite{ref:neuenhofen2021dynamic} based direct transcription method was established, and it was shown that they perform better than conventional direct collocation schemes for certain classes of dynamic optimization problems. In particular, for \emph{singular} optimal control problems these techniques were able to eliminate the \emph{ringing} phenomena exhibited by control trajectories in pseudospectral collocation methods. Subsequently, a mesh refinement scheme with early termination was established for dynamical feasibility problems \cite{ref:vila2023mesh}. In \cite{ref:neuenhofen2021dynamic}, instead of showing the convergence of the numerical minimizer to the exact minimizer, it was shown that the \emph{gap} between the cost function corresponding to the numerical and the exact minimizer goes to zero as the discretization step goes to zero. In \cite{ref:MA:WWH:switch} a transcription algorithm was established for classes of OCPs where the optimal control is singular or bang bang in nature and the exact location of switching points is known a priori.

In the context of direct collocation methods, several mesh refinement algorithms have been proposed. In \cite{ref:ph_meshref} a \emph{ph}-collocation method has been established, where first a relative error tolerance is determined based on the difference between a Legendre-Gauss-Radau quadrature of the dynamics and a Lagrange polynomial approximation of the state trajectories and next, based on the tolerance a mesh refinement scheme has been set up. In \cite{ref:adaptive-meshref-Liu} an adaptive refinement scheme has been established based on decay rates of Legendre polynomial coefficients. A nonsmoothness localization and adaptive refinement procedure was established in \cite{ref:meshref_nonsmooth_detection}. A Legendre-Gauss-Radau collocation and edge detection algorithm was established in \cite{ref:meshref_jump_func}, where a jump function approximation method was employed to detect the discontinuities, and then locally, on the detected region, the mesh was refined. The preceding refinement schemes are dependent on the underlying transcription mechanism and they are generally designed to reduce approximation error/tolerance of the parameterized state trajectories with respect to the quadrature rule (which determines the type of collocation) and the error is in the \(\lpL[2]\)-sense. 

The refinement algorithm that we provide is different from existing algorithms, and uses time-frequency analysis to identify regions that require a refined mesh to accurately capture the optimal control trajectories; our numerical results in \S\eqref{sec:num_exp} show remarkably better performance in terms of accuracy and sometimes also computation-time compared to conventional direct collocation methods for singular optimal control problems. 
\subsection*{Our contributions:}
Our primary motivation comes from uniform approximation of constrained optimal control trajectories; we believe that this particular type of approximation is extremely important in engineering but has received comparatively much sparse attention relative to, e.g., \(\lpL[2]\)-approximation. Accordingly our contributions are geared towards ensuring small uniform approximation errors.
\begin{enumerate}[label=\textup{\Alph*.}, leftmargin=*]
\item \label{contrib:quito_1} (Direct multiple shooting with tight control on the uniform error.) We establish a direct multiple shooting (DMS) algorithm for a class of continuous-time optimal control problems with path constraints by parameterizing the control trajectory via a quasi-interpolation scheme on a piecewise uniform grid. A key feature of our algorithm is tight uniform error guarantees on the difference between the optimal control trajectory \(\cont\as(\cdot)\) and an approximate control trajectory \(\ulam(\cdot)\) constructed via a suitable quasi-interpolation engine. We demonstrate that given a \emph{prespecified error tolerance} \(\eps>0\), under mild regularity and structural assumptions on the problem data, certain parameters related to the quasi-interpolation scheme can be chosen so that \(\ulam(\cdot)\) constructed over a uniform grid via the quasi-interpolation scheme remains within an \(\eps\)-ball in the \emph{uniform metric} around \(\cont\as(\cdot)\).

We reiterate the fact that techniques pertaining to approximation in the uniform sense are difficult and rare, although they are perhaps the most important for several engineering/practical considerations; consequently, such an error guarantee is one of the primary contributions of this article. While the DMS algorithm is advanced for a piecewise uniform grid \(\mathcal{G}\), for simplicity we provide the error estimates for a uniform grid only.

\item \label{contrib:quito_2} (Selective mesh refinement driven by sharp localization of change points.) For an accurate representation of control trajectories that may feature \emph{change points} involving kinks, narrow peaks, valleys over certain epochs, and even discontinuities, we establish a change point localization and mesh refinement scheme:
\begin{enumerate}[label=\textup{(\ref{contrib:quito_2}\roman*)}, leftmargin=*, widest=ii, align=left]
\item \label{subcontrib:quito_2.1} The localization mechanism employs tools from multi-resolution wavelet analysis that are widely employed in the signal processing community. However, unlike in typical signal processing applications, the point of departure in our setting is that the actual ``signal" --- the optimal control trajectory --- is inaccessible for computational purposes. To this end, we feed the approximate control trajectory generated by the DMS algorithm on a sparse uniform grid to the wavelet-based localization engine to detect possible change points in the control signal. We show that at fine scales of the mother wavelet, all change points will be localized under mild hypothesis on the problem data.
\item \label{subcontrib:quito_2.3} Subsequently, we refine the mesh by adding knot points in and around the localized regions iteratively until a preset threshold is crossed. The mesh refinement scheme saves computational time compared to the case of solving the optimal control problem on a very fine uniform mesh. 
\end{enumerate}

\item \label{contrib:quito_3} (Numerics, software \(\quito\) \(\vertwo\), and a GUI.) 
\begin{enumerate}[label=\textup{(\ref{contrib:quito_3}\roman*)}, leftmargin=*, widest=ii, align=left]
\item \label{subcontrib:quito_3.1} We provide a library of numerical examples in \S\eqref{sec:num_exp} in which control trajectories exhibits bang-bang, singular, or both of the preceding features, and with these examples we demonstrate that \(\quito\) \(\vertwo\) with its localization and refinement modules performs better --- in terms of solution accuracy and sometimes also computation-time --- compared to state-of-the-art pseudospectral collocation methods and integrated residue minimization methods.
\item \label{subcontrib:quito_3.2}  For ease of reproducibility, the change point localization and the mesh refinement modules were encoded in a software package, which we call \(\quito\) \(\vertwo\), that comes with a Graphical User Interface (GUI); brief functionalities of \(\quito\) \(\vertwo\) are given in \S\ref{sec:software}.\footnote{See \cite{ref:QuITO:SoftX} for more details on a previous version of the software \(\quito\) (without any localization or mesh refinement modules), where we employed quasi-interpolation over a uniform grid to approximate the control trajectory.}
\end{enumerate}
\end{enumerate}

\subsection*{Organization}
This article unfolds as follows. The primary problem that we address in this article is formulated in \S \ref{sec:prob_statement}. Mathematical background on the type of quasi-interpolation engine we employ herein is introduced in \S\ref{sec:appen_A}. The main contribution of this article has been presented in \S\ref{sec:main_results} which consists of three major parts -- \S\ref{subsec:shooting}: where we established the direct transcription mechanism, \S\ref{subsec:detection}: where we present the wavelet theory-based localization algorithm, and \S\ref{subsubsec:meshrf_scheme} contains the mesh refinement algorithm. A library of numerical examples are presented in \S\ref{sec:num_exp} to illustrate the effectiveness of \(\quito\) \(\vertwo\). The algorithms developed herein were encoded in a software package that comes with a GUI for ease of usage. A brief description and functionalities of the software are given in \S\ref{sec:software}.
\subsection*{Notation}
We employ standard notation here. We let \(\N \Let \aset{1,2,\ldots}\) denote the set of positive integers and let \(\Z\) denote the integers. For \(d \in \N\), the vector space \(\Rbb^d\) is assumed to be equipped with standard inner product \(\inprod{v}{v'}\Let v^{\top}v'\) for every \(v,v' \in \Rbb^d\). We denote the uniform function norm via the notation \(\unifnorm{\cdot}\); more precisely, for a real-valued bounded function \(f(\cdot)\) defined on a nonempty set \(S\), it is given by \(\unifnorm{f(\cdot)} \Let \sup_{x \in S}|f(x)|\). For vectors residing in \(\Rbb^d\) we employ the notation \(\norm{\cdot}_{\infty}\) for the uniform (box) norm of vectors. The one-dimensional closed ball (interval) centered at \(x\) and of radius \(\eta>0\) is denoted by \(\Ball[x, \eta]\). Let \(\upsilon \in \N\); for \(X\) and \(Y\) nonempty open subsets of Euclidean spaces, the set of \(\upsilon\)-times continuously differentiable functions from \(X\) to \(Y\) is denoted by \(\mathcal{C}^{\upsilon}(X;Y)\). The Schwartz space of rapidly decaying \(\Rbb\)-valued functions \cite{ref:adams2003sobolev} on \(\Rbb\) is denoted by \(\mathcal{S}(\Rbb)\).

\section{Problem Statement}\label{sec:prob_statement}
Let \(d,\dimcon, \ell \in \N\) and fix a final time horizon \(\tfin>0\). Let us consider a control-affine nonlinear controlled system defined on the time interval \(\lcrc{0}{T}\):
\begin{equation}
	\label{eq:sys}
	\dot \st(t) = \sys\bigl(\st(t)\bigr)+ G \bigl(\st(t)\bigr)\cont(t)\,\,\text{for a.e.\ } t \in \lcrc{0}{T}, 
\end{equation}
where \(\st(t)\in \Rbb^d\), and \(\cont(t) \in \Rbb^{\dimcon}\) are the state and the control vectors at time \(t\); the maps \(f: \Rbb^d \lra \Rbb^d\) and \(G: \Rbb^d \lra \Rbb^{d \times \dimcon}\) represents the drift vector field and the state-dependent control matrix field, respectively. Assume that the initial state \(x(0)\Let \param \in \Rbb^d\). Let us consider the constrained nonlinear optimal control problem:
\begin{equation}
	\label{eq:OCP}
\begin{aligned}
& \inf_{\cont(\cdot) }	&& V_{\horizon}\bigl(\param,\cont(\cdot)\bigr) \Let \fcost\bigl(\st(\tfin)\bigr) + \int_{\tinit}^{\horizon} \rcost\bigl(\st(t), \cont(t)\bigr) \odif{t} \\
&  \sbjto		&&  \begin{cases}
\dot \st(t) = \sys\bigl(\st(t)\bigr)+ G \bigl(\st(t)\bigr)\cont(t)\,\,\text{for a.e.\ } t \in \lcrc{0}{T},\\
\st(\tinit)= \param,\,\, r_{F}(\st(\horizon)) \le 0,\\
h_j(\st(t),\cont(t)) \le 0\,\,\text{for a.e }t \in \lcrc{0}{\horizon},\, j \in \aset[]{1,\ldots,\ell}, \\
u(t) \in \admcont\,\,\text{for a.e }t \in \lcrc{0}{\horizon},
\end{cases}
\end{aligned}
\end{equation}
over the set of admissible control trajectories which are bounded measurable functions such that \(\cont(t) \in \admcont \subset \Rbb^{\dimcon}\) for a.e. \(t\in \lcrc{0}{T}\). We have the following problem data:
\begin{enumerate}[label=(\ref{eq:OCP}-\alph*), leftmargin=*, widest=b, align=left]

\item \label{OCPdata1} the set \(\admcont\) is a nonempty compact subset of \(\Rbb^{\dimcon}\) with the representation:
\begin{align}
    \admcont \Let \prod_{i=1}^{\dimcon}\lcrc{U^i_{\text{min}}}{U^i_{\text{max}}}\;\;\text{with }-\infty< U^{i}_{\text{min}} < U^i_{\text{max}} <+\infty\;\;\text{for each }i;\nn
\end{align}

\item \label{OCPdata2} the drift vector field \(\dummyx \mapsto f(\dummyx)\) and the control matrix field \(\dummyx \mapsto G(\dummyx)\) are locally Lipschitz continuous and there exists positive numbers \(\mathsf{C}_f>0\) and \(\mathsf{C}_{G}>0\) such that \(\norm{f(\dummyx_0)} \le \mathsf{C}_f (1+\norm{\dummyx_0})\) for all \(\dummyx_0 \in \Rbb^d\) and \(\norm{G(\dummyx_1)} \le \mathsf{C}_{G} (1+\norm{\dummyx_1})\) for all \(\dummyx_1 \in \Rbb^d\);

\item \label{OCPdata3} the \emph{instantaneous cost} \((\dummyx,\dummyu) \mapsto \rcost(\dummyx,\dummyu) \in \lcro{0}{+\infty}\) is continuous and convex, and \(\dummyx \mapsto \rcost(\dummyx,\dummyu)\) is locally Lipschitz continuous. Moreover, for some \(\overline{r}>1\), there exists \(\overline{\alpha}, \overline{\beta}>0\) such that for all admissible state-control pairs \(\rcost(\dummyx,\dummyu) \ge \overline{\alpha} \norm{\dummyu}^{\overline{r}}+\overline{\beta}\).
The \emph{terminal cost} \(\dummyx \mapsto \fcost(\dummyx)\in \lcro{0}{+\infty}\) is locally Lipschitz continuous and there exists a positive constant \(\mathsf{C}_T>0\) such that \(\abs{\fcost(\dummyx)} \le \mathsf{C}_T (1+\norm{\dummyx})\) for all \(\dummyx \in \Rbb^d\);

\item\label{OCPdata4} For each \(j=1,\ldots,\ell\) the constraint function \(\Rbb^d \times \Rbb^{\dimcon} \ni (\dummyx,\dummyu) \mapsto h_j(\dummyx,\dummyu) \in \Rbb\) and the terminal constraint function \(\Rbb^d \ni \dummyx \mapsto r_F(\dummyx) \in \Rbb\) are locally Lipschitz continuous and there exists a constants \(\mathsf{C}_{\dummyx}>0, \mathsf{C}_{\dummyu}>0,\) and \(\mathsf{C}_{F}>0\) such that \(\abs{h_j(\dummyx,\dummyu)} \le \mathsf{C}_{\dummyx} (1 + \norm{\dummyx})\), \(\abs{h_j(\dummyx,\dummyu)} \le \mathsf{C}_{\dummyu} (1 + \norm{\dummyu})\) for all \(j \in \aset[]{1,\ldots,\ell}\) and for all \((\dummyx,\dummyu) \in \Rbb^d \times \Rbb^{\nu}\), and \(\abs{r_F(\dummyx)} \le \mathsf{C}_{F} (1+\norm{\dummyx})\) for all \(\dummyx \in \Rbb^d\). 
\end{enumerate}   
\vspace{2mm}
\begin{remark}
Assumption \ref{OCPdata1} and \eqref{OCPdata2} are standard in control theory. \ref{OCPdata1} concerns the boundedness and shape of the admissible action set; it is natural since most (if not all) realistic actuators cannot deliver signals of arbitrarily large amplitude, and each (\(\Rbb\)-valued) control action is independent of the others. \ref{OCPdata2} is basic and ensures the existence of solutions to the ODE \eqref{eq:sys}. \ref{OCPdata3}--\ref{OCPdata4} concerns certain regularity and growth properties of the cost functions and the constraint functions. Under \ref{OCPdata1}--\ref{OCPdata4}, if there exists at least one admissible process \(\bigl(x\as(\cdot),u\as(\cdot)\bigr)\) for which \(V_T(\xz,u\as(\cdot)) < +\infty\), then the problem \eqref{eq:OCP} admits a solution \cite[Chapter 23]{ref:clarke_ocpbook} and we call a solution \(\cont\as:\lcrc{0}{\horizon} \lra \admcont\) an \emph{optimal control trajectory}. Moreover, in the following sections, our theoretical results will require \( u\as(\cdot) \) to be Lipschitz continuous. Consequently, these results will be presented under additional standard regularity assumptions on the problem data; see Proposition \ref{prop:opt_traj_exst}, Theorem \ref{thrm:optimal estimate}, and Theorem \ref{thrm:lin optimal estimate} ahead for more details.
\end{remark}
Next, we present a brief overview of the primary approximation tool used in this work.

\section{Mathematical background on approximate approximations}\label{sec:appen_A}
We provide a summary and a few relevant results about a particular class of quasi-interpolation technique known as \emph{approximate approximation} \cite{ref:mazyabook} on a piecewise uniform grid. We start with the \emph{uniform grid} version. Let \(h>0\) and \(\Dd>0\). Consider the \emph{uniform grid} \(\aset[]{mh \suchthat m \in \Z} \subset \Rbb\). Let \(\genfn: \Rbb \lra \Rbb\) be a \emph{generating function} that belongs to the Schwartz class \(\mathcal{S}(\Rbb)\) of rapidly decaying functions.

For a continuous function \(\Rbb \ni t \mapsto \interp(t) \in \Rbb\), the quasi-interpolation scheme known as the approximate approximation is given by the summation  
\begin{align}\label{eq:gen_quasi_sdim}
	\Rbb \ni t \mapsto (\MM \interp)(t) \Let \frac{1}{\sqrt{\Dd}}\sum_{m \in \Z}\interp(mh)\, \genfn \bigl(\tfrac{t-mh}{h \sqrt{\Dd}}\bigr).
\end{align}
The quantity \(h\) is the discretization parameter or the step size, \(\Dd\) is the shape \emph{shape parameter} of the \emph{generating function} \(\genfn (\cdot)\). The function \(\genfn(\cdot)\) satisfies the additional properties:
\begin{itemize}[leftmargin=*, widest=b, align=left,label=\(\circ\)]
\item \emph{The continuous moment condition of order \(M \in \N\)},  i.e., \begin{align}\label{eq:moment}
        &\int_\Rbb \genfn(t)\,\dd t =1 \text{ and } \int_{\Rbb} t^\alpha \genfn(t)\,\dd t = 0\nn \\& \text{ for all } \alpha \in \N\text{ satisfying}, 1\leq \alpha \le M-1. \end{align}
\item \emph{The decay condition of order \(K\)}: \(\genfn (\cdot)\) satisfies the decay condition of exponent \(K\) if there exist \(c_0 >0\) and \( K>1\) such that:
        \begin{align}\label{decay}( 1+ \abs{t})^K \abs{\tfrac{\partial}{\partial t}\genfn(t)} \leq c_0 \text{ for all }t \in \Rbb.
        \end{align}
\end{itemize}
We refer the interested readers to \cite{ref:mazyabook} for a book-length treatment of the topic. Now we state the key estimate:
\begin{theorem}{\cite[Theorem 2.25]{ref:mazyabook}}
\label{thrm:Holder-Lipschitz-estimate}
Consider a Lipschitz continuous function \(\interp:\Rbb \to \Rbb\) of Lipschitz rank \(L_q\). Let \((h,\Dd) \in \loro{0}{+\infty}^2\) and let \(\mathsf{G} \Let \aset[]{mh \suchthat m \in \Z} \subset \Rbb\) be a uniform grid of mesh size \(h\). Suppose that \(\genfn(\cdot)\) satisfies \eqref{eq:moment} and \eqref{decay}. Then 
\begin{align}\label{eq:unif_estimate_lip}
\unifnorm{(\MM \interp)(\cdot)- \interp(\cdot)} \le c_{\gamma}L_q h \sqrt{\Dd}+\Delta_0(\genfn,\Dd),
\end{align}
	where \(\Delta_0(\psi,\Dd) \Let \mathcal{E}_0(\psi,\Dd)\unifnorm{\interp(\cdot)}\) is a \emph{saturation error}, \(c_{\gamma} \Let M \cdot \Gamma (M)/\Gamma(M+2)\) is a constant (where \(\Gamma\) denotes the standard Gamma function), and the term
\(\mathcal{E}_0(\psi,\Dd)\) is given by 
\begin{align}\label{eq:saturation_decay_term}
    \mathcal{E}_0(\genfn,\Dd)(\cdot) \Let \sup_{x \in \Rbb}\sum_{\nu \in \Z\setminus \{0\}}\mathcal{F} \genfn(\sqrt{\Dd}\nu)\epower{2\pi i {x}{\nu}},
\end{align}
where \(\mathcal{F}\) is the Fourier transform operator on \(\Rbb\).
\end{theorem}
\begin{remark}\label{rem:saturation_error}
On the right-hand side of the estimate \eqref{eq:unif_estimate_lip}, the first term converges to zero as \(h\) goes to zero. The second term \(\Delta_0(\genfn,\Dd)\) is called the \textbf{saturation error}, which can be reduced to an arbitrary small number by controlling the shape parameter \(\Dd\) associated with a particular choice of generating function \(\genfn(\cdot)\). For any \(\interp(\cdot)\) in an appropriate class of functions, \((\MM \interp)(\cdot)\) approximates \(\interp(\cdot)\), in the uniform metric, closely up to this small saturation error. This phenomenon is known as \textbf{pseudo convergence}. A lower bound on \(\Dd\) can be obtained (in some cases precise formulas can be given, see \cite{ref:maz_mfld}); see \cite[Chapter 2]{ref:mazyabook} and \cite[Chapter 3, Table 1]{ref:mazyabook} for precise error estimates and relevant data, to ensure that the saturation error always remains below a preassigned tolerance value. See \cite[Chapter 2]{ref:mazyabook}, \cite{ref:GanCha25} for a multi dimensional version and the related estimates. 
\end{remark}
\begin{remark}\label{rem:truncated_sum}
The approximation formula in \eqref{eq:gen_quasi_sdim} involves infinite sums, which can be truncated to just a few terms while maintaining tight control of the ensuing error; this is possible due to the sharp decay property enjoyed by the generating function \(\genfn(\cdot)\) (recall that \(\genfn(\cdot) \in \mathcal{S}(\Rbb)\)). Indeed, let \(\mathcal{T}>0\) and consider the interval \(\Ball[t,\mathcal{T}]\). Corresponding to the uniform approximation scheme in \eqref{eq:gen_quasi_sdim}, define the truncated sum
\begin{align}\label{eq:truncated_quasi_sdim_new}
 \Rbb \ni   t \mapsto \interp^{\dagger}(t) \Let \frac{1}{\Dd^{1/2}} \sum_{\mathclap{\substack{mh \in\, \Ball(t,\mathcal{T})\\ m \in \Z}}}\interp(mh)\,\genfn \bigl(\tfrac{t-mh}{h \sqrt{\Dd}}\bigr).
\end{align}
The approximation formula \eqref{eq:truncated_quasi_sdim_new} considers only the points \(mh\) inside the interval \(\Ball[t,\mathcal{T}]\), and thus the sum in \eqref{eq:truncated_quasi_sdim_new} 
is finite. Then the bound
\begin{align}\label{eq:truncated_bound_1}
    \abs{\interp^{\dagger}(t)-(\MM \interp)(t)} \le \mathcal{B} \bigl( \tfrac{\sqrt{\Dd}h}{\mathcal{T}} \bigr)^{K-1} \sup_{t \in \Rbb}\abs{\interp(t)}
\end{align}
holds for all \(t\in \Rbb\), where 
\begin{enumerate}[label=\textup{(R-\alph*)}, leftmargin=*, widest=b, align=left]
\item \label{trunc:rem:2} \(K>1\) is the decay exponent (see \eqref{decay}) of \(\genfn(\cdot)\), and
\item \label{trunc:rem:1} \(\mathcal{B}\) is a constant and a conservative upper bound of \(\mathcal{B}\) is \(\frac{c_0 \hat{C}}{K-1}\), where \(c_0\) is the right-hand side of \eqref{decay} and \(\hat{C}>0\) is a constant depending only on certain properties of \(\genfn(\cdot)\). The readers are referred to  \cite[\S 4.3.2, Page 50]{weisse2009global}, and \cite[\S2.3.2, Page 35]{ref:mazyabook} for concrete expressions and detailed derivations.
\end{enumerate}
An interesting case, one that we will employ in our subsequent analysis, is when \(\mathcal{T} \Let \rzero h\) for some \(\rzero >0\); this makes the bound in \eqref{eq:truncated_bound_1} independent of the parameter \(h\). Thus, \(t \mapsto \interp^{\dagger}(t)\) takes into account only the terms for which \(\abs{t/h - m} \le \rzero\), which makes the number of summands in \eqref{eq:truncated_quasi_sdim_new} independent of the step-size \(h.\) Similar truncation over intervals can be defined for the nonuniform grids as well; see \cite[Chapter 11]{ref:mazyabook}.
\end{remark}
At the heart of our direct multiple shooting algorithm \(\quito\) \(\vertwo\), is a piecewise uniform version of the approximate approximation scheme. The motivation behind employing such a scheme is the need to refine the time mesh locally, whenever needed (which will be dictated by the localization mechanism presented in \S\ref{subsec:detection}), for a better representation of the solution. Such needs may arise due to local irregularities in the behaviour of the function being approximated. To this end, we provide a brief technical background on the piecewise uniform version of the approximation engine, and we refer interested readers to \cite[Chapter 11]{ref:mazyabook} for more details. The following theorem, extracted from \cite[Chapter 11]{ref:mazyabook}, describes the formation of an approximate partition of unity using Gaussian generating functions centered on a piecewise uniform grid.
\begin{theorem}{\cite[\S 11.4.1]{ref:mazyabook}}\label{thrm:app-app-piecewise_grid_formula}
Let \(h>0\) and \(b \in \Rbb\) and let \(\Rbb \ni t \mapsto \genfn(t) \Let \tfrac{1}{\sqrt{\pi}}\epower{-\abs{t}^2}\) be the Gaussian generating function. Consider the given uniform grid \(\aset[]{hm+b \suchthat m \in \Z}\). Then there exist finite sets \(Z_1 \subset \Z\), \(Z(h_m) \subset \Z\) depending on the parameter \(h_m<\tfrac{1}{2}\), and consequently the sets \(\widehat{\mathcal{G}} \Let \aset[\big]{hm+b \suchthat m \in Z_1\setminus Z_2}\) where \(Z_2 \subset Z_1\) and \(\widecheck{\mathcal{G}}_m\Let \aset[\big]{hm + b + h h_m k \suchthat k \in Z(h_m)}\) which define the piecewise uniform grid \begin{align}\label{eq:pw_unif_grid}
 \mathcal{G} =  \{g_j\}_{j \in J} \Let \widehat{\mathcal{G}} \cup \bigcup_{m \in Z_2} \widecheck{\mathcal{G}}_m,
\end{align}
where \(J\) is the finite indexing set corresponding to the RHS\ of \eqref{eq:pw_unif_grid}. Then for given \(\eps>0\), there exist \(\Dd \in \loro{0}{+\infty}\) such that choosing the pair of parameters \((h_j,a_j)_{j \in J}\) as
\begin{equation}\label{eq:a_j_h_j}
\begin{aligned}
\begin{cases}
  h_j \Let h\sqrt{\Dd},\,a_j \Let \tfrac{1}{(\pi \Dd)^{1/2}}&\hspace{-2mm}\text{if }g_j \in \widehat{\mathcal{G}},\\
h_j \Let h h_m\sqrt{\Dd},\,a_j \Let \tfrac{\epower{-h_m^2|k|^2/(\Dd(1-h_m^2))}}{\pi\Dd\sqrt{1-h_m^2}}&\hspace{-2mm}\text{if } g_j \in \cup_{m \in Z_2} \widecheck{\mathcal{G}}_m,
\end{cases}
\end{aligned}
\end{equation} 
guarantees that
\begin{align}\label{eq:part_of_unity}
    \biggl| 1- \sum_{j \in J} a_j \;\genfn \bigl(\tfrac{t-g_j}{h_j}\bigr) \biggr| \le \eps \quad\text{for all }t \in \Rbb.
\end{align}
\end{theorem}
\begin{remark}
For proof of Theorem \ref{thrm:Holder-Lipschitz-estimate} we refer the reader to \cite[Theorem 2.25 and Remark 2.26]{ref:mazyabook}. Note that in Theorem \ref{thrm:Holder-Lipschitz-estimate}, \(\genfn(\cdot) \in \mathcal{S}(\Rbb)\) need not be Gaussian. Theorem \ref{thrm:app-app-piecewise_grid_formula} is a formal version of the contents of \cite[\S 11.4.1]{ref:mazyabook} where a quasi-interpolation engine was developed on piecewise uniform grids for Gaussian generating functions; similar results for general generating functions are given in \cite[\S 11.4.2]{ref:mazyabook}.
\end{remark}

 \section{Main results}\label{sec:main_results}
\label{sec:main_result}
This section contains our new direct transcription, localization, and mesh refinement algorithm \(\quito\) \(\vertwo\) along with the theoretical results concerning uniform control approximation guarantees and decay of certain wavelet coefficients of approximate control trajectories that are used for change point localization.  We begin with the \(\quito\) \(\vertwo\) algorithm in \S\ref{subsec:shooting}, and establish a change point localization mechanism using tools from signal processing in \S\ref{subsec:detection}. Finally, we present the complete transcription technique equipped with a mesh refinement module in \S\ref{subsubsec:meshrf_scheme}. Figure \ref{fig:quito_schematic} depicts the complete transcription, change point localization, and refinement process encoded in \(\quito\) \(\vertwo\).
\begin{figure}
    \centering
    \includegraphics[scale=0.65]{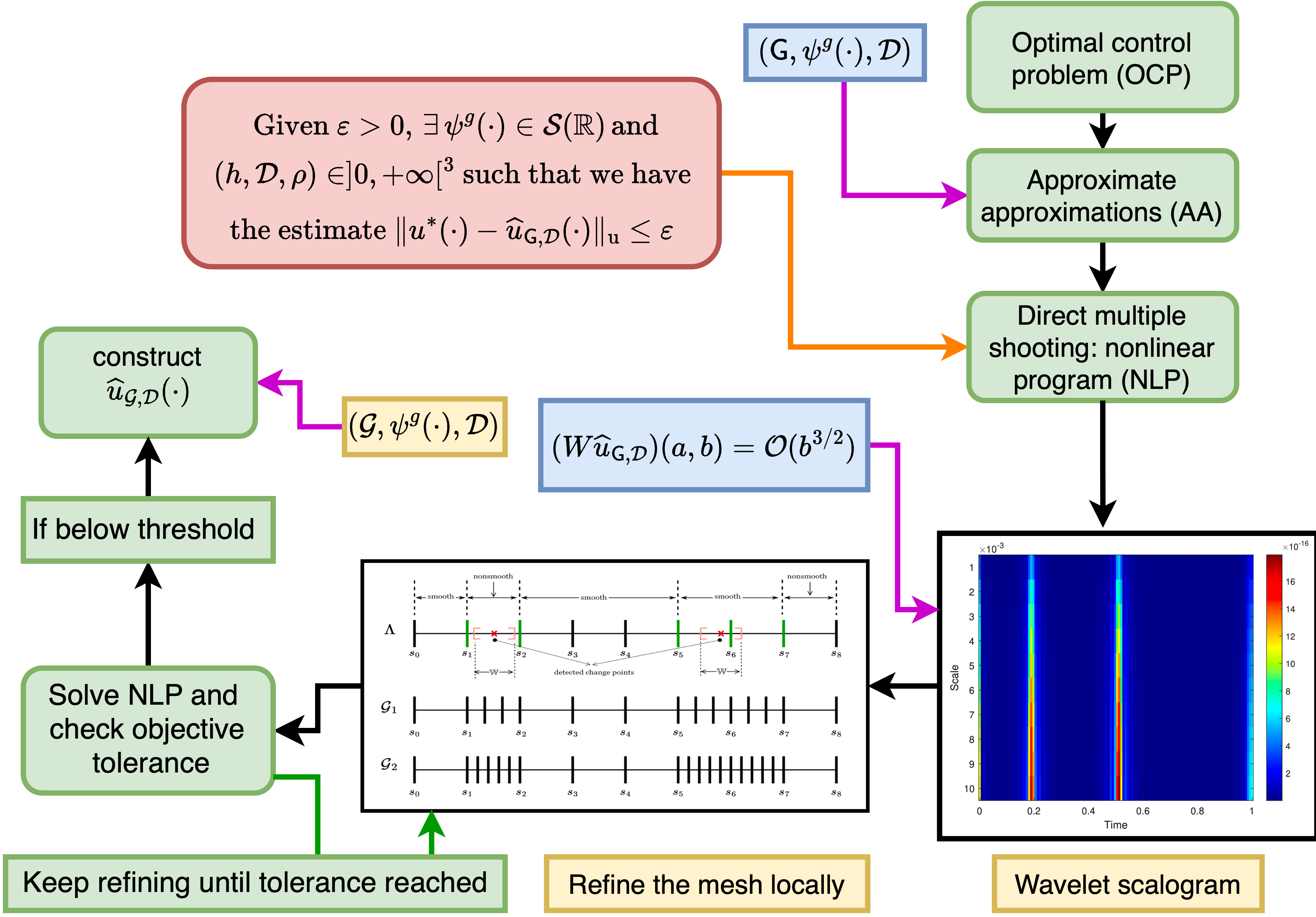}
    \caption{A schematic diagram of \(\quito\) \(\vertwo\). The grids \(\mathsf{G}\) and \(\mathcal{G}\) denotes a uniform and a piecewise uniform grid respectively.}
    \label{fig:quito_schematic}
\end{figure}
\subsection{\(\quito\) \(\vertwo\): The direct multiple shooting method.}\label{subsec:shooting}
We now provide one of our main results --- a direct multiple shooting technique (which we abbreviate as DMS in the sequel) method to solve the OCP \eqref{eq:OCP} numerically. We take the same route as \cite[\S III]{ref:SG:NR:DC:RB-22} and adapt the shooting procedure reported therein to our current setting. The primary point of departure from the existing literature is the employment of a tailored uniform approximation scheme on a piecewise uniform grid instead of a uniform grid; our motivation behind this development will be clear in \S\ref{subsubsec:meshrf_scheme}. The first step of our algorithm is a parameterization of the control trajectory \(t \mapsto \cont(t)\) using a quasi-interpolation engine designed to uniformly approximate functions with certain regularity properties on a \emph{piecewise uniform} grid. The mathematical details regarding the approximation engine and relevant error estimates are given in \S\ref{sec:appen_A}.

Consider a set of nodes \(\mathcal{G} \Let \{g_i\}_{i\in J} \subset \Rbb\), where \(J\) is a ordered finite index set, belonging to a piecewise uniform grid as specified in \eqref{eq:pw_unif_grid}. We choose generating functions from the family
\begin{align}\label{basis_set}
    \mathbb{B} \Let  \left\{ \genfn\left(\frac{\cdot-g_i}{h_i}\right) \in \mathcal{S}(\Rbb;\Rbb), \,i\in J \;\middle\vert\;
     \begin{array}{@{}l@{}}
     \genfn(t) \Let \tfrac{1}{\sqrt{\pi}}\epower{-|t|^2},\,h_i \text{ as specified in } \eqref{eq:a_j_h_j},\\\text{ and }g_i \in \mathcal{G}.
     \end{array}
        \right\},  
\end{align}
The control trajectory \(\lcrc{0}{T} \ni t \mapsto u(t) \in \mathbb{U}\) is to be approximated using the quasi-interpolation engine defined on the piecewise uniform grid \(\mathcal{G}\) and given by the candidate approximants from the set: 
\begin{align}\label{eq:Mhu_piecewise_grid_set}
    \admcon \Let \aset[\bigg]{\tilde{\cont}(\cdot) \suchthat   \lcrc{0}{T} \ni t \mapsto \tilde{\cont}(t) \Let  \sum_{g_i \in \mathcal{G}}\tilde{\cont}(g_i) a_i \epower{-|t-g_i|^2/h_i^2} \in \Rbb^{\dimcon},\;\tilde{u}(g_i) \in \Rbb^{\dimcon}}. 
\end{align}
Thus, any candidate control trajectory from \(\admcon\) can be expressed (on the piecewise uniform grid \(\mathcal{G}\) with the parameters \((a_i,h_i)\) and the generating function \(\genfn(\cdot)\)) via the approximation formula:
\begin{equation}\label{eq:Mhu_piecewise_grid}
  \lcrc{0}{T} \mapsto t \mapsto \overline{\cont}(t) \Let  \sum_{g_i \in \mathcal{G}}\overline{\cont}(g_i) a_i \epower{-|t-g_i|^2/h_i^2} \in \Rbb^{\dimcon},
\end{equation}
where \(\bigl(\overline{\cont}(g_i)\bigr)_{i\in J} \subset \Rbb^{\dimcon}\) are vectors with real entries that are to be determined. Choosing \((a_i,h_i)\) appropriately guarantees the formation of an entry-wise approximate partition of unity with the Gaussian generating functions in the set \(\mathbb{B}\); Theorem \ref{thrm:app-app-piecewise_grid_formula} in \S \ref{sec:appen_A} contains the explicit selection rules for \((a_i,h_i)\).
\begin{remark}
The Gaussian family of generating functions in \eqref{basis_set} was picked for the sake of simplicity; a large class of generating functions other than the Gaussian is available, for which the readers are referred to \cite[Chapter 11, \S11.4]{ref:mazyabook}. We employ the Gaussian family exclusively in our numerical examples reported in \S \ref{sec:num_exp}.
\end{remark}
The coefficients \(\bigl(\overline{u}(g_i)\bigr)_{i \in \fset}\subset \Rbb^{\dimcon}\) in \eqref{eq:Mhu_piecewise_grid} are vectors to be determined the objects to be picked via an optimization routine. Let \(\J^{\mathrm{d}}\) be the discrete quadrature approximation of the cost in \eqref{eq:OCP} arrived as follows. We pick an initial guess \(\eta_i\) and solve the ODE \eqref{eq:sys} on time windows \([g_i,g_{i+1}]\) for each \(i\in \fset\), by means of a suitable numerical integration technique (which could range from Euler's step through more advanced symplectic integrators):  
\begin{align}
\dot{x}_i\bigl(t,\eta_i,\overline{u}(g_i)\bigr) & = f\bigl(x_i(t,\eta_i,\overline{u}(g_i))\bigr)\nonumber\\&+G \bigl( x_i(t,\eta_i,\overline{u}(g_i))\bigr)\overline{u}(g_i)\,\text{for all}\, t \in [g_i,g_{i+1}],\,i \in \fset, \nn
\end{align}
with \(\eta_0 = \param\) and \(x_i \bigl(g_i,\eta_i,\overline{u}(g_i)\bigr)=\eta_i \in \Rbb^d\) for each \(i\in J\). The trajectories are joined together at each \(g_i\) by means of \emph{continuity/defect} condition
\begin{align}\label{shoot:continuity}
    \eta_{i+1}=x_i \bigl( g_{i+1},\eta_i,\overline{u}(g_i) \bigr),
\end{align}
which is encoded into the off-the-shelf nonlinear program (NLP) solver as equality constraints. Let \(\rcost_i(\dummyx,\dummyu)\) be a suitable quadrature approximation of the running cost. Thus, the OCP in \eqref{eq:OCP} is transcribed into the following nonlinear program for digital implementation:
\begin{equation}
\label{eq:col_NLP} 
\begin{aligned}
& \minimize_{\bigl(\eta_i,\overline{\cont}(g_i)\bigr)_{i\in \fset}}	&& \J^{\mathrm{d}} \Let \sum_{i=0}^{N}\rcost_i\bigl(\eta_i,\overline{u}(g_i)\bigr)+\fcost(\eta_N) \\
&\sbjto	 &&\begin{cases}
  \st_0=\eta_0,\\
\eta_{i+1}=x_i \bigl( g_{i+1},\eta_i,\overline{u}(g_i) \bigr),\\
h_j \bigl(\eta_i,\overline{u}(g_i)\bigr) \le 0,\,\, i\in \fset,\,j=1,\ldots,\ell,\\
\overline{u}(g_i) \in \mathbb{U},\,r_F(\eta_N) \le 0,
\end{cases}
\end{aligned}
\end{equation}
where \(N\) is the final point of the ordered finite set \(\fset\).

We now transition towards establishing one of the main results of this article, and present a technical result on Lipschitz regularity of the optimal control trajectory \(t \mapsto \cont\as(t)\), which we lift from \cite{ref:depinho-2011lipschitz}, as preparatory material.
\begin{proposition}\label{prop:opt_traj_exst}
Consider the optimal control problem \eqref{eq:OCP} with its associated data \ref{OCPdata1}--\ref{OCPdata4}. Let us impose the additional set of assumptions:
\begin{enumerate}[label=\textup{(D-\roman*)}, leftmargin=*, widest=b, align=left]

\item\label{OCPdata_new_1} The functions \(f(\cdot), G(\cdot), \rcost(\cdot)\), and \(\fcost(\cdot)\) are continuously differentiable.

\item \label{OCPdata_new_2} For all \(\dummyx \in \Rbb^d\), the map \(\dummyu \mapsto \rcost(\dummyx,\dummyu)\) is twice continuously differentiable and \((\dummyx,\dummyu) \mapsto \rcost(\dummyx,\dummyu)\) uniformly strongly convex.\footnote{See \cite[Assumption (H3)]{ref:depinho-2011lipschitz} for a definition of uniformly strongly convex functions.} Furthermore for any compact sets \(D_{\dummyu} \subset \Rbb^{\dimcon}\), and \(E _{\dummyx}\subset \Rbb^{d}\), the function \(\dummyx \mapsto \nabla_{\dummyu}\rcost(\dummyx,\dummyu)\) is Lipschitz continuous on \(E_{\dummyx}\) uniformly with respect to \(\dummyu \in D_{\dummyu}\).

\item \label{OCPdata_new_3} For each \(j=1,\ldots,\ell\), the function \((\dummyx,\dummyu) \mapsto h_j(\dummyx,\dummyu)\) is continuously differentiable; \(\dummyu \mapsto h_j(\dummyx,\dummyu)\) is twice continuously differentiable and is convex for all \(\dummyx\). Additionally, for any compact sets \(D_h \subset \Rbb^{\dimcon}\), and \(E_{h} \subset \Rbb^{d}\), the mapping \(\dummyx \mapsto \nabla_{\dummyu}h_j(\dummyx,\dummyu)\) is Lipschitz continuous on the set \(E_h\) uniformly with respect to \(\dummyu \in D_{\dummyu}\).

\item\label{OCPdata_new_4} Define the collection of active set constraints for the mixed constraints by:
\[\mathcal{T}_m(x(t),u(t)) \Let \aset[]{j \suchthat h_j(x(t),u(t))=0}.\] Then, for a.e \(t \in [\tinit,\tfin]\), non-negative numbers \((\beta_j)_{j \in \mathcal{T}_m(x\as(t),u\as(t))}\) not all zero, we have the constraint qualification condition (H5) as given in \cite[\S 2]{ref:depinho-2011lipschitz}:
\begin{align}
    \sum_{j \in \mathcal{T}_m(x\as(t),u\as(t))}\beta_j \frac{\partial}{\partial u} h_j(x\as(t),u\as(t)) \neq 0. \nn
\end{align}
\end{enumerate}
Then, if \(t \mapsto \bigl(\st\as(t),\cont\as(t)\bigr)\) is a normal extremal to problem \eqref{eq:OCP}, the optimal control trajectory \([0,\horizon] \ni t \mapsto \cont\opt(t) \in \admcont\) is Lipschitz continuous.\footnote{See \cite[\S2, Definition 2.1]{ref:vinter-shvartsman-2006regularity} for the definition of normal extremals.}
\end{proposition}
\begin{proof_sketch}
Notice that the assumptions \(\text{(H1)}\)--\(\text{(H4)}\) in \cite{ref:depinho-2011lipschitz} are satisfied. Thus with the constraint qualification condition \ref{OCPdata_new_4} enforced, the proof follows immediately by invoking \cite[Theorem 2.2]{ref:depinho-2011lipschitz}. \qed
\end{proof_sketch}
\begin{remark}
Control-affine systems are ubiquitous in control theory \cite{ref:agrachev2013control}, and we focus on this class of systems to adhere to a simple set of assumptions and simple analysis. While our approach can be generalized to a fully nonlinear setting, doing so would necessitate additional assumptions on the problem data to ensure the strong subregularity \cite[Chapter 17]{ref:dontchev_variational}, and consequently, Lipschitz continuity of \(u\as(\cdot)\).
\end{remark}

\begin{proposition}\label{prop:NLP_optimizers}
Consider the OCP \eqref{eq:OCP} with its associated data \ref{OCPdata1}--\ref{OCPdata4}, and \ref{OCPdata_new_1}--\ref{OCPdata_new_4}. Then, the NLP in \eqref{eq:col_NLP} obtained via our transcription algorithm employing the approximation formula \eqref{eq:Mhu_piecewise_grid} for a sufficiently small \(h>0\), admits a global optimizer.
\end{proposition}
\begin{proof}
Define the feasible set for the NLP \eqref{eq:col_NLP} by
\begin{equation}\label{eq:feasible_set}
   \mathsf{F} \Let   \left\{\bigl(\eta_i,\bar{u}(g_i)\bigr) \;\middle\vert\;  
    \begin{array}{@{}l@{}}
      x_0 = \eta_0, h_j(\eta_i, \bar{u}(g_i) ) \le 0, \bar{u}(g_i) \in \Ubb, r_F(\eta_N) \le 0, \\ \eta_{i+1} = \st \bigl(g_{i+1},\eta_i,\bar{u}(g_i) \bigr)
        \end{array}
        \right\}.
\end{equation}
Note that
\begin{enumerate}[label=(C-\alph*), leftmargin=*, widest=b, align=left]
\item \label{cont_of_maps_1} the map \((\eta_i, \bar{u}(g_i)) \mapsto h_j ((\eta_i, \bar{u}(g_i))\) is continuous for all \(j \in \aset[]{1,\ldots,\ell}\); 
\item \label{cont_of_maps_2} the map \(\eta_N \mapsto r_F(\eta_N)\) is continuous;
\item \label{cont_of_maps_3} the map \((\eta_i, \bar{u}(g_i)) \mapsto \st \bigl(g_{i+1},\eta_i, \bar{u}(g_i)\bigr)\) is continuous due to the fact that the vector fields \(f(\cdot)\) and \(G(\cdot)\) are continuously differentiable, and in particular, the solution trajectory is absolutely continuous and bounded.  
\end{enumerate}
In view of \eqref{eq:feasible_set}, the NLP reads
\begin{equation}
\label{eq:col_NLP_feasible} 
\begin{aligned}
& \minimize_{\bigl(\eta_i,\overline{\cont}(g_i)\bigr)_{i\in J}}	&& \J^{\mathrm{d}} \Let \sum_{i=0}^{N}\rcost_i\bigl(\eta_i,\overline{u}(g_i)\bigr)+\fcost(\eta_N) \\
&\sbjto	 &&\begin{cases}
  \bigl( \eta_i, \bar{u}(g_i) \bigr) \in \mathsf{F} \quad\text{for all }i\in J.
\end{cases}
\end{aligned}
\end{equation}
Observe that, \(\mathsf{F}\) is closed set, which follows immediately due to the continuity of the constraint maps in \eqref{cont_of_maps_1}--\eqref{cont_of_maps_3} and compactness of \(\Ubb\). Note that, \(\rcost(\cdot,\cdot)\) strongly convex, and thus it is coercive/near-monotone, i.e. \(\lim_{\norm{(\dummyx,\dummyu)} \to +\infty}c(\dummyx,\dummyu) = +\infty\). Then, the existence of a global optimizer over the feasible set \(\mathsf{F}\) follows immediately from \cite[Theorem 2.32]{ref:beck2014introduction}.
\end{proof}

\begin{theorem}\label{thrm:optimal estimate}
Consider the optimal control problem \eqref{eq:OCP} with its associated data \ref{OCPdata1}--\ref{OCPdata4}, \ref{OCPdata_new_1}--\ref{OCPdata_new_4}, let the pair \(\bigl(\st\as(\cdot),\cont\as(\cdot)\bigr)\) be a normal extremal to problem \eqref{eq:OCP}, and suppose that the assumptions (II.1)--(II.6) in \cite{ref:malanowski2020convergence} hold.\footnote{Refer to footnote \(3\) for ``normal extremals".} 

For \((h,\Dd,\rzero) \in \loro{0}{+\infty}^3\):
\begin{enumerate}[label=\textup{(\roman*)}, leftmargin=*, widest=b, align=left]
    \item \label{thrm:data:i} define the set of grid points in a \(\rzero h\) radius interval around each \(t\) by \(\unifgrid \Let \aset[]{mh \suchthat m \in \Z} \cap \Ball(t,\rzero h)\) and consider the special case of \(\mathcal{U}\) (defined in \eqref{eq:Mhu_piecewise_grid_set}) given by
\begin{equation}\label{e:UF def}
   \mathcal{U}' \Let   \left\{\lcrc{0}{T} \ni t \mapsto \frac{1}{\sqrt{\pi \Dd}}\sum_{g  \in \unifgrid \cap \Tgrid} \hspace{-4mm}\bar{u}(g) \; \epower{-|t-g|^2/\Dd h^2} \in \Rbb^{\dimcon} \;\middle\vert\;  
    \begin{array}{@{}l@{}}
       \Tgrid \Let \aset[]{mh \suchthat m \in \Z} \\ \cap \lcrc{0}{T},\text{ and } (\bar{u}(g))_{g\in \Tgrid} \\  \text{ are the control}\\ \text{coefficients}
        \end{array}
        \right\};
\end{equation}

\item\label{thrm:data:ii} let \(\ulam(\cdot)\) be the control trajectory obtained by solving
\begin{equation}
	\label{eq:OCP_finite}\tag{(OCP\(\mathrm{'}\))}
\begin{aligned}
& \inf_{\cont(\cdot) \in \mathcal{U}'}	&& V_{\horizon}(\param,\cont(\cdot)) \Let \fcost\bigl(\st(\tfin)\bigr) + \int_{\tinit}^{\horizon} \rcost\bigl(\st(t), \cont(t)\bigr) \odif{t} \\
&  \sbjto		&&  \begin{cases}
\dot{x}(t) = f(x(t))+G(x(t))u(t) \,\,\text{for a.e }t \in \lcrc{0}{\horizon},\\
\st(\tinit)= \param,\,\, r_{F}(\st(\horizon)) \le 0,\\
h_j(\st(t),\cont(t)) \le 0\,\,\text{for all }t \in \lcrc{0}{\horizon},\,j \in\aset[]{1,\ldots,\ell}, \\
u(t) \in \mathbb{U}\,\,\text{for a.e }t \in \lcrc{0}{\horizon},
\end{cases}
\end{aligned}
\end{equation}
using the transcription algorithm in \S\ref{subsec:shooting} via the uniform Euler discretization of the cost and the dynamics;

\item\label{thrm:data:iii} let \((\dummyopt(g))_{g \in \Tgrid} \subset \Rbb^{\dimcon}\) be the optimizers of the ensuing NLP. 
\end{enumerate}
Define the extension \(\extdummyopt(\cdot)\) of the sequence \((\dummyopt(g))_{g \in \Tgrid}\) by
\begin{equation}\label{eq:extended_optimizers}
\begin{aligned}
\aset[]{mh \suchthat m \in \Z} \ni g \mapsto \extdummyopt(g) \Let\begin{cases}
\dummyopt(0)~&\text{for}~g<0,\\
\dummyopt(T)~&\text{for} ~g>T,\\
\dummyopt(g) &\text{elsewhere},
\end{cases}
\end{aligned}
\end{equation} 
and the approximate control trajectory (quasi-interpolated in terms of \(\extdummyopt(\cdot)\))
\begin{align}\label{eq:approx_cont_traj}
   \lcrc{0}{T} \ni t \mapsto \ulam(t) \Let \frac{1}{\sqrt{\pi\Dd}} \sum_{g \in \unifgrid}\extdummyopt(g) \; \epower{-|t-g|^2/\Dd h^2}. 
\end{align}

Under the preceding premise, for every uniform error margin \(\eps>0\), one can pick a triplet \((h,\Dd,\rzero) \in \loro{0}{+\infty}^3\) such that
\begin{equation}
	\label{eq:optimal_estimate}
    \unifnorm{\cont\as(\cdot)-\ulam(\cdot)} \le \eps. \nn
\end{equation}
\end{theorem}
A proof of Theorem \ref{thrm:optimal estimate} is given below.
\begin{remark}
Theorem \ref{thrm:optimal estimate} asserts that a control trajectory approximating the optimal control trajectory within a preassigned uniform error can be found, and such a trajectory can be derived via the quasi-interpolation engine. The trajectory \(\ulam(\cdot)\) approximates \(u\as(\cdot)\) up to some error that is controllable to any desirable precision by appropriately choosing the triplet \((h,\Dd,\rzero)\). It is important to note that the approximation error is measured \emph{in the uniform sense}, and can be controlled via the parameters \((h,\Dd,\rzero)\in \loro{0}{+\infty}^3\). This is different from standard approximation theoretic results in the sense that this error need not go to zero and \(h \to 0\). Also note that Theorem \ref{thrm:optimal estimate} is of the \emph{challenge-answer} type: given the challenge of a preassigned uniform error margin \(\eps>0\), it is possible to find a triplet \((h,\Dd,\rzero) \in \loro{0}{+\infty}^3\) (the answer) such that the given error margin is respected. 
\end{remark}

\begin{remark}
We formulated the DMS algorithm in \S\ref{subsec:shooting} on a piecewise uniform grid \(\mathcal{G}\) in order to address localization of change points and subsequent refinement of the time grid around the localized regions (see \S\ref{subsec:detection}, and \S\ref{subsubsec:meshrf_scheme}). However, we provide the error estimate in Theorem \ref{thrm:optimal estimate} on a uniform grid. Similar estimates can be furnished for a piecewise uniform grid restricting attention to uniform patches and employing the estimates in \cite[Chapter 11]{ref:mazyabook} and leveraging the fact that the error relative to the uniform metric on disjoint patches is the maximum of the errors on the individual patches. Given a uniform estimate between the optimal solutions of \eqref{eq:OCP} (on \(\Lambda\)) and \eqref{eq:col_NLP}, our interpolation mechanism remains intact and furnishes uniform approximation guarantees as demonstrated by the proof of Theorem \ref{thrm:optimal estimate}. 
\end{remark}

\begin{remark}\label{rem:holder}
Theorem \ref{thrm:optimal estimate} can be extended to H\"{o}lder continuous optimal controls as well, if such regularity results are known for a given class of optimal control problems. For example, in the setting of the optimal control problem tackled in \cite{ref:vinter-shvartsman-2006regularity} the control constraint set is a time-dependent multifunction, and it was shown that \(t \mapsto \cont\as(t)\) admits \(\gamma\)-H\"{o}lder regularity (for \(\gamma \in \lorc{0}{1}\)) under additional constraint qualification conditions and further regularity assumptions on \(\admcont(t)\); we refer to \cite{ref:vinter-shvartsman-2006regularity} for more details on this topic. To deal with such a situation, one can employ an estimate such as \cite[Theorem 2.25]{ref:mazyabook} instead of the specialized Lipschitz estimate employed in Theorem \ref{thrm:optimal estimate}. 
\end{remark}

\begin{proof_n}
Let us define \(\Tgrid \Let \aset[]{mh \suchthat m\in \Z} \cap \lcrc{0}{T}\), and recall that we defined \(\unifgrid \Let \aset[]{mh \suchthat m\in \Z} \cap \Ball(t,\rzero h)\) for each \(t \in \Rbb\) in the theorem; the latter set contains all grid points of the form \(\aset[]{mh \suchthat m \in \Z}\) on the interval \(\lcrc{t -\rzero h}{t+ \rzero h} \teL \Ball(t,\rzero h)\). For any candidate control trajectory from \(\admcon'\) such as 
\begin{align}
   \lcrc{0}{T}\ni t \mapsto \bar{u}(t) \Let \frac{1}{\sqrt{\pi \Dd}} \sum_{g \in \unifgrid \cap \lcrc{0}{T}} \hspace{-4mm}\bar{u}(g)\;\epower{-|t-g|^2/\Dd h^2} \in \Rbb^{\dimcon},
\end{align} 
where \((\bar{u}(g))_{g \in \Tgrid}\) are the control coefficients as defined in \eqref{e:UF def}, we arrive at the NLP of the form \eqref{eq:col_NLP}. Subsequently, we define the approximate control trajectory 
\begin{align}
   \lcrc{0}{T} \ni t \mapsto \ulam(t) \Let \frac{1}{\sqrt{\pi\Dd}} \sum_{g \in \unifgrid}\extdummyopt(g) \; \epower{-|t-g|^2/\Dd h^2} \in \Rbb^{\dimcon},\nn
\end{align}
where \(\extdummyopt(\cdot)\) was defined in \eqref{eq:extended_optimizers}.

Fix \(\varepsilon>0\). From Proposition \ref{prop:opt_traj_exst} we know that \(u\as(\cdot)\) is Lipschitz continuous (with Lipschitz rank \(L_0\)). To apply the estimates in \S\ref{sec:appen_A} we need to extend \(\cont\as(\cdot)\) so that the extension is Lipschitz on \(\Rbb\). To this end, we define the extension
\begin{equation}\label{eq:extended_con}
\begin{aligned}
\Rbb \ni t \mapsto \cont_{E}\as(t) \Let\begin{cases}
\cont\as(0)~&\text{for}~t \in \loro{-\infty}{0},\\
\cont\as(t)~&\text{for} ~t \in \lcrc{0}{\horizon},\\
\cont\as(\horizon) &\text{for}~t \in \loro{\horizon}{+\infty}.
\end{cases}
\end{aligned}
\end{equation}
We claim that \(\cont_{E}\as(\cdot)\) is Lipschitz with the same rank \(L_0\) as \(\cont\as(\cdot)\). Indeed, by construction \(\norm{\cont_{E}\as(t_1)- \cont_{E}\as(t_2)}=0\) for \((t_1,t_2) \in \loro{-\infty}{0}^2\) or \((t_1,t_2) \in \loro{\horizon}{+\infty}^2\). Also, since \(\cont\as(\cdot)\) is Lipschitz continuous with Lipschitz rank \(L_0\), for \((t_1,t_2)\in \lcrc{0}{\horizon}^2\) we have the inequality \(\norm{\cont_{E}\as(t_1)- \cont_{E}\as(t_2)}= \norm{\cont\as(t_1)- \cont\as(t_2)} \leq L_0\abs{t_1-t_2}\) . When \(t_1 \in \loro{-\infty}{0}\) and \(t_2 \in \lcrc{0}{\horizon}\), we have
\begin{align}\label{eq:lipschitz_estimates1}
    \norm{\cont_{E}\as(t_1) - \cont_{E}\as(t_2)} \nn &\le \norm{\cont_{E}\as(t_1) - \cont\as(0)} + \norm{\cont\as(0) - \cont_{E}\as(t_2)} \nn\\& 
    \le \norm{\cont\as(0) - \cont\as(0)} + \norm{\cont\as(0) - \cont\as(t_2)} \nn \\&   
    \le L_0 t_2 \nn \\&     
    \le L_0\abs{t_2-t_1}.
\end{align}
Similarly, one can show \(\norm{\cont_{E}\as(t_1) - \cont_{E}\as(t_2)} \le L_0\abs{t_1-t_2}\) for \((t_1,t_2) \in \loro{-\infty}{0} \times \loro{T}{+\infty}\) and \((t_1,t_2) \in \lcrc{0}{T} \times \loro{T}{+\infty}\). Hence, \(\norm{\cont_{E}\as(t_1) - \cont_{E}\as(t_2)} \le L_0\abs{t_1-t_2}\) for all \((t_1,t_2)\in \Rbb \times \Rbb\), and the proof of our claim is complete.

We also define
\begin{equation}\label{eq:extended_approx_con}
\begin{aligned}
\Rbb \ni t \mapsto \ulamext(t) \Let\begin{cases}
\ulam(0)~&\text{for}~t \in \loro{-\infty}{0},\\
\ulam(t)~&\text{for} ~t \in \lcrc{0}{\horizon},\\
\ulam(T) &\text{for}~t \in \loro{\horizon}{+\infty}.
\end{cases}
\end{aligned}
\end{equation}
It trivially follows that
\begin{align}
    \unifnorm{u\as(\cdot) - \ulam(\cdot)} \le \unifnorm{u_E\as(\cdot) - \ulamext(\cdot)},\nn
\end{align}
and we claim that it is possible to pick \((h,\Dd,\rzero) \in \loro{0}{+\infty}^3\) such that \[\unifnorm{u_E\as(\cdot) - \ulamext(\cdot)} \le \eps,\] where \(\eps>0\) was fixed above. 

Since \(u_E\as(\cdot)\) is \(L_0\)-Lipschitz, from Theorem \ref{thrm:Holder-Lipschitz-estimate} the quasi-interpolation of \(\cont_{E}\as(\cdot)\) given by (recall the expansion formula \eqref{eq:gen_quasi_sdim})
\begin{align}\label{eq:u_hat}
  \Rbb \ni t \mapsto (\mathcal{M} u_E\as)(t) \Let \frac{1}{\sqrt{\pi\Dd}} \sum_{g \in \aset[]{mh \suchthat m\in \Z}}\cont_E\as(g) \; \epower{-|t-g|^2/\Dd h^2} \in \Rbb^{\dimcon}
\end{align}
satisfies
\begin{align}\label{eq:u_hat_estimate}
    \unifnorm{(\mathcal{M} u_E\as)(\cdot)- \cont_E\as(\cdot)}\leq L_0 c_{\gamma} h\sqrt{\Dd}+ \Delta_0(\genfn,\Dd),
\end{align} 
where the various quantities in the estimate \eqref{eq:u_hat_estimate} are:
\begin{itemize}
\item \(L_0\) is the Lipschitz rank of \(\cont\as(\cdot)\);
    
\item the Gaussian generating function satisfies moment condition with \(M=2\), and therefore \(c_{\gamma}=\tfrac{1}{3}\);

\item if \(\mathcal{F}\) is the Fourier transform operator on \(\Rbb\), then
    \begin{align}
    \mathcal{E}_0(\genfn,\Dd)(\cdot) \Let \sup_{t \in \Rbb}\sum_{\nu \in \Z\setminus \{0\}}(\mathcal{F} \genfn)\bigl(\sqrt{\Dd}\nu\bigr)\epower{2\pi i t \nu}\nn
\end{align}
can be made smaller than any preassigned \(\eps'>0\) by suitable choice of \(\Dd>0\) as given in \cite[Corollary 2.13]{ref:mazyabook}; and then \(\Delta_0(\genfn,\Dd) = \mathcal{E}_0(\genfn,\Dd)\unifnorm{u_E\as(\cdot)}\). 
\end{itemize}
From \cite[Corollary 2.13]{ref:mazyabook} we know that for the preassigned \(\eps>0\), we can find \(\overline{\Dd}>0\) such that whenever \(\Dd \ge \overline{\Dd}\) we have
\begin{align}\label{eq:E_0_term}
    \mathcal{E}_0(\genfn,\Dd) \le \frac{\eps}{\blah\unifnorm{u_E\as(\cdot)}}.
\end{align} 
We pick \(\Dd \ge \overline{\Dd}\) that ensures
\begin{align}\label{eq:Delta_0_term}
    \Delta_0(\genfn,\Dd) \le \frac{\eps}{\blah}.
\end{align} 
Now we fix some \(h>0\) satisfying
\begin{align}\label{eq:h_term_first}
    h \le   \frac{\eps}{\blah L_0c_{\gamma}\sqrt{\Dd}}, 
\end{align}
which guarantees that the first term on the right-hand side of \eqref{eq:u_hat_estimate} is dominated by \(\tfrac{\eps}{\blah}\). Define the finite-sum truncation \(u^{\dagger}(\cdot)\) of \((\MM u_E\as)(\cdot)\) given in \eqref{eq:u_hat} by (see Remark \eqref{rem:truncated_sum})
\begin{align}\label{eq:truncated_quasi_sdim}
  \Rbb \ni  t \mapsto u_E^{\dagger}(t) \Let \frac{1}{\sqrt{\pi \Dd}}\sum_{g \in \unifgrid}u_E\as(g)\,\epower{-|t-g|/\Dd h^2} \in \Rbb^{\dimcon}.
\end{align}
We now proceed to show that \(\ulamext (\cdot)\) defined in \eqref{eq:extended_approx_con} satisfies \(\unifnorm{u_E\as(\cdot) - \ulamext(\cdot)} \le \eps\) via several triangle inequalities:
\begin{align}\label{eq:part_estimate}
     \unifnorm{\cont_E\as(\cdot)-\ulamext(\cdot)} &\le  \unifnorm{\cont_E\as(\cdot)-(\mathcal{M}u_E\as)(\cdot)}+ \unifnorm{(\mathcal{M}u_E\as)(\cdot) - \ulamext(\cdot)} \nn \\& 
     \stackrel{\mathclap{(\dag)}}{\le}
     \frac{\eps}{\blah}+\frac{\eps}{\blah}+\unifnorm{(\mathcal{M}u_E\as)(\cdot) - \ulamext(\cdot)} \nn \\ & \le \frac{\eps}{2}+ \unifnorm{(\mathcal{M}u_E\as)(\cdot)-u_E^{\dagger}(\cdot)}+ \unifnorm{u_E^{\dagger}(\cdot) - \ulamext(\cdot)},
\end{align}
where the inequality \((\dag)\) follows from \eqref{eq:Delta_0_term} and our choice of \(h\) in \eqref{eq:h_term_first}. For our fixed \(\eps>0\) the parameter \(\rzero>0\) can be chosen (see \eqref{eq:truncated_bound_1} and \cite[\S 2.3.2]{ref:mazyabook}) to be
\begin{align}\label{eq:choosing_R}
\rzero \Let \sqrt{\Dd} \biggl( \frac{\eps}{\blah \mathcal{B} \unifnorm{u_E\as(\cdot)}} \biggr)^{1/(1-K)},\end{align}
where the constant \(\mathcal{B}>0\) depends on \(\genfn(\cdot)\) and the decay order \(K\) in \eqref{decay}. Thus, the second term on the right-hand side of \eqref{eq:part_estimate} satisfies
\begin{align}\label{eq:truncated_bound_1_proof}
    \unifnorm{(\mathcal{M}u_E\as)(\cdot) - u_E^{\dagger}(\cdot) } \le \frac{\eps}{\blah},
\end{align}
and the estimate \eqref{eq:part_estimate} becomes
\begin{align}\label{eq:part_estimate_ii}
\unifnorm{\cont_E\as(\cdot)-\ulamext(\cdot)} \le \frac{3\eps}{4} + \unifnorm{u_E^{\dagger}(\cdot) - \ulamext(\cdot)}.
\end{align}
We now look at the second term on the right-hand side of \eqref{eq:part_estimate_ii}: 
\begin{align}\label{eq:part_estimate_third}
     \unifnorm{u_E^{\dagger}(\cdot)-\ulamext(\cdot)} & = \sup_{t \in \Rbb}\|u_E^{\dagger}(t)- \ulamext(t)\|_{\infty} \nn \\& = \sup_{t \in \Rbb} \left\| \frac{1}{\sqrt{\pi \Dd}}\sum_{g \in \unifgrid} \bigl( u_E\as(g)-\extdummyopt(g) \bigr) \;\epower{-|t-g|^2/\Dd h^2}\right\|_{\infty} \nn \\ & \le \frac{1}{\sqrt{\pi \Dd}} \sup_{t \in \Rbb}\sum_{g \in \unifgrid} \|u_E\as(g)- \extdummyopt(g)\|_{\infty} \cdot \sup_{t \in \Rbb} \left|\epower{-|t-g|^2/\Dd h^2} \right| \nn \\&  \le
     \frac{1}{\sqrt{\pi \Dd}} \sup_{t \in \Rbb}\sum_{g \in \unifgrid} \|u_E\as(g)- \extdummyopt(g)\|_{\infty} \nn \\&
     \stackrel{\mathclap{(\dag)}}{\le} 
      \frac{\mathsf{K}_0 (2\rzero+1)h}{\sqrt{\pi\Dd}},
\end{align}
where in the inequality  \((\dag)\), \(\mathsf{K}_0\) is an absolute constant independent of \(h\) which arises due to the application of \cite[Theorem 5.7]{ref:malanowski2020convergence} in the final step of \eqref{eq:part_estimate_third} to bound \(\|u_E\as(g)- \extdummyopt(g)\|_{\infty}\).  

Finally, to ensure \(\unifnorm{u_E^{\dagger}(\cdot)-\ulamext(\cdot)} \le \frac{\eps}{\blah}\) we need
\begin{align}\label{eq:h_term_pre_final}
    h \le \frac{\sqrt{\pi\Dd}\eps}{\blah \mathsf{K}_0 (2\rzero+1)}.
\end{align}
Accordingly, we update our prior selection of \(h\) in \eqref{eq:h_term_first} to a new value satisfying
\begin{align}\label{eq:h_term_final}
 h \in \left]0, \mathrm{min} \left\{\frac{\sqrt{\pi\Dd}\eps}{\blah  \mathsf{K}_0 (2\rzero+1)}, \frac{\eps}{\blah L_0 c_{\gamma}\sqrt{\Dd}} \right\} \right[
\end{align}
in order to satisfy both \eqref{eq:h_term_first} and \eqref{eq:h_term_pre_final}.

In summary, observe that \(h\) was picked in \eqref{eq:h_term_final}, \(\Dd\) was picked just above \eqref{eq:Delta_0_term}, and \(\rzero\) was picked in \eqref{eq:choosing_R} in order to satisfy our required bound. Our proof is complete. \qed
\end{proof_n}

\begin{remark}
In Theorem \ref{thrm:optimal estimate}, the function \(\genfn(\cdot) \in \mathcal{S}(\Rbb)\) can be picked arbitrarily provided it satisfies the conditions in \eqref{eq:moment} and \eqref{decay}, while ensuring that the estimate \eqref{eq:unif_estimate_lip} remains intact. Our choice of the Gaussian generating function is merely for simplicity of presentation. See \S\ref{sec:illustration_unif} for an illustration of how Theorem \ref{thrm:optimal estimate} can be applied.
\end{remark}

\begin{remark}
Note that \cite[Theorem 5.7]{ref:malanowski2020convergence} was lifted off the shelf in the preceding proof; any other uniform estimate with a convergence rate depending on \(\mathcal{O}(h^p)\) for some \(p>0\) can be employed instead. For example \cite[Theorem 2.1]{ref:ALD:WWH:EulerAppOptCon} gives an uniform estimate of order \(\mathcal{O}(h^{2/3})\) for state constrained optimal control problems under a somewhat parallel set of assumptions and that estimate under their hypotheses can be applied in the corresponding context; see the smoothness conditions in \cite[\S2, pp. 176]{ref:ALD:WWH:EulerAppOptCon}. But for Lipschitz continuous \(u\as(\cdot)\), using our approach, the convergence rate cannot be further improved. However, if \(u\as(\cdot)\) is more regular,  for instance, if \(u\as(\cdot) \in \mathcal{C}^2([0,T];\Ubb)\), one can derive the following estimate instead of the estimate in \eqref{eq:u_hat_estimate}:
\begin{align}\label{eq:u_hat_estimate}
    \unifnorm{(\mathcal{M} u_E\as)(\cdot)- \cont_E\as(\cdot)}\leq\mathcal{O}(h^2) + \Delta_0(\genfn,\Dd); \nn 
\end{align}
see \cite[Theorem 2.17, pp. 34]{ref:mazyabook}. For such OCPs, employing higher-order discretization schemes, such as RK4 discretization \cite{ref:hager_runge_2000_1,ref:hager_runge_2000_2,ref:hager_runge_dontchev2000second}, can lead to uniform convergence of order \(\mathcal{O}(h^2)\), up to a controllable saturation error.
\end{remark}

\begin{remark}\label{rem:on_triangle_ineq}
The extension \(u_E\as(\cdot)\) of the optimal control trajectory \(u\as(\cdot)\) and \(\ulamext(\cdot)\) of the approximating trajectory \(\ulam(\cdot)\) have the same domain \(\Rbb\) and we can apply various uniform error estimates recorded in \S\ref{sec:appen_A} directly. In particular, the ability to furnish uniform estimates under the finite-sum truncation of the infinite quasi-interpolant is noteworthy (see Remark \ref{rem:truncated_sum}) and was crucial for us. This feature helps us to arrive at the estimate \eqref{eq:part_estimate_third}, with the truncation parameter \(\rzero\) capturing the number of terms in the sum, in the absence of which it will be difficult to bound all the terms in \eqref{eq:part_estimate_third}.
\end{remark}

The following result is specifically designed for linear optimal control problems.
\begin{corollary}\label{thrm:lin optimal estimate}
Consider the linear OCP
\begin{equation}
	\label{eq:Lin_OCP}
\begin{aligned}
& \inf_{\cont(\cdot) }	&& \hspace{-3mm}\inprod{x(\horizon)}{P x(\horizon)} + \int_{\tinit}^{\horizon}\hspace{-1mm} \bigl(\inprod{x(t)}{Q x(t)}+ \inprod{u(t)}{P u(t)} \bigr) \odif{t} \\
&  \sbjto		&&  \hspace{-3mm}\begin{cases}
\dot \st(t) = Ax(t)+Bu(t)\,\,\text{for a.e.\ } t \in \lcrc{0}{T},\,x_0=\overline{x},\\
\stconset(x(t)) \le 0\text{ and } \contconset(u(t))\le 0 \text{ for all } t \in \lcrc{0}{T},\\
\st(\horizon) \in \mathsf{C} \subset \Rbb^d,
\end{cases}
\end{aligned}
\end{equation}
where \(A \in \Rbb^{d \times d}\), \(B \in \Rbb^{d \times \dimcon}\), \(Q \in \Rbb^{d \times d}\) and \(P \in \Rbb^{d \times d}\) are positive semidefinite, \(R \in \Rbb^{\dimcon \times \dimcon}\) is positive definite, and \(\mathsf{C}\) is is compact and convex set. Let the pair \(\bigl(\st\as(\cdot),\cont\as(\cdot)\bigr)\) be a normal extremal for problem \eqref{eq:Lin_OCP}. Suppose that the functions \(\stconset(\cdot)\) and \(\contconset(\cdot)\) are twice continuously differentiable and convex, and that \(\xi \mapsto \nabla_{\xi} \stconset(\xi)\) is also twice continuously differentiable. Moreover, suppose that there exists \(\alpha \in \loro{0}{+\infty}\) such that for all \(t \in \lcrc{0}{\horizon}\) and all \(\overline{y}\), the constraint qualification condition 
\[
\norm{\bigl(V_c(t)^{\top}, B^{\top}V_s(t)^{\top}\bigr) \overline{y}} \geq \alpha \norm{\overline{y}}
\]
holds, where \(V_c(t)\) is the matrix whose rows are the gradients of the components of \(\contconset(\cdot)\) evaluated at \(u\as(t)\), and \(V_s(t)\) is defined similarly for \(\stconset(\cdot)\). Then, under the data \ref{thrm:data:i}--\ref{thrm:data:iii} in Theorem \ref{thrm:optimal estimate} (with obvious modifications), for every uniform error margin \(\eps>0\), one can choose a triplet \((h,\Dd,\rzero) \in \loro{0}{+\infty}^3\) such that
\begin{equation}
	\label{eq:optimal_estimate_final_linear}
    \unifnorm{\cont\as(\cdot)-\ulam(\cdot)} \le \eps. 
\end{equation}
\end{corollary}
\begin{proof}
The existence of \(u\as(\cdot)\) and Lipschitz continuity of \(t \mapsto u\as(t)\) follow directly from \cite[Theorem 4.2]{ref:hager1979lipschitz}, \cite[Theorem 3.1]{ref:vinter-galbraith-lipschitz}. Subsequently, the estimate \eqref{eq:optimal_estimate_final_linear} follows immediately by applying the same steps as in Theorem \ref{thrm:optimal estimate}.
\end{proof}

\subsection{A couple of illustrations on how to employ Theorem \ref{thrm:optimal estimate}}\label{sec:illustration_unif}
Using an example, we demonstrate how to apply the technique established in Theorem \ref{thrm:optimal estimate} to achieve a preassigned uniform error margin by selecting appropriate approximation parameters.
\subsection*{The Aly-Chan problem}
Let us consider the benchmark \emph{Aly-Chan} optimal control problem given by: 
\begin{equation}
\label{eq:aly-chan-ocp}
\begin{aligned}
& \minimize_{\cont(\cdot)}	&&  x_3 \left( \frac{\pi}{2}\right) \\
&  \sbjto		&&  \begin{cases}
\dot{\st}_1(t) =x_2(t),\, \dot{\st}_2(t)=u(t),\\  \dot{\st}_3(t)=\frac{1}{2}\left(x_2(t)^2-x_1(t)^2\right),\\ 
\st_1(0)=0, \,\st_2(0)=1, \,\,\text{and}\,\,\st_3(0)=0,\\
|u(t)|\le 1. 
\end{cases}
\end{aligned}
\end{equation}
Note that the problem data satisfies the hypothesis of Theorem \ref{thrm:optimal estimate}. Indeed, the analytical optimal control trajectory is known and it is given by \(t \mapsto u\as(t) \Let -\sin(t)\), which is Lipschitz continuous with Lipschitz rank \(L_0=1\). 

Fix \(\eps = 0.01\) and employ the Gaussian generating function \(t \mapsto \genfn(t) \Let \frac{1}{\sqrt{\pi}} \epower{-|t|^2}\). Note that, \(\genfn(\cdot)\) satisfies the moment condition with order \(M=2\); and thus \(c_{\gamma}=\tfrac{1}{3}\). We will pick the pair \((h,\Dd, \rzero) \in \loro{0}{+\infty}^3\) such that \( \unifnorm{\cont\as(\cdot)-\ulam(\cdot)} \le 0.01\). As \(u\as(\cdot)\) is Lipschitz, we have the straight forward estimate 
\begin{align}\label{num:aly_chan_Lip_estimate}
     \unifnorm{(\mathcal{M} u_E\as)(\cdot)- \cont_E\as(\cdot)}\leq  \tfrac{h}{3}\sqrt{\Dd} +\Delta_0(\genfn,\Dd).
\end{align}
The saturation error \(\Delta_0(\genfn,\Dd) \Let \mathcal{E}_0(\genfn,\Dd)\unifnorm{u_E\as(\cdot)}\), in which the term
\begin{align}\label{num:AC_E_0_term}
    \mathcal{E}_0(\genfn,\Dd) = \sup_{t \in \Rbb}\sum_{\nu\in \Z\setminus\{0\}}\bigl|\mathcal{F}\genfn\bigl(\sqrt{\Dd}\nu\bigr)e^{2\pi i t \nu}\bigr|= \tfrac{1}{\sqrt{\Dd}}\sum_{\nu \in \Z \setminus\{0\}}\bigl| \epower{-\pi^2 \Dd \nu^2}\bigr|,\nn
\end{align}
which admits a numerical value of \(5.35\times 10^{-9}\) when \(\Dd=\overline{\Dd}=2\) \cite[Chapter 3]{ref:mazyabook} and as per \eqref{eq:Delta_0_term}, and thus \(\Delta_0(\genfn,\Dd) \le \tfrac{\eps}{4}\) is respected. For the fixed \(\Dd=2\), we pick the discretization parameter \(h=\frac{\eps}{4L_0 c_{\gamma}\sqrt{2}}=0.0054\); consequently, the estimate in \eqref{eq:u_hat_estimate} holds. Let us refer to \(h\) as \(h_{\text{init}}\) for now.
\begin{figure}[h!]
\centerline{\includegraphics[scale=0.6]{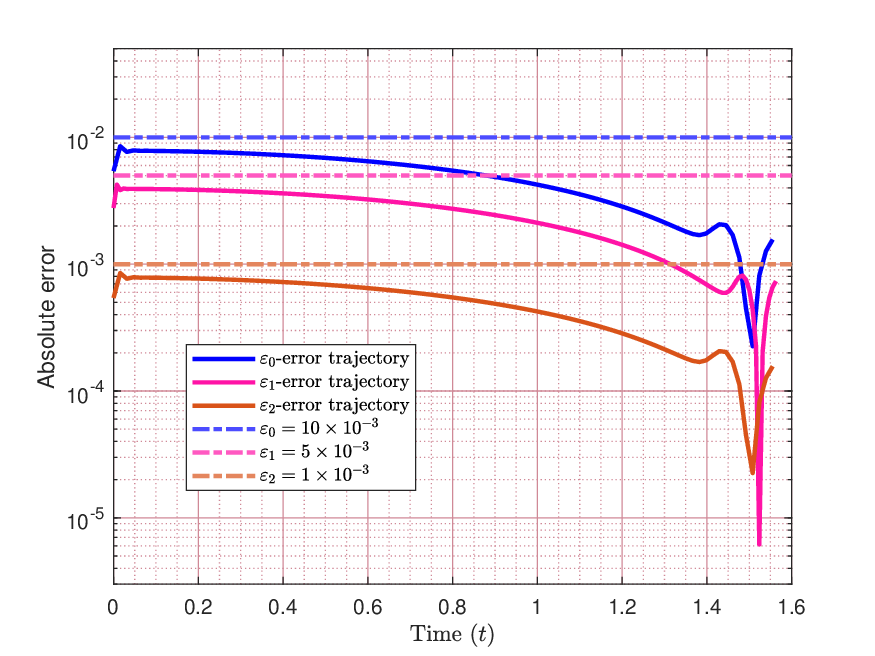}}
\caption{Error trajectory for different thresholds \(\eps\) for the Aly-Chan problem.}
\label{fig:Alychan_error}
\end{figure}
\begin{figure}[!ht]
  \centering
  \begin{subfigure}[b]{0.49\linewidth}
    \includegraphics[width=\linewidth]{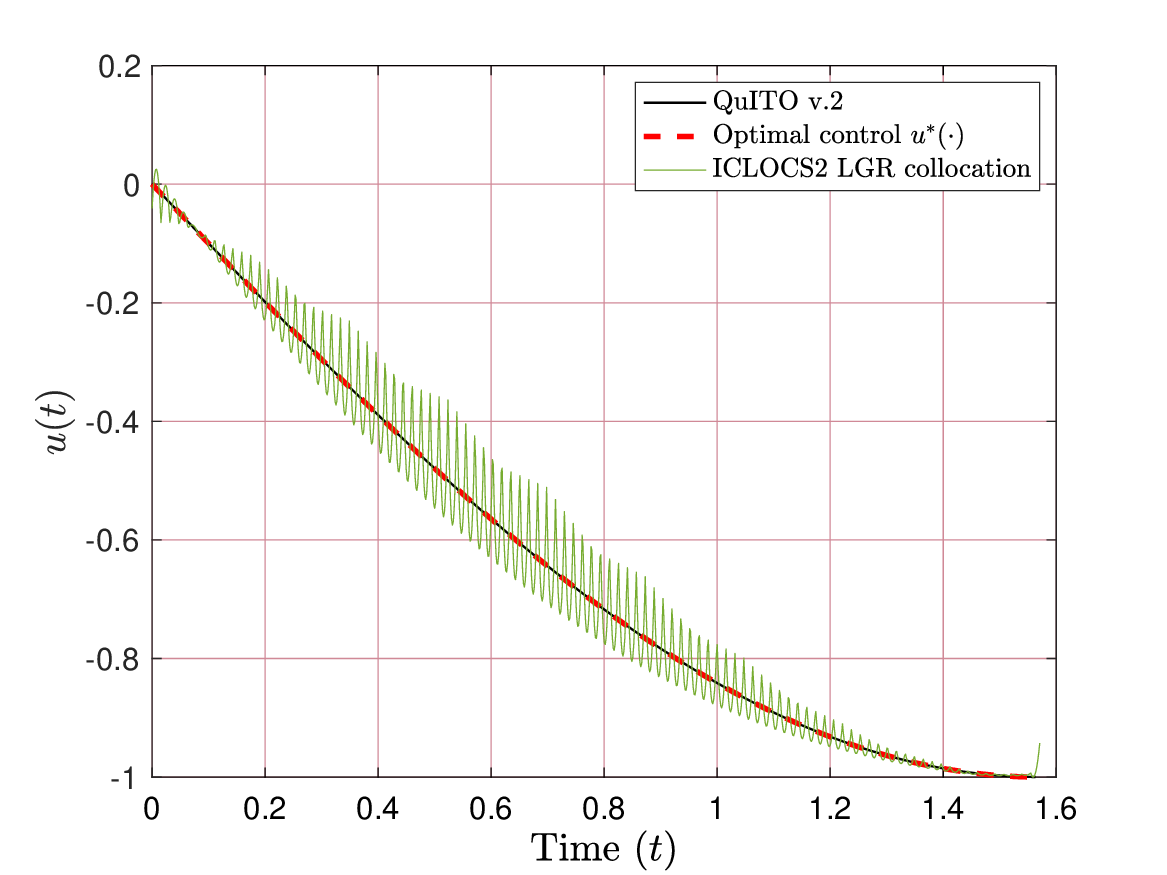}
  \end{subfigure}
  \begin{subfigure}[b]{0.49\linewidth}
    \includegraphics[width=\linewidth]{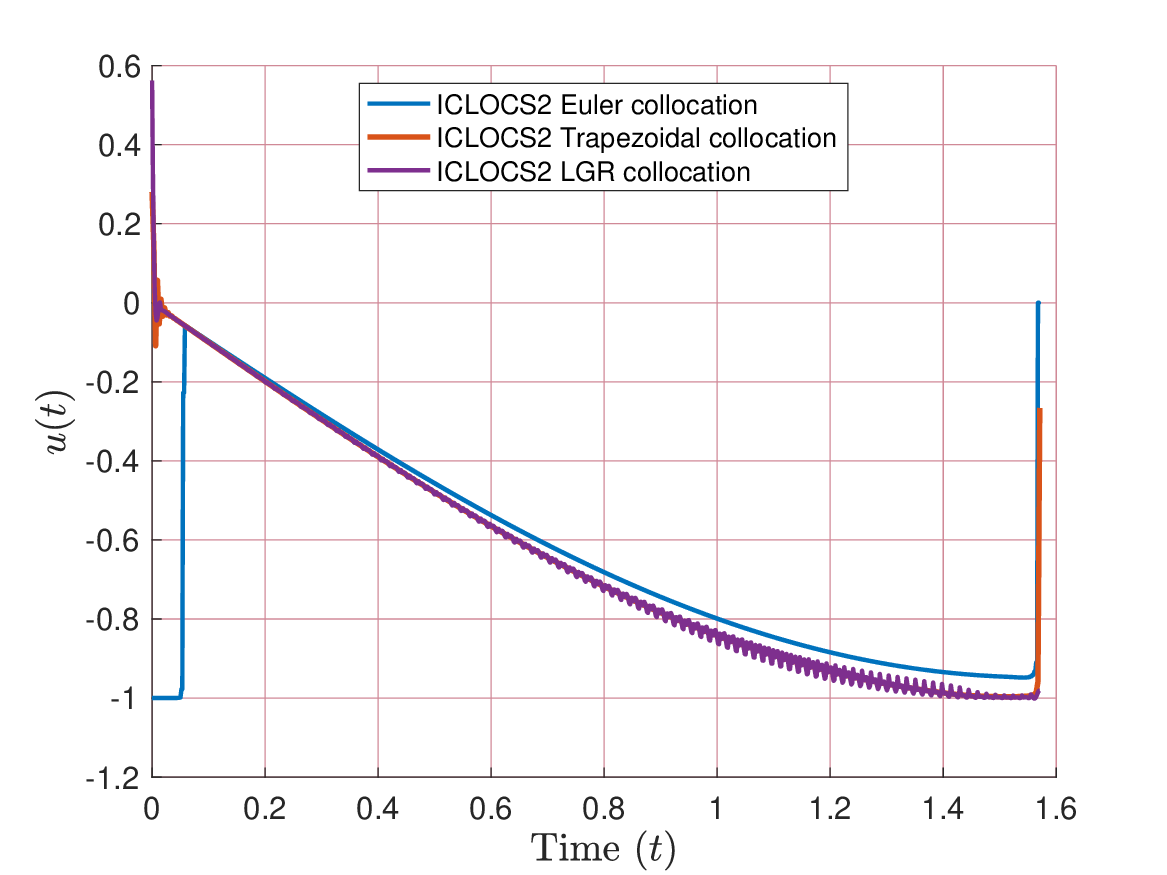}
  \end{subfigure}
 \caption{The left-hand subfigure shows numerical control trajectories obtained via ICLOCS2 by employing LGR collocation using Legendre polynomials for state and control representation for the Aly-Chan problem. The right-hand subfigure depicts Euler, Trapezoidal, and Hermite-Simpson collocation using piecewise cubic hermite interpolating polynomial for state and control representation for the same problem.
}
 \label{fig:aly_control}
\end{figure}
Following \cite[\S 1.2.1]{ref:mazyabook} we pick \(\rzero = 5\) and thus we have the direct analog of estimate in \eqref{eq:part_estimate_ii}. Finally, to update \(h\) as per \eqref{eq:h_term_final} we need a numerical value of the constant \(\mathsf{K}_0\). To this end, employ the same technique as \cite[Eq. 6.17 and 6.18]{ref:malanowski2020convergence} --- we used the parameters \(\bigl(h_{\text{init}},\Dd,\rzero\bigr)\) in the quasi-interpolation engine and solved the associated NLP. We observed that \(\norm{u\as_E(g) - \extdummyopt(g)}_{\infty}\) is \(0.0083\); consequently \(\mathsf{K}_0 = \tfrac{\norm{u\as_E(g) - \extdummyopt(g)}_{\infty}}{h_{\text{init}}} = \tfrac{0.0083}{0.0054} = 1.53.\) 
Then, the term on the right-hand side of \eqref{eq:h_term_pre_final} is \(\tfrac{\sqrt{\pi \Dd} \eps}{4\mathsf{K}_0 (2 \rzero+1)}=0.0038\). Finally, we update \(h_{\text{init}}\) with \(h =  \mathrm{min} \left\{0.0054,0.0038\right\} = 0.0038\). We have depicted the error trajectories for \(\eps \in \aset[]{0.01,0.005,0.001}\) in Figure \ref{fig:Alychan_error} where the vertical axis is in log scale for better visualization. The control trajectories are illustrated in Figure \ref{fig:aly_control}, which clearly demonstrates the presence of ringing phenomena in higher-order collocation techniques. In contrast, the numerical trajectories produced by \(\quito\) are accurate and free from ringing.


\subsection*{The Bryson-Denham problem} We provide illustration for another example. Consider the benchmark path constrained Bryson-Denham optimal control problem \cite{ref:Bry75} given by:
\begin{equation}
\label{eq:Bryson_OCP}
\begin{aligned}
& \minimize_{\cont(\cdot)}	&&  \int_{0}^{1} \cont(t)^2 \, \dd t \\
&  \sbjto		&&  \begin{cases}
\dot{x}_1(t)=x_2(t),\,\,\dot{x}_2(t)=u(t),\\
\st(0)=(0,1)^{\top},\, \st(1)\Let (0,-1)^{\top},\\
\st_1(t)\le l \Let \frac{1}{9} \quad\text{for all}\,\,t \in [0,1].
\end{cases}
\end{aligned}
\end{equation}
Note that the analytical optimal control trajectory is known and it is given by \cite{ref:Bry75}
\begin{equation}
\label{eq:Bry_analytic_u} 
 \begin{aligned}
 u^{*}(t) \Let \begin{cases}
  -\frac{2}{3l}\left( 1-\frac{t}{3l}\right)~&\text{if}\quad 0\le t \le 3l,\\
 0~&\text{if}\quad 3l \le t \le 1-3l,\\
 -\frac{2}{3l}\left( 1-\frac{1-t}{3l}\right) &\text{if}\quad 1-3l \le t \le 1, \nn
 \end{cases}
 \end{aligned}
 \end{equation}
which is Lipschitz continuous with Lipschitz rank \(L_0=18\). 

As before, fix \(\eps = 0.05\) and for illustration, we employ a higher-order Lagguerre-Gaussian generating function 
\begin{align}
    t \mapsto \genfn(t) \Let \frac{1}{\sqrt{\pi}} \left(\frac{315}{128} - \frac{105}{16}t^2 + \frac{63}{16}t^4 - \frac{3}{4}t^6 + \frac{1}{24}t^8 \right) \epower{-|t|^2}
\end{align}
which satisfies satisfies the moment condition with order \(M=10\); and thus \(c_{\gamma}=\tfrac{1}{11}\). Similarly to the Aly-Chan example, we have the estimate
\begin{align}\label{num:bd_chan_Lip_estimate}
     \unifnorm{(\mathcal{M} u_E\as)(\cdot)- \cont_E\as(\cdot)}\leq  \tfrac{18}{11}h\sqrt{\Dd} +\Delta_0(\genfn,\Dd).
\end{align}
\begin{figure}[h]
\centerline{\includegraphics[scale=0.6]{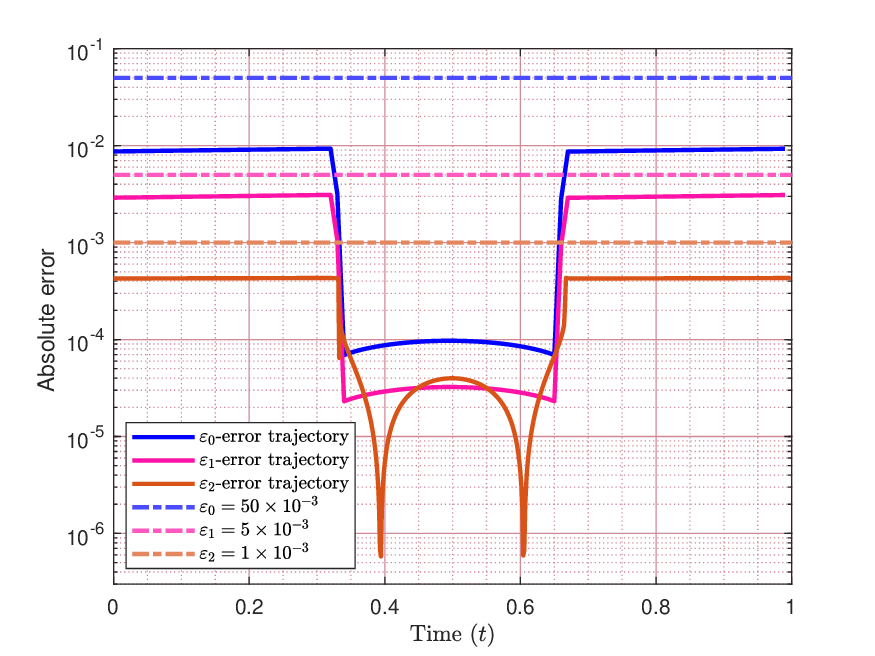}}
\caption{Error trajectory for different thresholds \(\eps\) for the Bryson-Denham problem. }
\label{fig:brysondenham_error}
\end{figure}
\begin{figure}[h]
\centerline{\includegraphics[scale=0.45]{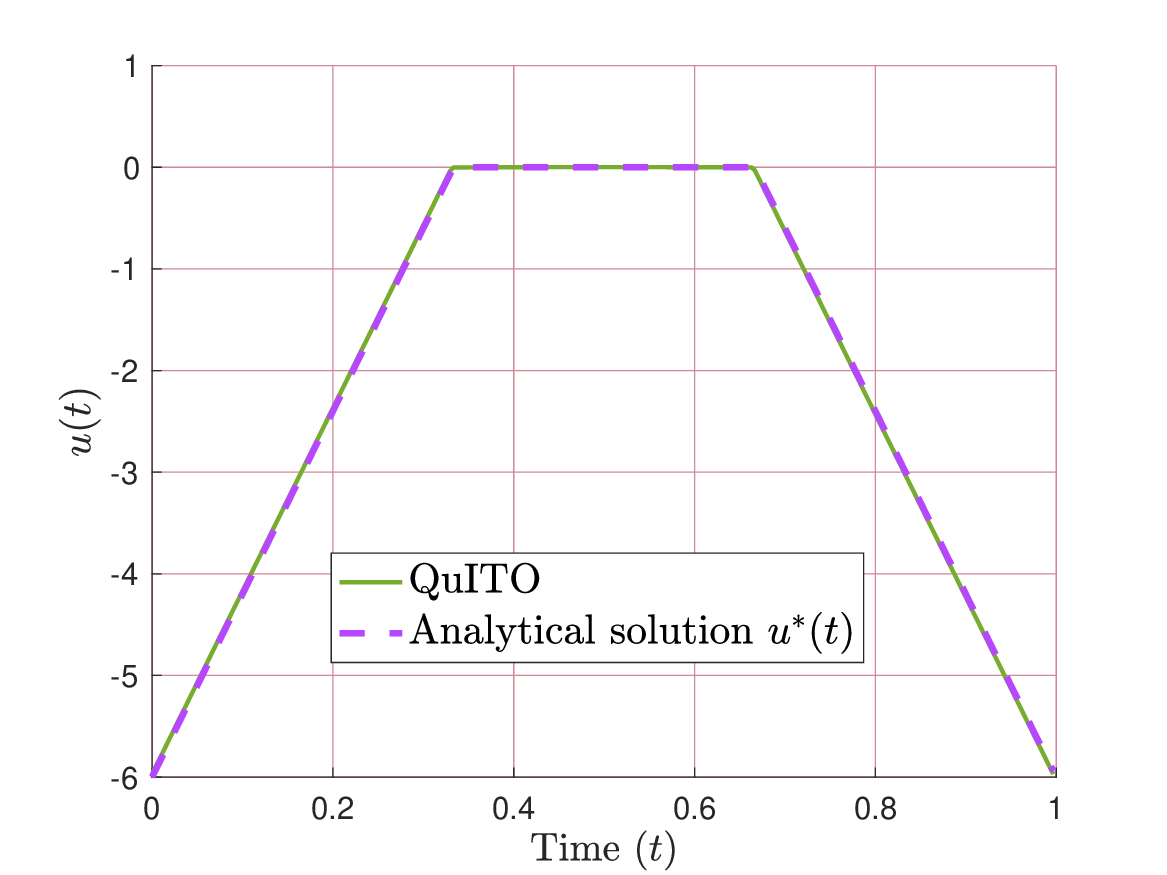}}
\caption{The analytical solution \(u\as(\cdot)\) and the numerical control trajectory obtained by \(\quito\).}
\label{fig:brysondenham_control}
\end{figure}
due to Lipschitz continuity of \(u\as(\cdot)\). Fixing \(\Dd = 2\) as per \cite[Chapter 3]{ref:mazyabook}, the term \( \mathcal{E}_0(\genfn,\Dd)\) evaluates to \(8.29\times 10^{-5}\) when \(\Dd=\overline{\Dd}=2\) and as per \eqref{eq:Delta_0_term}, and thus \(\Delta_0(\genfn,\Dd) \le \tfrac{\eps}{4}\). We pick the discretization parameter \(h_{\text{init}}=\frac{\eps}{4L_0 c_{\gamma}\sqrt{2}}=0.005\); consequently, the estimate in \eqref{eq:u_hat_estimate} holds. Following \cite[\S 1.2.1]{ref:mazyabook} we pick \(\rzero = 5\) and thus we have the direct analog of estimate in \eqref{eq:part_estimate_ii}. As earlier, we used the parameters \(\bigl(h_{\text{init}},\Dd,\rzero\bigr)\) in the quasi-interpolation engine and observed that \(\norm{u\as_E(g) - \extdummyopt(g)}_{\infty} = 0.0091\). This gives us \(\mathsf{K}_0 = \tfrac{\norm{u\as_E(g) - \extdummyopt(g)}_{\infty}}{0.005}=1.82\), and consequently the term on the right-hand side of \eqref{eq:h_term_pre_final} evaluates to \(0.001\). Finally, we update \(h_{\text{init}}\) with \(h =  \mathrm{min} \left\{ 0.005,0.001\right\} = 0.001\). Note that this value of \(h\) is conservative is the same sense as described in Remark \ref{rem:on_conservatism}. We have depicted the error trajectories for \(\eps \in \aset[]{0.05,0.005,0.001}\) in Figure \ref{fig:brysondenham_error} where the vertical axis is in log scale for better visualization. The numerical and the analytical control trajectories are depicted in \ref{fig:brysondenham_control}.

\begin{remark}\label{rem:on_conservatism}
Notice that the value of \(h\) furnished by our result is conservative due to very nature of our proof technique which involves bounding the uniform error via iterative application of triangle inequalities. Moreover, some conservatism is also introduced due to the application of \cite[Theorem 5.7]{ref:malanowski2020convergence} in estimate \eqref{eq:part_estimate_third}. One can observe that, for both the examples, with \(h_{\text{init}}\), the prespecified tolerance \(\eps\) was met.
\end{remark}

\subsection{Change point localization algorithm}\label{subsec:detection}
Although we have primarily addressed optimal control problems for which the optimal control trajectories are Lipschitz continuous (see Proposition \ref{prop:opt_traj_exst}), it is well-known that from a numerical viewpoint, the efficacy of mesh refinement algorithms are best illustrated for problems where the control trajectories are more irregular. With this practical need in mind we developed a mesh refinement algorithm targeting classes of optimal control problems that feature \emph{change points} involving kinks, narrow peaks, valleys over certain epochs, and even discontinuities. To this end, we first establish a localization algorithm that detects any abrupt change points in the control profile and acts as an input oracle to the mesh refinement algorithm that adds points locally (based on the output of the localization algorithm) to the time grid on top of an initial fixed, preferably coarse, uniform grid, and subsequently solves the ensuing optimal control on a denser and piecewise uniform grid via parameterizing the control trajectory by the quasi-interpolant \eqref{eq:Mhu_piecewise_grid}. Notice that the use of the quasi-interpolant \eqref{eq:Mhu_piecewise_grid} is natural in the context of mesh refinement because the scheme accommodates uniform approximation on piecewise uniform grids, which arises naturally in our context. 

We need some essential elements from the theory of wavelets: Given a (sufficiently regular) \emph{mother wavelet} \(\genwave:\Rbb \lra \Rbb\), the \emph{wavelet transform} decomposes a given function in a certain class over translated and dilated wavelets generated from the mother wavelet. Translation helps in localization and scaling allows a closer look at the signal at multiple scales. A \emph{mother wavelet} \(\Rbb \ni t \mapsto \genwave(t) \in \Rbb\) is a measurable function with finite energy and zero mean:
\begin{equation}\label{eq:wavelet_zero_average}
\norm{\genwave(\cdot)}_{\lpL[2]} <+\infty \quad \text{and} \,\,\int_{\Rbb} \genwave(t) \,\dd t = 0.
\end{equation}
We normalized \(\genwave(\cdot)\) to ensure \( \norm{\genwave(\cdot)}_{\lpL[2]} = 1\). We define the dictionary or the basis set of \emph{time-frequency atoms} generated by translation and scaling of the mother wavelet by: 
\begin{equation}\label{eq:wavlelt_dictionary}
    \mathcal{A} \Let \aset[\bigg]{ \genwave_{a,b}(t) = \frac{1}{\sqrt{b}} \genwave \biggl( \frac{t-a}{b} \biggr) }_{a \in \Rbb,\,b\in \loro{0}{+\infty}}.
\end{equation}
The wavelet transform of any given \(\sys(\cdot) \in \lpL[2](\Rbb;\Rbb)\) at time \(a\) and scale \(b\) is the projection of \(\sys(\cdot)\) on the corresponding wavelet atom defined by
\begin{equation}\label{eq:WT}
    (W \sys) (a,b) = \langle \sys,\genwave_{a,b} \rangle = \int_{\Rbb} \sys(t) \frac{1}{\sqrt{b}} \genwave \Bigl( \frac{t-a}{b} \Bigr)\, \dd t.
\end{equation}
Here we consider mother wavelets \(\genwave(\cdot)\) that satisfy the so-called \emph{admissibility condition}; see \cite[Theorem 4.4]{ref:mallat1999wavelet}. 
\begin{remark}\label{rem:why:wavelets}
The WT is widely used in signal processing for detecting irregularities and change points due to its advantageous properties. These include multiresolution analysis which facilitates feature detection at different scales, guarantees for localization and detection under mild conditions, and time-frequency localization enabling precise identification of when specific frequencies occur. Motivated by these properties, we employ the WT in this work.
\end{remark}
\begin{remark}\label{rem:on_Lipschitz_reg_of_approximate_control}
Local regularity analysis using the wavelet transform is well-known in signal processing \cite{ref:mallat1999wavelet}, where certain decay rates of the amplitude of the wavelet transform point towards H\"{o}lder/Lipschitz continuity of the underlying signal. The primary tool that enables this feature is called the \emph{wavelet zoom}. This is performed via diminishing the scale \(b\): in particular when \(f(\cdot)\) is Lipschitz continuous and \(b \downarrow 0\), then \(|(Wf)(a,b)| = \mathcal{O}(b^{3/2})\), and thus fine-scale properties of the signal in the neighborhood of \(a\in \Rbb\) is captured. Another important characterization that is highly relevant for us is to locate the local maxima of the wavelet as change points are localized by finding the abscissae where the wavelet modulus maxima converge at fine scales. A wavelet modulus maximum is defined as a point \( \left(a_0,b_0 \right)\) such that \(|(W \sys)(a,b_0)|\) is locally maximum at \(a=a_0\).
\end{remark}


\begin{remark}
From now on baseline-\(\quito\) will indicate that the transcription algorithm in \S\ref{subsec:shooting} in employed with a finite uniform grid \(\Lambda \Let \{s_i\}_{i \in I}\), where \(I \Let \aset[]{0,1,\ldots,N}\) for some \(N \in \N\), such that
\begin{align}\label{eq:uniform_partition_boxed}
     \lcrc{0}{\horizon} = \bigcup_{i=1}^{N} \interval_i,\text{ where }\interval_i \Let \lcrc{s_{i-1}}{s_i}, \,s_i \Let ih,\text{ and }h \Let \tfrac{T}{N}.
\end{align}
In particular, \(\Lm\) contains grid points within the time interval \(\lcrc{0}{T}\) only, unlike the uniform grid \(\unifgrid\) we defined in Theorem \ref{thrm:optimal estimate} for technical reasons. By \(\quito\) \(\vertwo\) we will refer to the algorithm established herein consisting of the transcription, localization, and refinement together (see Algorithm \ref{alg:refinement_algo} ahead), and the first step of \(\quito\) \(\vertwo\) involves baseline-\(\quito\), i.e., solving the given OCP \eqref{eq:OCP} on a uniform grid.
\end{remark}
\begin{theorem}
\label{thrm:wavelet_decay_rates}
Let \(\ulam(\cdot)\) be the approximate control trajectory obtained by solving \eqref{eq:col_NLP} via the baseline-\(\quito\) on the uniform grid \(\waveunifgrid\), i.e., 
\begin{align}\label{eq:unif_quasi_traj}
   \lcrc{0}{T} \ni t \mapsto \ulam(t) \Let \frac{1}{\sqrt{\pi\Dd}} \sum_{g \in \Lm}\dummyopt(g) \; \epower{-|t-g|^2/\Dd h^2} \in \Rbb^{\dimcon},
\end{align}
where \((\dummyopt(g))_{g \in \Lm} \subset \Rbb^{\dimcon}\) are the optimizers obtained from solving the NLP \eqref{eq:col_NLP}. Let \(n,m\in \N\), fix a mother wavelet \(\genwave(\cdot) \in \mathcal{C}^n(\Rbb;\Rbb)\) such that \(\genwave(\cdot)\) has \(n\) vanishing moments and \(\genwave(\cdot)\) satisfies the decay condition \(|(\genwave)^k(t)| \le \frac{K_{\genwave}}{1+|t|^m}\) for some \(K_{\genwave}>0\) and for \(k=0,\ldots,n\).\footnote{The term \((\genwave)^k\) denotes \(k^{\text{th}}\) derivative of \(\genwave(\cdot)\), \(k=0,\ldots,n\).} Then there exists an absolute constant \(A>0\) such that \begin{align}\label{eq:Holder_estimate}
    |(W\ulam)(a,b)| \le A b^{3/2} \quad \text{for all }a,b \in \Rbb \times \loro{0}{+\infty}.
\end{align}
\end{theorem}
\begin{proof}
Note that \(\ulam(\cdot)\) is a finite sum of weighted and scaled Gaussian generating functions that are bounded, smooth, and rapidly decaying. Moreover for each \(g \in \Lm\), these Gaussians are square-integrable and Lipschitz continuous, and thus the finite sum \eqref{eq:unif_quasi_traj} is square-integrable and Lipschitz. Consequently, for a given mother wavelet \(\genwave(\cdot)\) satisfying our hypotheses, the estimate
\begin{align}
    |(W\ulam)(a,b)| \le A b^{3/2} \quad \text{for all }a,b \in \Rbb \times \loro{0}{+\infty},\nn
\end{align}
follows immediately from \cite[Chapter 6, Theorem 6.4]{ref:mallat1999wavelet}. 
\end{proof}
\begin{remark}
Let the mother wavelet \(\genwave(\cdot)\) be continuously differentiable and suppose that it satisfies the hypotheses of Theorem \ref{thrm:wavelet_decay_rates}. In addition, let \(\genwave(\cdot)\) have compact support. Define \(\theta:\Rbb \lra \Rbb\) such that \(\int_{\Rbb}\theta(t)\,\dd t \neq 0\) with \(\genwave(\cdot) = (-1)^n \theta^{(n)}(\cdot)\). Then from \cite[Chapter 6, Theorem 6.5]{ref:mallat1999wavelet} it follows that there exists \(b_0>0\) such that \(|(W\ulam)(a,b)|\) has no local maximum for \(a\in \lcrc{0}{\horizon}\) and \(b < b_0\), then \(\ulam(\cdot)\) is uniformly Lipschitz on \(\lcrc{\bar{\delta}}{\horizon-\bar{\delta}}\), for any \(\bar{\delta} > 0\). The estimate in Theorem \ref{thrm:wavelet_decay_rates} in conjunction with the preceding fact guarantees that all change points are localized by following the wavelet transform modulus maxima at the fine scale. Further refinements of the estimate \eqref{eq:Holder_estimate} are possible which can cater to particular cases and types of change points/singularities in \(\ulam(\cdot)\); we refer \cite[Chapter 6, \S6.1.3]{ref:mallat1999wavelet} for more details. We make a note that we will be employing derivative of the Gaussian function as mother wavelets in our numerical examples. Even though the derivative of Gaussian wavelets do not have compact support, in practice this feature has very little effect on the results. Moreover, the choice of derivatives of Gaussian functions as the mother wavelet guarantees that a modulus maximum located at \((a_0,b_0)\) belongs to a maxima line that propagates toward finer scales \cite[Chapter 6, Theorem 6.6]{ref:mallat1999wavelet}.\footnote{See \cite[\S 6.2.1]{ref:mallat1999wavelet} for a precise definition of the term maxima line.}

\end{remark}
\begin{algorithm2e}[htpb]
\DontPrintSemicolon
\SetKwInOut{ini}{Initialize}
\SetKwInOut{giv}{Data}
\SetKwInOut{nlp}{NLP}
\SetKwInOut{out}{Output}
\giv{\(\rcost(\cdot,\cdot), \fcost(\cdot), f(\cdot), G(\cdot), h_j(\cdot), \admcont, \param, r_F(\cdot)\)}
\ini{\(\circ\) A uniform grid \(\Lm\), \(\genfn (\cdot) \in \mathbb{B}\), \(\Dd \in \loro{0}{+\infty}\), and  \(b_{0}\)}
\nlp{\(\circ\) Solve the nonlinear program \eqref{eq:col_NLP} on \(\Tgrid\) \\
\(\circ\) Obtain \(\ulam(\cdot)\) using \eqref{eq:unif_quasi_traj} and record the cost \(\refcost\)}
Choose mother wavelet \(\genwave(\cdot)\) and perform wavelet transform of \(\ulam(\cdot)\) \\
Define \(\tau \Let \aset[\big]{t\as\in \lcrc{0}{\horizon} \suchthat \frac{\partial  (W \ulam)}{\partial t}\big \rvert_{t=t\as}(t,b_0) = 0}\) and find \(t \as \in \tau\)\\
\out{\((W\ulam)(a,b)\) where \((a,b) \in \Rbb \times \loro{0}{+\infty}\) and \(t\as \in \tau\).}  
\caption{The wavelet-based localization algorithm}
\label{alg:detection_algo}
\end{algorithm2e}
\begin{remark}
Multiresolution- and wavelet-based methods for localization of variations in numerical control trajectories have been employed in the past. In \cite{ref:multires:SJ:CDC07}, the authors developed a multiresolution-based trajectory optimization algorithm, and \cite{ref:WaveletDetect:2005:WM, ref:WaveletDetect:2014:WM} established dynamic optimization algorithms with a wavelet-based localization engine catering to specific problems arising from chemical process industries; no theoretical guarantees of convergence were provided there. In contrast, our results contain uniform convergence guarantees, localization guarantees, and a numerical algorithm with a software package, catering to different classes of numerically difficult trajectory optimization problems for low through moderately high dimensional systems.
\end{remark}
\subsection{Mesh Refinement Scheme.}\label{subsubsec:meshrf_scheme}
We now present a targeted \(h\)-refinement (localized addition of grid/mesh points over the localized time-patches) algorithm for control trajectories leveraging knowledge from the localization mechanism described in \S\ref{subsec:detection}. The overarching idea behind the mesh refinement algorithm is to insert additional grid points at the localized patches in an iterative fashion until an error threshold in terms of the optimized cost function (as defined in \eqref{eq:cost_decrease}) is satisfied. 
 
\par Recall that we defined the uniform grid \(\Lambda \Let \aset[]{s_i}_{i \in I}\) in \eqref{eq:uniform_partition_boxed} with the property 
\begin{align}\label{eq:uniform_partition}
     \lcrc{0}{\horizon} = \bigcup_{i=1}^{N} \interval_i,\text{ where }\interval_i \Let \lcrc{s_{i-1}}{s_i}, \,s_i \Let ih,\text{ and }h \Let \frac{\horizon}{N}.
\end{align}
We transcribe the OCP \eqref{eq:OCP} as described in \S\ref{subsec:shooting} employing baseline-\(\quito\) via parameterizing the control trajectory over the uniform grid \(\Lambda\) using the quasi-interpolants in \(\mathcal{U}'\). We record the optimizers \((\dummyopt(g))_{g \in \Lm}\) obtained from solving the corresponding NLP \eqref{eq:col_NLP} and the minimized cost \(\refcost_0\). Generate the trajectory
\begin{align}\label{eq:unif_quasi_traj_mf}
 \lcrc{0}{T} \ni t \mapsto \ulam(t) \Let \frac{1}{\sqrt{\pi \Dd}} \sum_{g \in \Lambda }\dummyopt(g)\;\epower{-|t-g|^2/\Dd h^2} \in \Rbb^{\dimcon}.
\end{align}

\begin{remark}\label{rem:approx_control_as_proxy}
Note that the optimal control trajectory \(\cont\as(\cdot)\) is generally unavailable in an analytical form. However, for a given \(\eps>0\), we have shown via Theorem \eqref{thrm:optimal estimate} that the approximating trajectory \(\ulam(\cdot)\) is \(\eps\)-close to \(\cont\as(\cdot)\). To this end, we use \(\ulam(\cdot)\) as a proxy to \(\cont\as(\cdot)\) in the change point localization algorithm and proceed to refine the grid at the localized change points.
\end{remark}
We employ the localization Algorithm \ref{alg:detection_algo} on \(t \mapsto \ulam(t)\) in \eqref{eq:unif_quasi_traj_mf} and obtain the pinpointed temporal location of wavelet modulus maxima at a chosen fine scale \(b_0\). Recall that, \(\tau\) denotes the set of time instances at which these wavelet maxima occur i.e., 
\begin{align}\label{eq:wavelt_max_loc}
    \tau \Let \aset[\bigg]{t\as\in \lcrc{0}{\horizon} \suchthat \frac{\partial  (W \ulam)}{\partial t}\bigg \rvert_{t=t\as} (t,b_0) = 0},
\end{align}
and the temporal locations \(t\as \in \tau\) corresponds to the change points in \(\ulam(\cdot)\). Let \(\width>0\), which we call the refinement width parameter, and define
\begin{align}\label{eq:ref_intervals}
    \rfi \Let \aset[\bigg]{(\interval_i)_{i \in I} \suchthat \interval_i \cap \bigcup_{t\as \in \tau} \Ball\left[t\as,\frac{\width}{2}\right] \neq \varnothing }.
\end{align}
The set \(\rfi\) contains prospective intervals \(\interval_i\) having nonempty intersection with the interval \(\Ball\left[t\as,\frac{\width}{2}\right]\) when \(t\as\) varies over the set \(\tau\), and these \(\interval_i\) intervals will subsequently be refined. For us, the refinement intervals \(\interval_i \in \rfi\) for each \(i \in I\) will remain fixed for the iterative refinement process. 

Let us describe the first iteration in detail. Let \(\lambda \in \N\) be a refinement parameter and we insert \(\lambda\) number of uniformly spaced grid points in each of the intervals \(\interval_i \in \rfi\). Let \(Y \Let \aset[]{i \in I \suchthat \interval_i \in \rfi}\) be an index set. Construct the piecewise uniform grid
\begin{align}\label{eq:PW_grid}
    \Grid_1
    & \Let \aset[\big]{s_i \suchthat i \in I } \cup \bigcup_{i\in Y} \aset[\bigg]{s_{i-1}+\frac{h}{\lambda+1},\ldots,s_{i-1}+\frac{\lambda h}{\lambda+1}}\nn \\
    &= \Lm \cup \bigcup_{i\in Y} \aset[\bigg]{s_{i-1}+\frac{h}{\lambda+1},\ldots,s_{i-1}+\frac{\lambda h}{\lambda+1}}.
\end{align}
Next, we solve the OCP \eqref{eq:OCP} by transcribing it as per the multiple shooting algorithm established in \S\ref{subsec:shooting} via parametrizing the control trajectory on \(\Grid_1\) using the quasi-interpolant \eqref{eq:Mhu_piecewise_grid}, and the corresponding minimized cost \(\refcost_1\) is recorded. For the second iteration, the above process is repeated by adding additional grid points within refinement intervals \(\interval_i \in \rfi\) at each iteration. We employ a linear increment in grid size over iterations, by inserting  \(\lambda j\) additional grid points in \(\interval_i \in \rfi\) at the \(j\)th iteration. More precisely, we define the indexed grid \(\Grid_j\) to denote the subsequent refined grid at the \(j\)th iteration by
\begin{align}\label{eq:PW_grid_iterative}
    \Grid_j \Let \Lm \cup \bigcup_{i\in Y} \aset[\bigg]{s_{i-1}+\frac{h}{(\lambda j)+1},\ldots,s_{i-1}+\frac{(\lambda j) h}{(\lambda j)+1}}
\end{align}
(note that, in particular, \(\Grid_0 = \Lm\)), and we denote the minimized cost by \(\refcost_j\). 
\begin{figure}[htpb]
\includegraphics[scale=0.55]{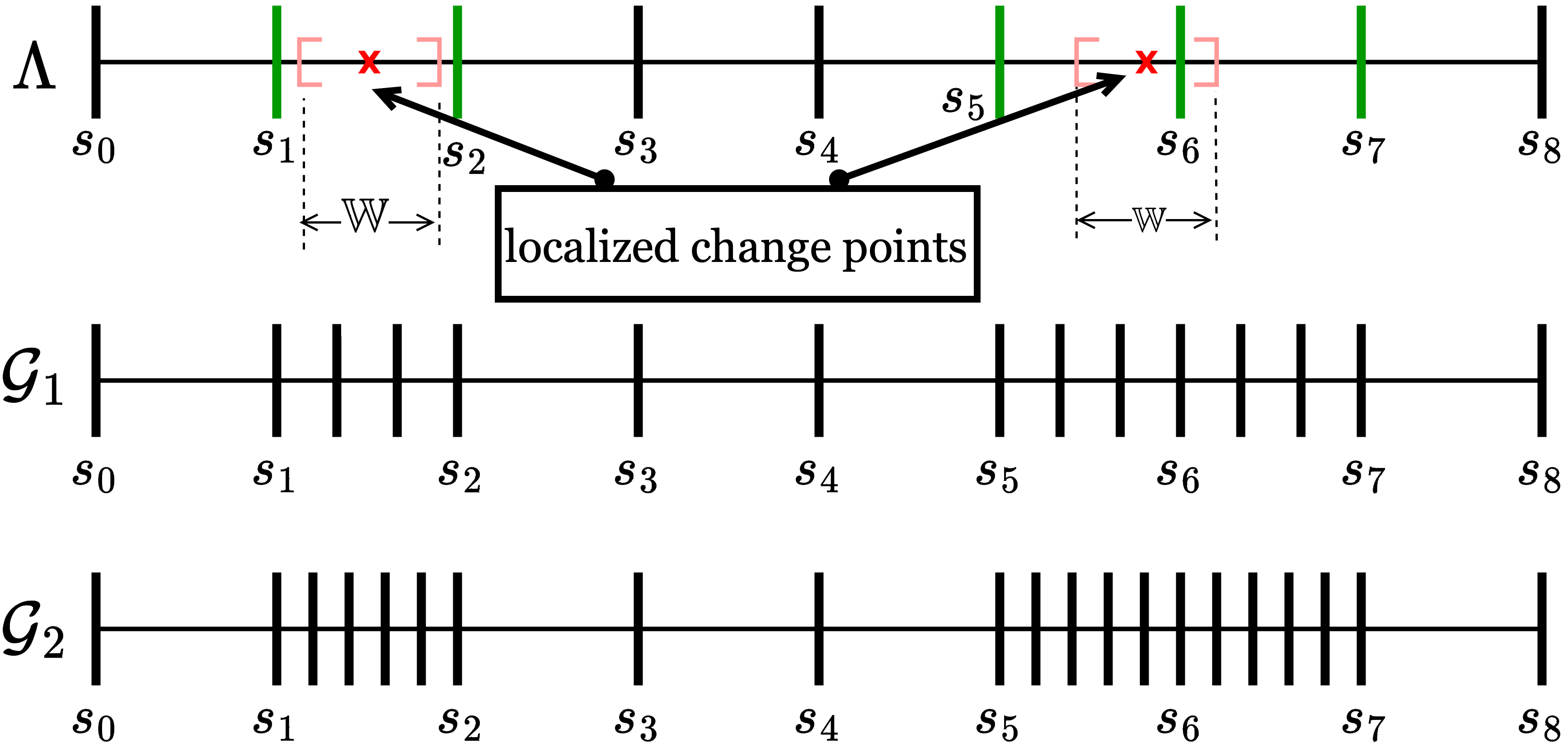}
\caption{Example of mesh refinement procedure for two change point localization is shown for two refinement iterations with \(\lambda=2\) where the change points are marked in red \(\textcolor{red}{\times}\). As usual, \(\Lm\) denotes an initial uniform grid (top) and the refinement intervals are represented by green lines. \(\Grid_1,\Grid_2\) are the refined grids (bottom) after the first and second iterations respectively. }\label{fig:refinement_procedure}
\end{figure}
We use the recorded minimized cost \(\refcost_j\) at each iteration as a relative indicator of the quality of the solution. To this end, let \(\widehat{K} \in \N\) denote the maximum number of refinement iterations. The absolute of the relative change in cost after refining the grid at the \(j\)th iteration is denoted as 
\begin{align}\label{eq:cost_decrease}
\reftol_j \Let \abs{\frac{\refcost_{j-1}-\refcost_{j}}{\refcost_{j-1}}}\quad \text{for } j=1,\ldots,\widehat{K}
\end{align}
The refinement iterations are repeated until the cost saturates, i.e., when \(\reftol_j\) is smaller than a predefined termination tolerance \(\bar{\reftol}\). A pictorial view of the refinement process is shown in Figure \ref{fig:refinement_procedure}.
\begin{remark}\label{rem:refinement_remarks}
The mesh refinement parameters, namely \(\lambda, \width,\bar{\reftol}\), decide the degree of refinement and can be tuned quite easily based on the OCP at hand. To wit:
\begin{enumerate}[label=(\alph*), leftmargin=*, widest=b, align=left]

\item The refinement parameter \(\lambda\) decides the local grid spacing \(hh_m\) of the piecewise uniform grid introduced in Theorem \eqref{thrm:app-app-piecewise_grid_formula}, i.e., for \(j=1\), we have   \(h_m=1/(\lambda+1)\), and  choosing \(\lambda>1\) satisfies the desired requirement \(h_m<1/2\) for \(j=1,\ldots,\widehat{K}\). Moreover, by construction \(Z(h_m) = \aset[\big]{0,1,\ldots,\tfrac{1}{h_m}-1}\).
 
\item Depending on the location of a change point relative to the grid points in \(\Lambda\),  approximately \(\lceil \width/h \rceil\) partitions surrounding the change point are refined further. This ensures sufficient mesh refinement for the initial uniform grid \(\Lambda\). The minimal error in localizing change points caused due to employing \(\ulam(\cdot)\) as a proxy for \(\cont\as(\cdot)\) in the change point localization algorithm (Remark \eqref{rem:approx_control_as_proxy}) can be well accommodated by increasing the refinement width parameter \(\width\).
\end{enumerate}
\end{remark}
The DMS technique \(\quito\) \(\vertwo\), as a whole, with localization and refine modules is given in Algorithm \ref{alg:refinement_algo}.
\begin{algorithm2e}[htpb]
\DontPrintSemicolon
\SetKwInOut{ini}{Initialize}
\SetKwInOut{giv}{Data}
\SetKwInOut{nlp}{NLP\((\cdot)\)}
\SetKwInOut{out}{Output}
\giv{\(\rcost(\cdot,\cdot), \fcost(\cdot), f(\cdot), G(\cdot), h_j(\cdot), \admcont, \param, r_F(\cdot),b_0\)}
\ini{\(\circ\) A finite uniform grid \(\Lm\) and a corresponding \(h\), \(\genfn(\cdot) \in \mathbb{B}\) and \(\Dd \in \loro{0}{+\infty}\)\\
\(\circ\) Refinement parameters, \(\lambda, \width\) \\
\(\circ\)  Cost termination tolerance \(\bar{\reftol}\), maximum number of iterations \(\widehat{K}\)}
\nlp{\(\circ\) Solve the nonlinear program \eqref{eq:col_NLP} on a given grid \(\ol{\mathcal{G}}\)
\\
\(\circ\) Obtain corresponding control trajectory \(\widehat{u}_{\ol{\mathcal{G}},\Dd}\), and \(\refcost\).}

Compute \(\widehat{\cont}_{\Lm,\Dd}(\cdot)\) and record the minimized cost \(\refcost_0\) solving \textbf{NLP}\((\Lm)\) \\
Fix \(\N \ni \bar{M}>1\) and define \(N'\Let N+\gridindex\) where \(\gridindex \in \aset[]{1,\ldots,\bar{M}}\) is a random positive integer, and solve \textbf{NLP}\((\Lm')\) on a uniform grid \(\Lm'\) with number of steps \(N'\).\\
Employ Algorithm \ref{alg:detection_algo} on \(\wh{\cont}_{\Lm,\Dd}(\cdot)\), and compute prospective refinement intervals  \(\interval_i\in \rfi\) as defined in \eqref{eq:ref_intervals}\\

\For{ \(j=1,\ldots,\widehat{K}\)}
    {
    Add \(\lambda j\)  uniformly spaced grid points within each refinement interval and obtain piecewise uniform grid  \(\Grid_j\) as defined in \eqref{eq:PW_grid_iterative} \\ 
    Compute \(\wh{\cont}_{\Grid_j,\Dd}(t)\) and record the minimized cost \(\refcost_j\) solving \textbf{NLP}\((\Grid_j)\)   \\
    Evaluate \(\reftol_j \Let \abs{\frac{\refcost_{j-1}-\refcost_{j}}{\refcost_{j-1}}}\) as defined in \eqref{eq:cost_decrease}

    \If{ \(\reftol_j<\bar{\reftol}\)}
    {
    \textbf{Terminate}    
    }    
    }
  \out{Approximate control trajectories after mesh refinement}  
\caption{\(\quito\) \(\vertwo\): transcription, localization, and refinement algorithm}
\label{alg:refinement_algo}
\end{algorithm2e}

\begin{remark}
Note that:
\begin{itemize}[leftmargin=*]
    \item Step 2 of Algorithm \ref{alg:refinement_algo} introduces some degree of randomization by incorporating a step where, for a given OCP, our transcription technique (via the baseline-\(\quito\)) is performed for several random step sizes. This step can be bypassed for most problems but it is needed to avoid classes of degenerate problems, where due to the nature of the path constraints,  a naive choice of \(h\) may be insufficient to accurately capture certain problem specifications. For example, consider the following OCP:
\begin{equation}\label{eq:some_OCP}
\begin{aligned}
& \min_{\cont(\cdot)}	&& \int_0^1 \st(t)\,\odif{t}\\
&  \sbjto		&&  \begin{cases}
\dot{x}(t) = u(t),\,\st(0)= \bar{x} = 0,\\
\st(t) \ge \sin (10 \pi t) \quad \text{for a.e. }t \in \lcrc{0}{1},\\
\st(t) \ge 0 \quad \text{for a.e }t \in \lcrc{0}{1}, \\
u(t) \in \lcrc{-1}{1} \quad \text{for a.e }t \in \lcrc{0}{1}. \nn
\end{cases}
\end{aligned}
\end{equation}
where with an initial \(h = \tfrac{1}{10}\) , the approximate solution will be zero everywhere with objective value \(0\), while a lower bound for the objective of the OCP is \(\int_0^1 \max \left\{ 0, \sin (10 \pi t) \right\} \,\dd t = \tfrac{1}{\pi} >0.\)\footnote{We thank the anonymous reviewer who pointed out this example to us.}

\item Moreover, note that as with every other numerical technique, \(\quito\,\vertwo\) is not a general purpose numerical tool that can cater to \emph{every possible} OCP. For instance, there are classes of problems where the problem data is smooth but \(u\as(\cdot)\) is as discontinuous as one likes \cite[Chapter 23]{ref:clarke_ocpbook}, and consequently, it is easy to engineer problems that lie beyond the ambit of each technique. Here are two such cases:
\begin{enumerate}
    \item Let \(u\as:\lcrc{0}{1} \lra \Rbb^d\) be any measurable function such that \(u\as(t) \in \Ball(0,1)\) for a.e. \(t \in \lcrc{0}{1}\). Define \(x\as(t) \Let \int_0^t u\as(\tau)\,\dd{\tau}\). Note that \(x\as(\cdot)\) is Lipschitz and its graph \(G_x\) is a closed subset of \(\lcrc{0}{1} \times \Rbb^d\), and consequently there exists a nonnegetive map \(\rcost \in \mathcal{C}^{\infty}\bigl(\Rbb^{d+1};\Rbb\bigr)\) such that \( G_x = \aset[]{(s,\xi) \in \lcrc{0}{1} \times \Rbb^d \suchthat \rcost(s,\xi) = 0}\). Now consider the OCP: 
\begin{equation}\label{eq:another_pathological_OCP}
\begin{aligned}
& \min_{\cont(\cdot)}	&& \int_0^1 \rcost(t,x(t))\,\dd{t}\\
&  \sbjto		&&  \begin{cases}
\dot{x}(t) = u(t),\,\st(0)= \bar{x} = 0,\\
u(t) \in \Ball(0,1) \quad \text{for a.e. }t \in \lcrc{0}{1}. \nn
\end{cases}
\end{aligned}
\end{equation}
Clearly, \(\bigl(x\as(\cdot),u\as(\cdot)\bigr)\) is a unique optimal state-action pair that gives a zero cost \cite[\S 23.3]{ref:clarke_ocpbook}, but \(u\as(\cdot)\) is merely measurable and can, therefore, be as discontinuous as one wishes. This is one of the reasons behind many of the regularity assumptions, such as \ref{OCPdata_new_1}--\ref{OCPdata_new_4} (specially \ref{OCPdata_new_2}, i.e., the instantaneous cost should be a strongly convex function of \(u\)), so that certain smoothness conditions (for example, Lipschitz continuity in our case) on \(u\as(\cdot)\) are ensured.

\item A second example is the Fuller’s problem \cite[Chapter 2]{ref:zelikin2012theory}. Despite the smoothness of the problem data in the underlying optimal control problem (OCP), the optimal control trajectory, which is not even singular, demonstrates a bang-bang solution with an infinite number of switches. Naturally, this problem lies beyond the ambit of \emph{any} numerical solver.
\end{enumerate}
\end{itemize}
\end{remark}

\section{Numerical experiments and discussion}\label{sec:num_exp}
We record a library of numerical examples ranging from low through high dimensions to showcase the numerical performance of \(\quito\) \(\vertwo\) and its localization and refinement components. Most of the examples we consider herein are well-known benchmark numerically hard problems because the underlying optimal control problems are \emph{singular}; see the examples in \S\ref{num:bressan} and \S\ref{num:catalystmixOCP}. We compared our results with the state-of-the-art pseudospectral direct collocation techniques and the integrated residue minimization technique \cite{ref:YN:ECK:residue} via the open source ICLOCS2 \cite{ref:nie2018iclocs2} solver, and with the multiple shooting technique via the ACADO toolkit \cite{ref:software_acado}. We demonstrate the ability of our technique to accurately localize change points in the control profile and to furnish accurate solutions through targeted mesh refinement as compared to the three preceding tools. Equipped with the localization and the mesh refinement modules, we obtain superior performance in terms of the quality of the solution (as compared to the state-of-the-art) while mildly sacrificing the computation time (in certain cases). We have also encoded the Algorithm \ref{alg:refinement_algo} as a software package, and a brief description and functionalities of the software are given in \S\ref{sec:software}.

The terminologies and data in Table \ref{tab:data_for_numerics} are used in the upcoming examples unless otherwise specified.
\begin{table}[htbp]
\begin{tblr}{l c}
\hline[2pt]
\SetRow{azure9}
Terminology/data & Meaning  \\ 
\hline[2pt]
\(\Ninit\) (same as the \(N\) given in \eqref{eq:uniform_partition}) &  (No. of grid points in  initial mesh) \(- 1\)\\

 \(\Nfin\) & No. of grid points in final mesh (after refinement) \\ 

\(\genfn(\cdot)\) (generating function) & Gaussian \\ 

\(\Dd\) (shape parameter) & 2 \\

\(\lambda\) (refinement parameter) & 4 \\ 

\(\width\) (refinement width parameter) & \(\frac{\tfin}{21}\) \\ 

\(\bar{\reftol}\) (tolerance) & \(10^{-4}\) \\
\hline[2pt]
\end{tblr}
\centering
\vspace{2.5mm}
\caption{Common data and terminologies}
\label{tab:data_for_numerics}
\end{table}
Note that the parameters in Table \ref{tab:data_for_numerics} can be tuned, further to cater to a specific problem for better performance. All examples were solved using MATLAB version R2021a with CasADi interface \cite{ref:CASADI_andersson2019} using the NLP solver IPOPT \cite{ref:IPOPT_Biegler}. For the wavelet transform employed in the localization procedure, we use the Ricker/Mexican hat wavelet at scales of order \(10^{-4}\), which is the negative normalized second derivative function of the Gaussian density. For locating wavelet maxima, we used MATLAB's \texttt{findpeaks} functionality. For the ICLOCS2 direct collocation method, we used Euler discretization and active mesh refinement for a fair comparison. All computations were performed on Windows \(2.00\)GHz AMD Ryzen-\(7\) \(4700\)-U processor with \(16\)GB RAM. The central processing unit (CPU) times for all the solvers reported here are the total execution time, including the NLP time required to solve the OCP over multiple refinement iterations. 

A list of the problems that we solved via \(\quito\) \(\vertwo\) (Algorithm \ref{alg:refinement_algo}) is given in Table \ref{tab:list_of_problems}. The computation time and performance of \(\quito\) \(\vertwo\) across these examples are collected separately in Table \ref{tab:computation_list} in \S\ref{subsec:CPUtimediscussion}.
\begin{table}[h!]
\begin{tblr}{l c c c}
\hline[2pt]
\SetRow{azure9}
Problem & Dimension & Constraints & Type \\ 
\hline[2pt]
Example \eqref{num:bangbangOCP} &  \(1\) & control & bang-bang\\
Example \eqref{num:bressan} & \(3\) & control & singular\\ 
Example \eqref{num:catalystmixOCP}  & \(2\) & control & bang-singular-bang \\ 
Example \eqref{num:SIRIOCP} & \(3\) & state and control & bang-bang\\
Example \eqref{numexp:auv_pathpanning}  & \(18\) & state, control, path & normal\\ 
\hline[2pt]
\end{tblr}
\centering
\vspace{2.5mm}
\caption{A list of problems treated in \S\ref{sec:num_exp}.}
\label{tab:list_of_problems}
\end{table}
\subsection{Bang Bang problem}\label{num:bangbangOCP}
We start with a simple problem with one-dimensional dynamics and simple bounds on the control. Consider the OCP \cite{ref:ST:2011numOCP} :
\begin{equation}
\label{eq:bangbangmesh}
\begin{aligned}
& \minimize_{\cont(\cdot)}	&&  \frac{1}{2} \int_{0}^{1}  \left( \epower{-t} - 2t\right)x(t)\, \dd t \\
&  \sbjto		&&  \begin{cases}
\dot{\st}(t)= -tx(t) + \log(u(t) + t + 3) \,\,\text{for all }t \in \lcrc{0}{1},\\
\st(0) = 0,\, \st(1) = 0.8,\,|u(t)| \le 1,
\end{cases}
\end{aligned}
\end{equation}
featuring a bang-bang solution. We fixed a \emph{uniform grid} with size \(\Ninit=50\) and solved the problem \eqref{eq:bangbangmesh} using baseline-\(\quito\). Next, we employed Algorithm \ref{alg:detection_algo}, i.e., we performed the wavelet transform of the obtained control trajectory and located the wavelet modulus maxima. As shown in Figure \eqref{fig:BB_DMCS} (top left), we obtained wavelet modulus maxima at \(\tau = \aset[]{0.188,0.511}\) (also see the scalogram in Figure \ref{fig:BBscalocon} (left)). Subsequently, the mesh was refined around these specific locations as per Algorithm \eqref{alg:refinement_algo}.
\begin{figure}[htpb]
\centering
\includegraphics[scale=0.42]{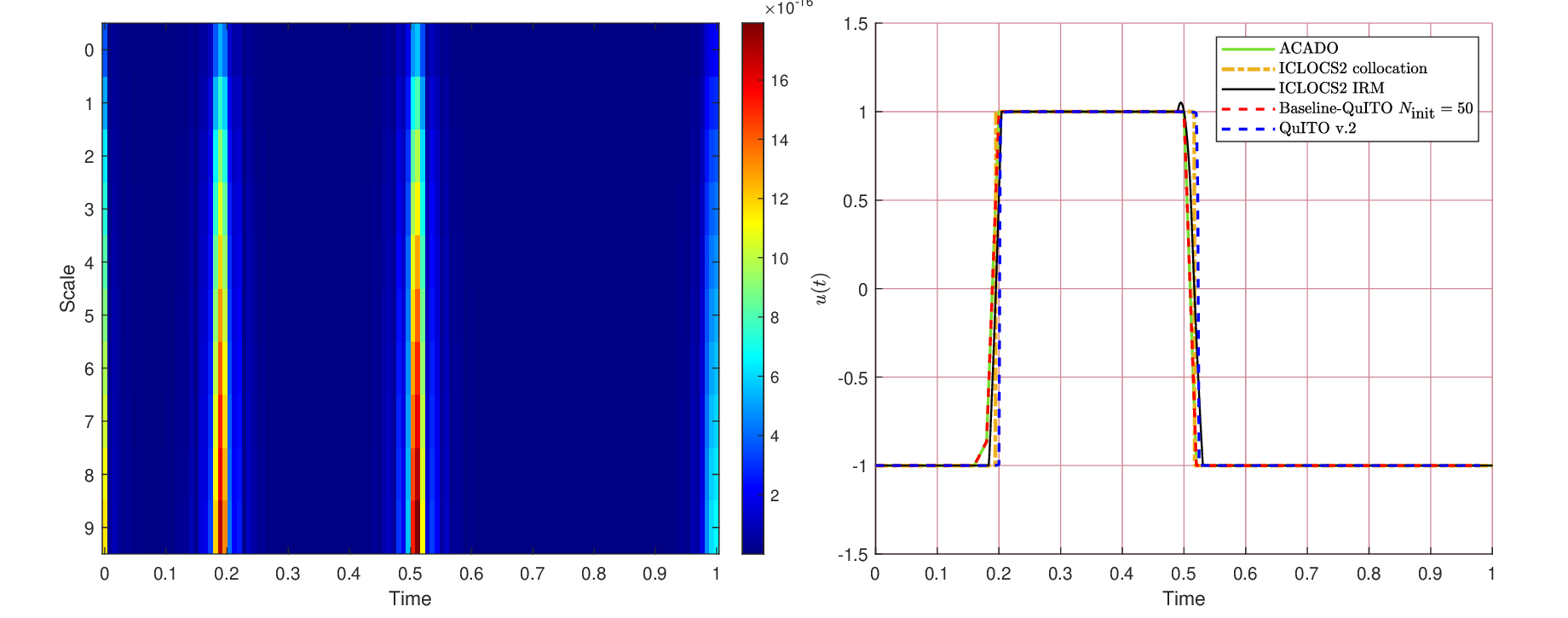}
\caption{The wavelet scalogram (left) of the control trajectory obtained from the baseline-\(\quito\) where the vertical axis captures the scale \(b\) of the wavelet. The control trajectories (right) obtained using \(\quito\) \(\vertwo\), ICLOCS2's collocation, and baseline-\(\quito\) for Example \ref{num:bangbangOCP}.}\label{fig:BBscalocon}
\end{figure}
The termination tolerance was set to \(\bar{\reftol}= 10^{-3}\). The refinement width parameter was set to \(\width= 1/21\) units and subsequently \( \lceil (\Ninit+1) /21 \rceil =3\) grid intervals surrounding the localized change points in the uniform grid were refined; this is shown in Figure \eqref{fig:BB_DMCS} (top right). We observed that the specified termination tolerance of \(\bar{\reftol}= 10^{-3}\) was met after the two iterations, and the algorithm terminated with final mesh size \(\Nfin = 99\). The corresponding minimized cost over the iterations and the final state trajectories are shown in Figure \eqref{fig:BB_DMCS} (bottom left and right). 

The control trajectories obtained using the initial uniformly spaced mesh via baseline-\(\quito\) (with \(50\) steps), the refined control trajectory from \(\quito\) \(\vertwo\), the control trajectories obtained via LGR collocation (using ICLOCS2 with \(50\) steps and with active mesh refinement), Integrated Residue Minimization (IRM) (using ICLOCLS2 with \(50\) steps), direct multiple shooting with piecewise constant parameterization of control (using ACADO with \(50\) steps) are depicted in Figure \eqref{fig:BBscalocon} (right). The corresponding optimization statistics are shown in Table \eqref{vdo:BB-cpu-time}.  We observe that the control trajectory after mesh refinement accurately captures the jump discontinuities and overlaps with the solutions obtained via other algorithms, validating the effectiveness of our mesh refinement algorithm. 
\begin{figure}[htpb]
\includegraphics[scale=0.75]{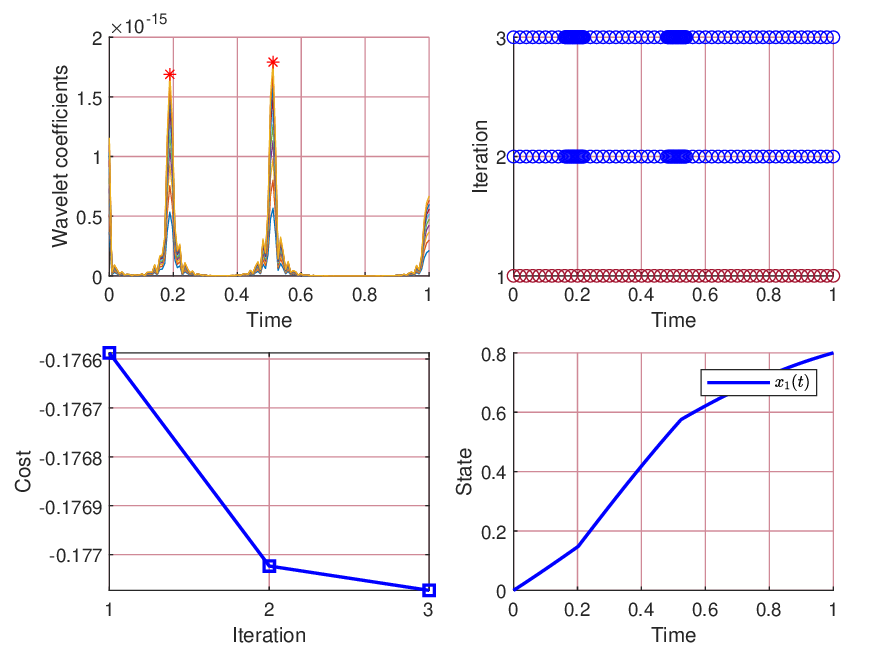}
\caption{ Wavelet coefficients for a set of scales of order \(10^{-3}\) (top left), mesh history during refinement (top right), cost variation (bottom left), and state trajectory after refinement (bottom right)  for Example \ref{num:bangbangOCP}.}\label{fig:BB_DMCS}
\end{figure}


\begin{table}[htpb]
\centering
\begin{tblr}{l c c }
	\hline[2pt]
 \SetRow{azure9}
Method & \(\Ninit\) & \(\Nfin\) &  Minimized cost  \\ 
\hline[2pt]
\(\quito\) \(\vertwo\) &  \(50\) & 99 &  \(-0.183\)  \\
ICLOCS2 collocation with mesh-refinement & \(50\) & 4186 &  \(-0.183\)   \\
ICLOCS2 Integrated Residue Method  & \(50\) & 50  & \(-0.183\)  \\
ACADO  & \(50\) & 50 & \(-0.183\)    \\
\hline[2pt]
\end{tblr}
\vspace{2.5mm}
\caption{A comparison of optimization statistics for Example \ref{num:bangbangOCP}.}
\label{vdo:BB-cpu-time}
\end{table}


\subsection{Bressan problem}\label{num:bressan}
Consider the following optimal control problem with time horizon \(\horizon=10\) reported in \cite{ref:BressPicc:MathCon}:
\begin{equation}
\label{eq:bressanOCP}
\begin{aligned}
& \minimize_{\cont(\cdot)}	&&  -\st_3(\horizon)\\
&  \sbjto		&&  \begin{cases}
\dot{\st_1}(t)= \cont(t),\,\dot{\st_2}(t)= -\st_1(t),\\
\dot{\st_3}(t)= \st_2(t) - \st_1(t)^2,\\
\st(0)= (0,0,0), \, |u(t)| \le 1, \,\horizon=10.
\end{cases}
\end{aligned}
\end{equation}
The problem \eqref{eq:bressanOCP} is known to be  singular, and the analytical expression of the optimal control trajectory is given by
\begin{equation}
\label{eq:bressanoptimalsol}
\begin{aligned}
&  \lcrc{0}{\horizon} \ni  t \mapsto \cont\as(t) \Let  \begin{cases}
-1 & \text{if} \,  \, 0\le t < \frac{\horizon}{3}, \\
\frac{1}{2} & \text{if}\, \, \frac{\horizon}{3} \le t < \horizon,
\end{cases}
\end{aligned}
\end{equation}
which is a bang-bang solution. We fixed a \emph{uniform grid} with size \(\Ninit=20\) and solved the problem \eqref{eq:bressanOCP} using baseline-\(\quito\).
\begin{figure}[htpb]
\includegraphics[scale=0.55]{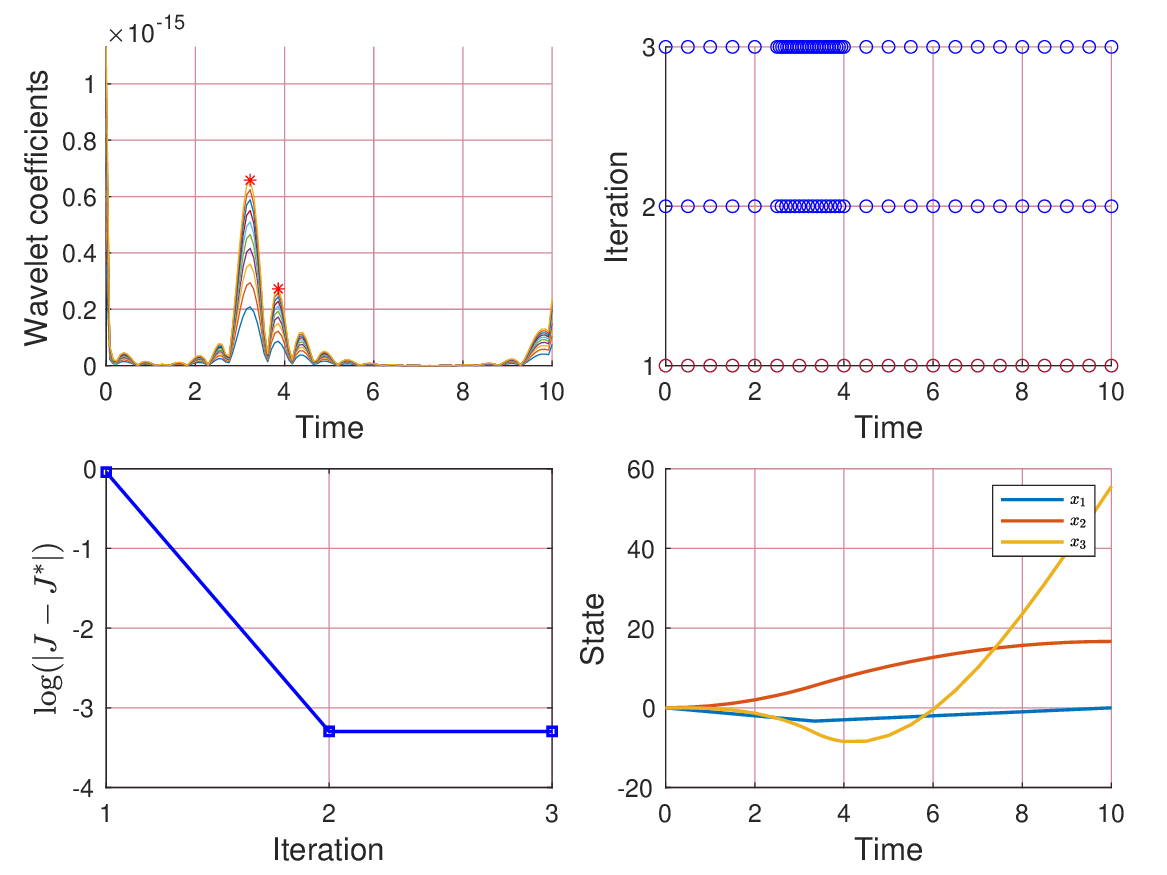}
\caption{ Wavelet coefficients for a set of scales of order \(10^{-3}\) (top left), mesh history during refinement (top right), cost variation (bottom left), and state trajectories after refinement(bottom right)  for Example \ref{num:bressan}.}\label{fig:BRSN_DMCS}
\end{figure}
We performed the wavelet transform of the resulting control trajectory and located the wavelet modulus maxima; see Figure \eqref{fig:BRSN_DMCS} (top left). The modulus maxima were located at \(\tau = \aset[]{3.228,3.858}\), and hence the mesh was refined around these specific locations as per Algorithm \eqref{alg:refinement_algo} with the width parameter \(\width= 1/21 \) units. Moreover, three grid intervals surrounding the localized change points in the uniform grid were refined; this is shown in Figure \eqref{fig:BRSN_DMCS} (top right). We observed that the specified termination tolerance of \(\bar{\reftol}= 10^{-4}\) was met after the two iterations, and as a result, the algorithm terminated with final mesh size \(\Nfin = 45\).

The corresponding minimized cost over the iterations and the final state trajectories are shown in Figure \eqref{fig:BRSN_DMCS} (bottom). The control trajectory obtained using the initial uniformly spaced mesh via baseline-\(\quito\) (with \(\Ninit=20\)), the control trajectory from \(\quito\) \(\vertwo\), the control trajectories obtained via ICLOCS2 with LGR collocation and Integrated Residue Minimization (IRM) are depicted in Figure \eqref{fig:BRSNcontrol}. The corresponding optimization statistics are shown in Table \eqref{vdo:BRSN-cpu-time}.

We observed that \(\quito\) \(\vertwo\) performs better (with mere \(\Ninit=20\) steps along with refinement) than the baseline-\(\quito\) solution obtained via employing \(\Ninit=200\) steps (i.e., a much finer grid); see Figure \ref{fig:BRSNcontrol}, the red trajectory. We also observed that \(\quito\) \(\vertwo\) performs better than the integrated residue minimization and LGR collocation algorithms (both implemented via ICLOCS2), and the multiple shooting technique using piecewise constant parameterization of control (implemented via ACADO) in terms of closeness with the analytical optimal control trajectory. Notice the \emph{ringing phenomena} in the control trajectory corresponding to the LGR collocation module, which is a common numerical defect that direct collocation techniques suffer from in many singular optimal control problems (the same defect will appear in the next two problems as well; see \S\ref{num:catalystmixOCP} and \S\ref{num:SIRIOCP}). 
\begin{figure}[htpb]
\centering
\includegraphics[scale=0.6]{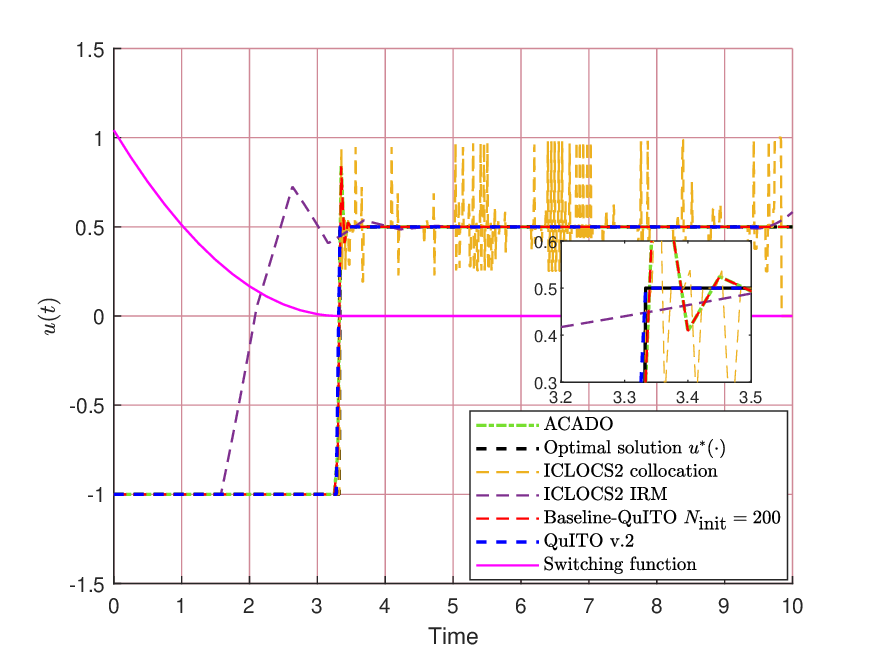}
\caption{Control trajectories obtained using \(\quito\) \(\vertwo\), ICLOCS2's collocation and IRM, and baseline-\(\quito\) for Example \ref{num:bressan}.}\label{fig:BRSNcontrol}
\end{figure}
\begin{figure}[h!]
\centering
\includegraphics[scale=0.6]{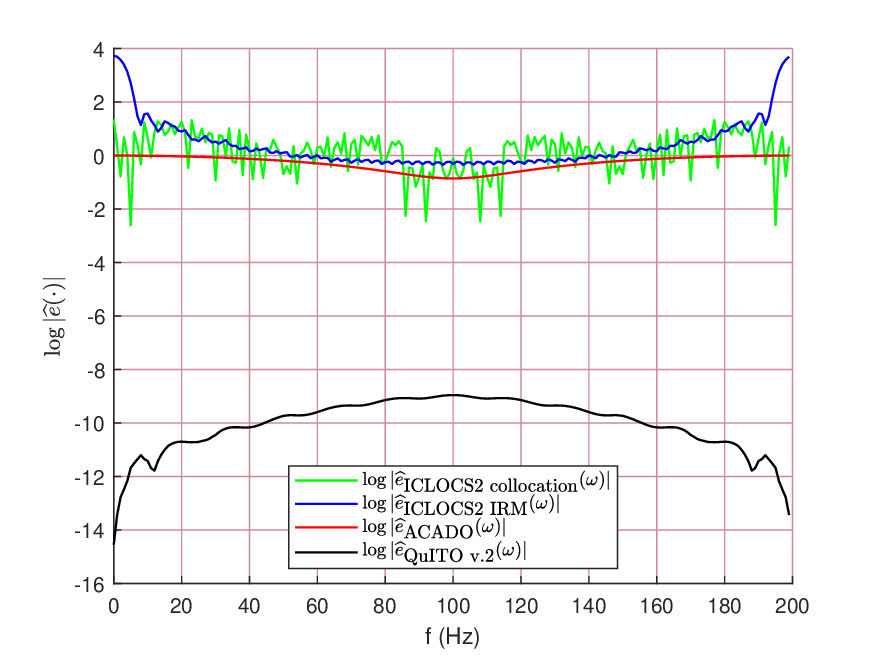}
\caption{DFT plot of the control error trajectories in log scale for Example \ref{num:bressan}.}\label{fig:BRSNerror&dft}
\end{figure}
In contrast, \(\quito\) \(\vertwo\) shows no such unwanted artifacts. The DFT (Discrete Fourier Transform) plots of the control error trajectories obtained by employing \(\quito\) \(\vertwo\), LGR collocation (with active mesh refinement), IRM, and ACADO are given in Figure \ref{fig:BRSNerror&dft}, where \({e}_{\texttt{method}}(\cdot)\) denotes the difference between the analytical control \(u\as(\cdot)\) and the numerical control obtained via choosing any of the \texttt{method} (\(\quito\) \(\vertwo\), IRM, collocation, or multiple shooting via ACADO), and the \(\hat{\cdot}\) denotes the DFT operator. Analyzing the DFT of the error trajectory between \(u\as(\cdot)\) and \(\ulam(\cdot)\) provides a comprehensive frequency-domain characterization, revealing irregularities appearing in multiple regions over the time horizon. In singular optimal control problems, these irregularities manifest as high-frequency jitter in the DFT of numerical control trajectories generated by collocation techniques. The \(\ell^2\)-norms of the DFT of the errors, i.e., \(\norm{\widehat{e}_{\quito}(\cdot)}_{\ell^2}\), \(\norm{\widehat{e}_{\text{collocation}}(\cdot)}_{\ell^2}\), \(\norm{\widehat{e}_{\text{IRM}}(\cdot)}_{\ell^2}\), and \(\norm{\widehat{e}_{\text{ACADO}}(\cdot)}_{\ell^2}\) respectively, were found to be \(0.001\), \(23.7391\), \(107.3\), and \(11.244\) respectively, via standard MATLAB command \texttt{fft}\footnote{See the URL \url{https://in.mathworks.com/help/matlab/ref/fft.html}. These numbers give a quantitative measure of the superior performance of \(\quito\) \(\vertwo\) in this benchmark example.} and \texttt{norm}.
\begin{table}[htpb]
\centering
\begin{tblr}{l c c}
\hline[2pt]
 \SetRow{azure9}
Method & \(\Ninit\) & \(\Nfin\) & Minimized cost \\ 
\hline[2pt]
\(\quito\) \(\vertwo\) &  \(20\) & \(45\) & \(-55.527\)  \\
 baseline-\(\quito\) &  \(200\) & \(200\) & \(-55.529\)  \\
 ICLOCS2 collocation with mesh refinement \\ \hspace{2mm} (optimal control rings) & \(20\)  & \(824\) & \(-55.494\)  \\
  ICLOCS2 Integrated Residue Method  & 20 & \(20\)  & \(-23.012\)   \\
  ACADO  & \(45\) & \(45\) & \(-55.355\)    \\
  ACADO  & \(200\) & \(200\) & \(-55.529\)    \\
 \hline[2pt]
\end{tblr}
\vspace{2.5mm}
\caption{A comparison of optimization statistics for Example \ref{num:bressan}.}
\label{vdo:BRSN-cpu-time}
\end{table}
To enhance the comparison, we also present control trajectories obtained using higher-order collocation techniques; see Figure \ref{fig:bressan_trapandhermite}.
\begin{figure}[!ht]
  \centering
  \begin{subfigure}[b]{0.49\linewidth}
    \includegraphics[width=\linewidth]{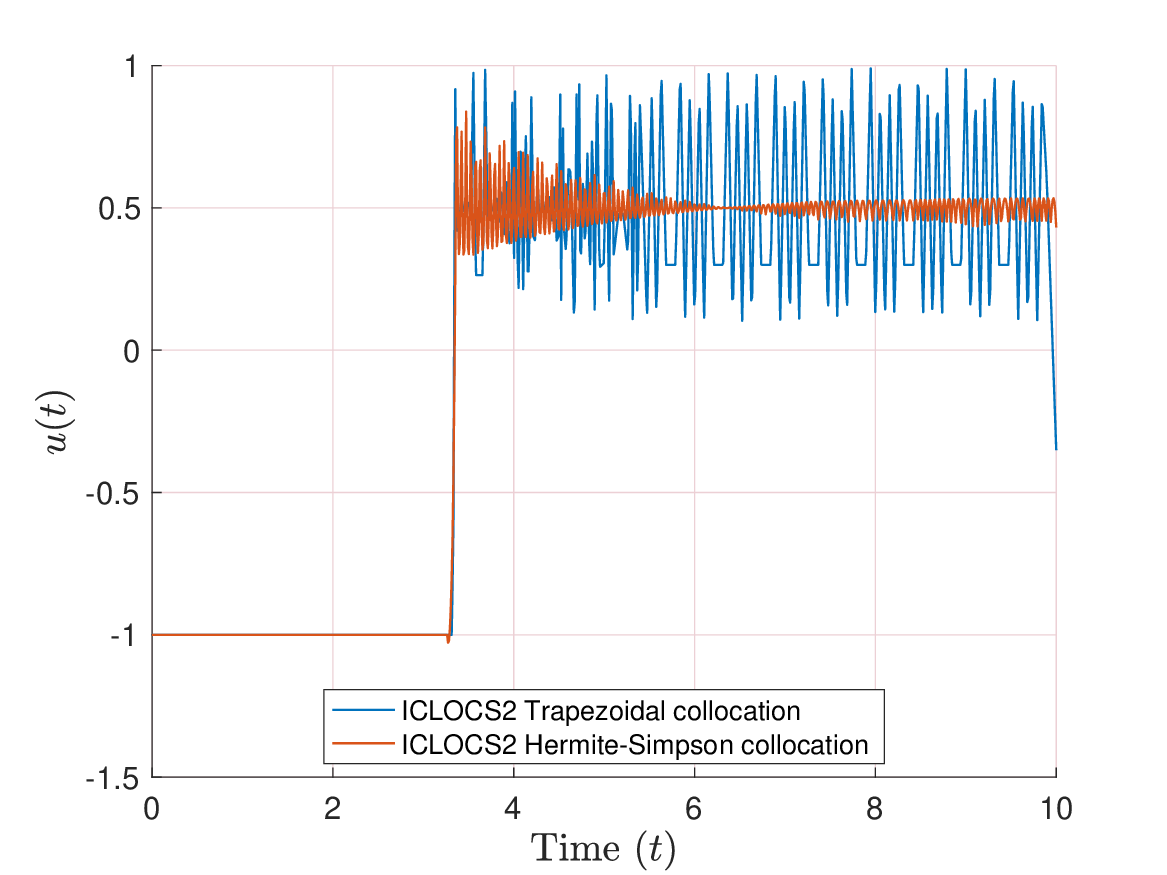}
  \end{subfigure}
  \begin{subfigure}[b]{0.49\linewidth}
    \includegraphics[width=\linewidth]{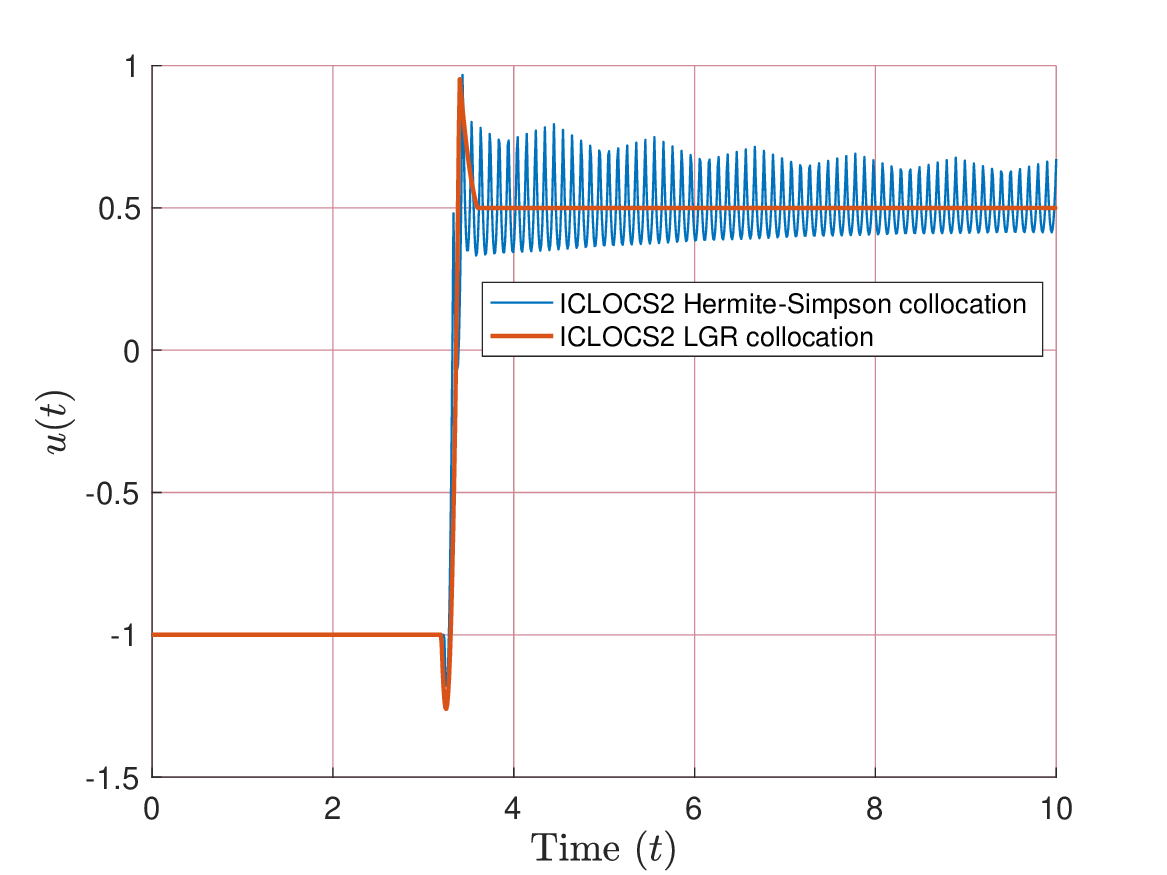}
  \end{subfigure}
 \caption{(Bressan problem) Numerical control trajectories obtained by Trapezoidal and Hermite-Simpson collocation using piecewise cubic Hermite interpolating polynomial for state and control representation (left-hand subfigure). Numerical control trajectories obtained by Hermite-Simpson and LGR collocation (via ICLOCS2) using piecewise cubic and Legendre polynomial respectively for state and control representation (right-hand subfigure). ICLOCS2 was employed for both the cases. 
}
 \label{fig:bressan_trapandhermite}
\end{figure}




\subsection{Catalyst Mixing Problem}\label{num:catalystmixOCP}
Consider the catalyst mixing optimal control problem in \cite{ref:NLP2010LorenzT, ref:CatalystmixingRJackson}, which is known to feature a bang-singular-bang control profile: 
\begin{equation}
\label{eq:catalystmix_OCP_con}
\begin{aligned}
& \minimize_{\cont(\cdot)}	&&   \st_1(T) + \st_2(T) - 1 \\
&  \sbjto		&&  \begin{cases}
\dot{\st}_1(t)= -\cont(t) \left( k_1 \st_1(t) - k_2 \st_2(t) \right),\\
\dot{\st}_2(t)= \cont(t) \left( k_1 \st_1(t) - k_2 \st_2(t) \right) - \left(1-\cont(t)\right)k_3 \st_2(t),\\
\st(0)^{\top} =  (1 \,\, 0)^{\top},\,\horizon \Let 4, \\
 0 \le u(t) \le 1, \\
k_1= 1,\, k_2=10, \text{and } k_3=1,
\end{cases}
\end{aligned}
\end{equation}
and the analytical expression of the optimal control trajectory is given by 
\begin{equation}
\label{eq:catalystoptimalsol}
\begin{aligned}
&  \lcrc{0}{\horizon} \ni  t \mapsto \cont\as(t) \Let  \begin{cases}
1 & \text{if} \,  \, 0\le t < 0.1363, \\
0.2271 & \text{if}\, \, 0.1363 \le t < 3.725, \\
0 & \text{if}\, \, 3.725 \le t \le \horizon.
\end{cases}
\end{aligned}
\end{equation}

We transcribed the OCP \eqref{eq:catalystmix_OCP_con} using baseline-\(\quito\) and solved it on a \emph{uniform grid} with size \(\Ninit=100\) followed by a wavelet localization of the obtained control trajectory.
\begin{figure}[htpb]
\includegraphics[scale=0.55]{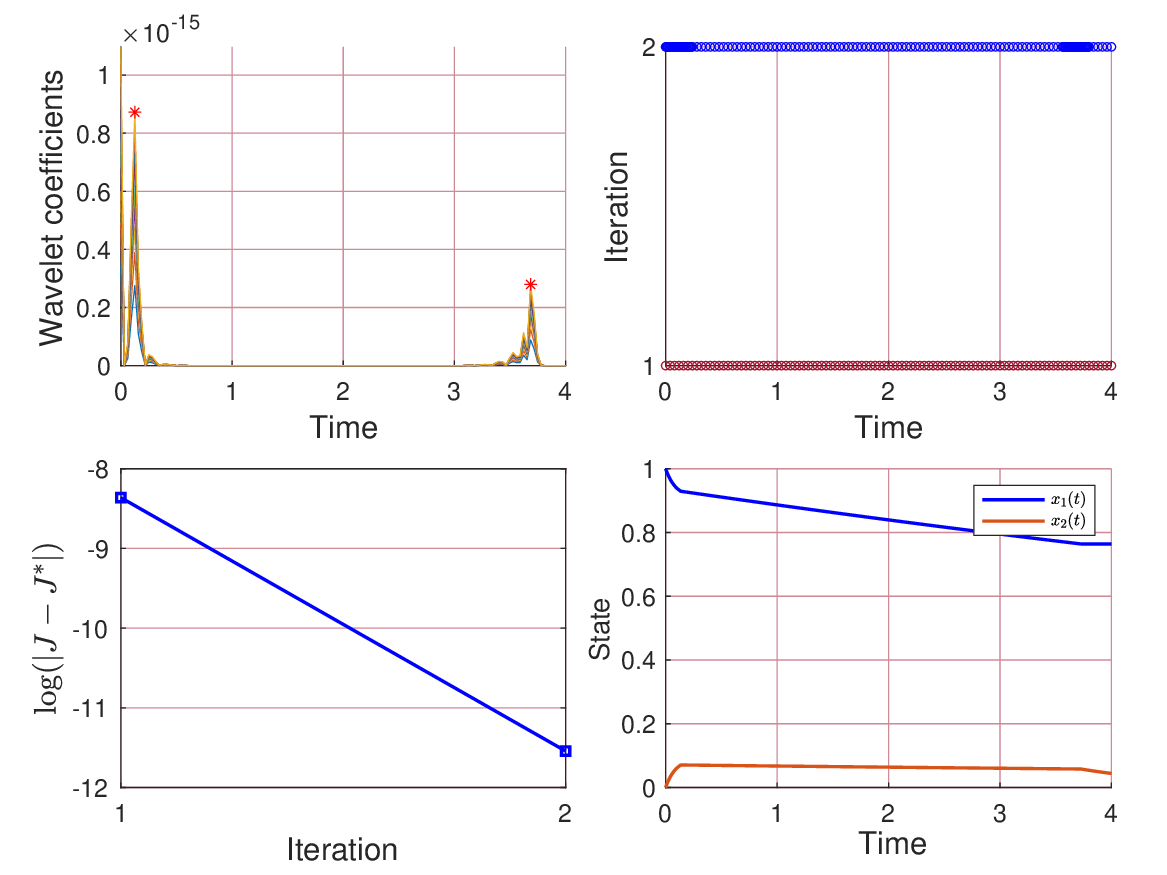}
\caption{ Wavelet coefficients for a set of scales of order \(10^{-3}\) (top left), mesh history during refinement (top right), cost variation (bottom left), and state trajectories after refinement (bottom right)  for Example \ref{num:catalystmixOCP}.}\label{fig:CM_DMCS}
\end{figure}
As shown in Figure \eqref{fig:CM_DMCS} (top left), we obtained wavelet modulus maxima at \(\tau = \aset[]{0.125,3.685}\), and subsequently the mesh was refined around these specific locations as per Algorithm \eqref{alg:refinement_algo}. We observed that the specified error tolerance was met after the first iteration and the algorithm terminated successfully. The corresponding minimized cost over the iterations and the state trajectories are shown in Figure \eqref{fig:CM_DMCS} (bottom left and right). The control trajectory on a fine uniform grid with \(\Ninit = 200\) steps obtained by baseline-\(\quito\), the refined trajectory via \(\quito\) \(\vertwo\), LGR collocation and IRM with \(100\) steps (both via ICLOCS2), and direct multiple shooting with piecewise constant parameterization of control trajectories with \(100\) steps (via ACADO) are depicted in Figure \eqref{fig:CMcontrol}. 
\begin{figure}[h!]
\centering
\includegraphics[scale=0.7]{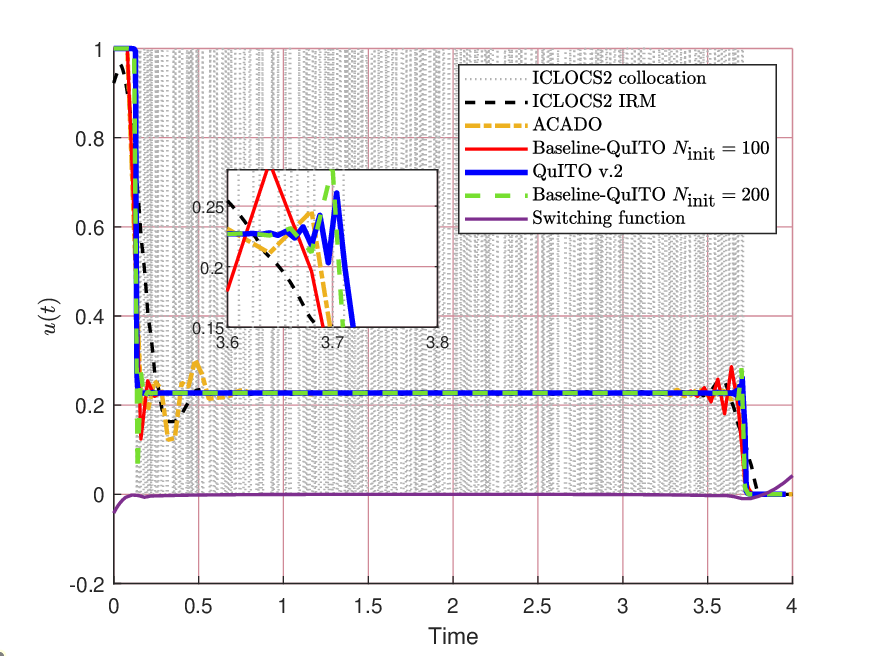}
\caption{Control trajectories obtained using \(\quito\) \(\vertwo\), ICLOCS2's collocation and IRM, and baseline-\(\quito\) for Example \ref{num:catalystmixOCP}.}\label{fig:CMcontrol}
\end{figure}

We observed that the oscillatory behavior near the jumps in the baseline-\(\quito\) solution (with \(\Ninit = 100\) steps) was significantly minimized by selective mesh refinement. Additionally, the computation time to obtain the refinement solution was lower and the solution behaved qualitatively better (in terms of minimizing the cost) compared to the baseline-\(\quito\) solution on a much denser uniform grid with \( \Ninit = 200\) steps. This shows the efficacy of our mesh refinement algorithm in comparison to baseline-\(\quito\)'s solution with a finer uniform grid. We also noted that the control trajectory obtained by employing LGR collocation/pseudospectral method via ICLOCS2 exhibits \emph{ringing behaviour} while no such undesirable behaviour was observed in \(\quito\) \(\vertwo\) solutions with or without refinement. Additionally, the IRM (employed via ICLOCS2) did not suffer from ringing, but it was computationally much more demanding and took several iterations to solve the problem. The solution obtained by ACADO also shows oscillations near the change points. The corresponding optimization statistics are recorded in Table \eqref{vdo:CM-cpu-time} and it can be seen that \(\quito\) \(\vertwo\) performed better (in terms of solution quality and computation time) in comparison to the baseline-\(\quito\), LGR collocation (solution quality), integrated residue minimization (solution quality and computation time), and ACADO (solution quality). This is demonstrated in Figure \ref{fig:CMerror&dft} where we depicted the DFT plots of control error trajectories (obtained via \(\quito\) \(\vertwo\), LGR collocation, IRM, and multiple shooting via ACADO) in log scale. The \(\ell^2\)-norms of the DFT of the errors, i.e., \(\norm{\widehat{e}_{\quito}(\cdot)}_{\ell^2}\), \(\norm{\widehat{e}_{\text{collocation}}(\cdot)}_{\ell^2}\), \(\norm{\widehat{e}_{\text{IRM}}(\cdot)}_{\ell^2}\), and \(\norm{\widehat{e}_{\text{ACADO}}(\cdot)}_{\ell^2}\) respectively, were found to be \(1.183\), \(80.7597\), \(11.1856\), and \(8.2158\) respectively. To enhance the comparison, we also present control trajectories obtained using higher-order collocation techniques; see Figure \ref{fig:catalyst_trapandhermite}.
\begin{figure}[htpb]
\centering
\includegraphics[scale=0.6]{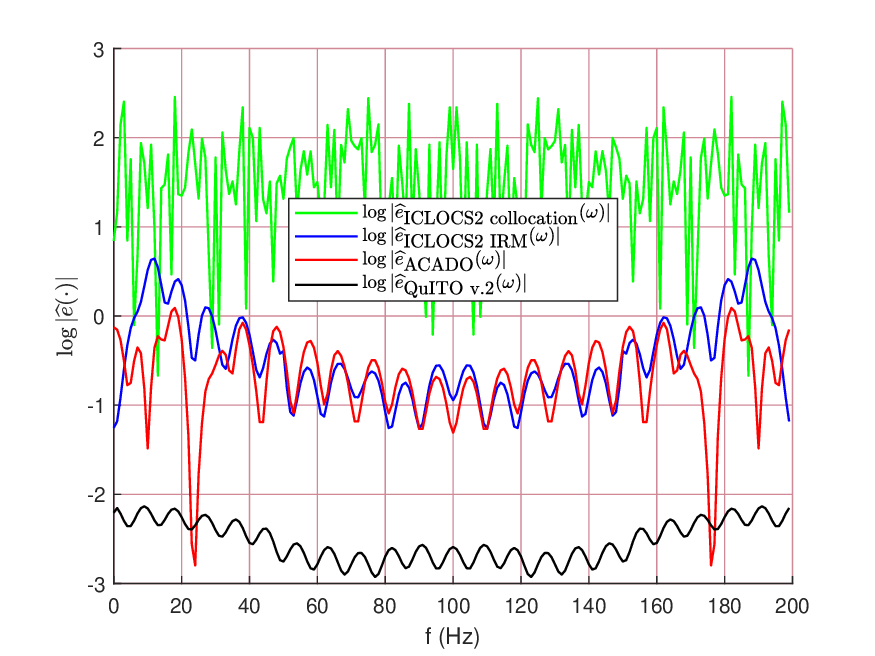}
\caption{DFT plot of the control error trajectories in log scale for Example \ref{num:catalystmixOCP}.}\label{fig:CMerror&dft}
\end{figure}
\begin{table}[h!]
\centering
\begin{tblr}{l c c}
\hline[2pt]
\SetRow{azure9}
Method & \(\Ninit\) & \(\Nfin\) & Minimized cost \\ 
\hline[2pt]
\(\quito\) \(\vertwo\) &  100 & 149 & \(- 0.191\)  \\
  baseline-\(\quito\) & 200 & 200 & \(- 0.191\) \\
 ICLOCS2 collocation with\\mesh refinement (optimal control rings) & 100  & 12197&  \(- 0.191\)  \\
 ICLOCS2 Integrated Residue Method  & 100 & 100 & \(- 0.191\)  \\
 ACADO  & \(100\) & 100  & \(-0.191\)  \\
 \hline[2pt]
\end{tblr}
\vspace{2.5mm}
\caption{A comparison of optimization statistics for Example \ref{num:catalystmixOCP}.}
\label{vdo:CM-cpu-time}
\end{table}



\begin{figure}[!ht]
  \centering
  \begin{subfigure}[b]{0.49\linewidth}
    \includegraphics[width=\linewidth]{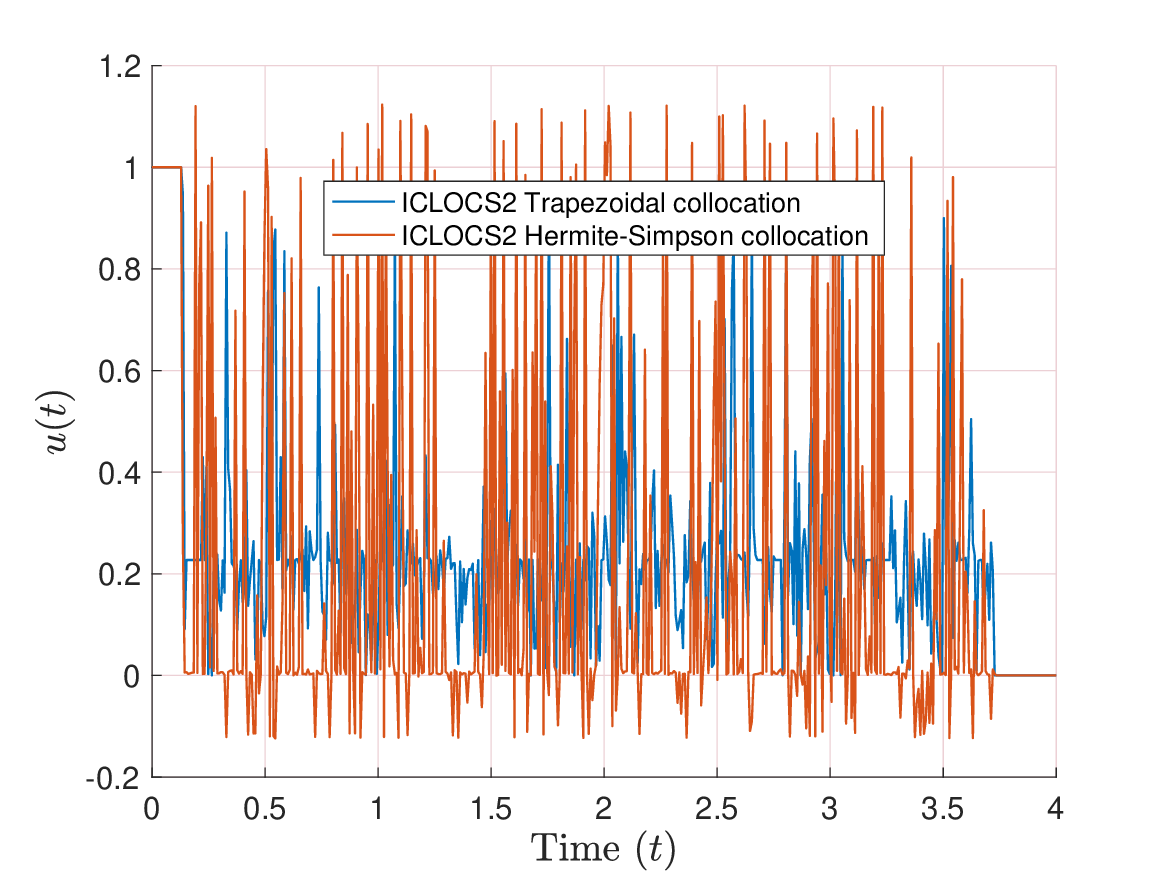}
  \end{subfigure}
  \begin{subfigure}[b]{0.49\linewidth}
    \includegraphics[width=\linewidth]{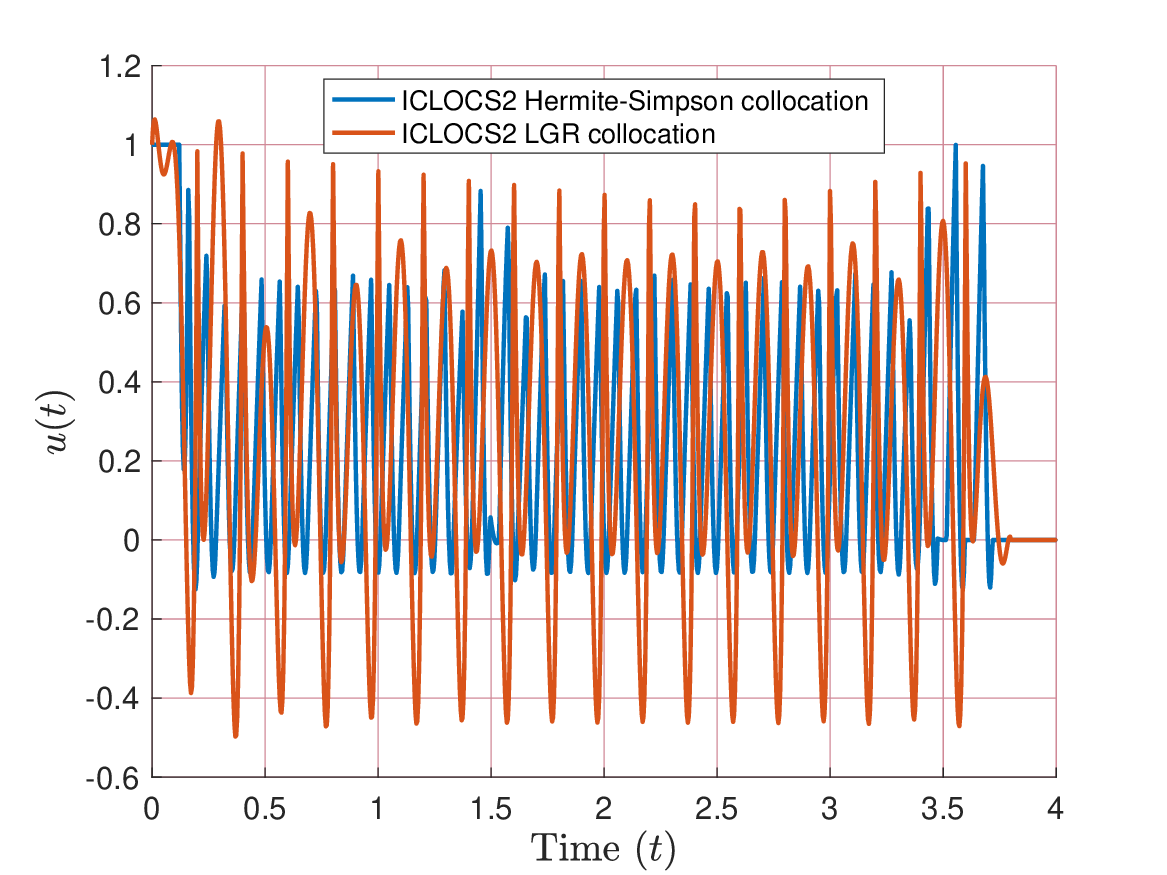}
  \end{subfigure}
 \caption{(Catalyst mixing problem) Numerical control trajectories obtained by Trapezoidal and Hermite-Simpson collocation using piecewise cubic Hermite interpolating polynomial for state and control representation (left-hand subfigure). Numerical control trajectories obtained by Hermite-Simpson and LGR collocation (via ICLOCS2) using piecewise cubic and Legendre polynomial respectively for state and control representation (right-hand subfigure). ICLOCS2 was employed for both the cases. 
}
 \label{fig:catalyst_trapandhermite}
\end{figure}


\subsection{SIRI epidemics}\label{num:SIRIOCP}
We considered the Susceptible-Infected-Recovered-Infected model (SIRI)\footnote{We sincerely thank Urmee Maitra and Ashish R. Hota from the Indian Institute of Technology Kharagpur, for providing us with the model and relevant problem data associated with this problem.} where an individual is characterized by three possible states --- susceptible, infected, or recovered. The quantities \(\beta \in \loro{0}{1}\) and \(\hat{\beta} \in \loro{0}{1}\) denote the rate at which a susceptible and a recovered individual becomes infected when coming in contact with an infected individual, respectively. The recovery rate of an individual is denoted by \(\gamma_r \in \loro{0}{1}\). The controlled SIRI model consists of the following control actions: vaccination rate of susceptible individuals denoted by \(u_v(\cdot)\), rate of social distancing adopted by individuals in the susceptible and recovered state which we denote by \(u_s(\cdot)\) and \(u_r(\cdot)\). With this premise, the controlled-SIRI dynamics are given by:
\begin{align}\label{eq:SIRI_dyn}
    \dot{x}_i (t) &= c_ s x_s(t) + c_ r x_r(t) + c_i(1-x_s(t) - x_r(t)) - c_s x_s(t) u_s(t) \nn \\&- c_r x_r(t) u_r(t) + c_v x_s(t) u_v(t) \nn \\
\dot{x}_s (t)  & =  - \beta x_s(t) (1-x_s(t) - x_s(t)) u_s(t) - x_s(t) u_v(t) \nn \\
\dot{x}_r(t) & = \gamma_r (1-x_s(t)-x_r(t)) - \hat{\beta} x_r(t) (1-x_s(t)-x_r(t))u_r(t) + x_s(t) u_v(t),
\end{align}
where \(x_i(t), x_s(t),\) and \(x_r(t)\) are the system states at time \(t\) and they represent the proportion of individuals in infected, susceptible, and recovered states. Over the set of feasible controls \(u_s(\cdot)\), \(u_r(\cdot)\), and \(u_v(\cdot)\), we considered the optimal control problem:
\begin{equation}
\label{eq:SIRI}
\begin{aligned}
& \minimize_{u_s(\cdot),\,u_r(\cdot),\,u_v(\cdot)}	&&  x_i(\tfin) \\
& \hspace{5mm} \sbjto		&&  \begin{cases}
\text{dynamics }\eqref{eq:SIRI_dyn},\\
0.1 \le u_s(t) \le 1,\, 0.1 \le u_r(t) \le 1, \\ 
0 \le u_v(t) \le 0.9, \, 0 \le x_s(t) \le 1, \\
0 \le x_r(t) \le 1, x(0) \Let \bigl(0.2\,\, 0.8\,\, 0 \bigr)^{\top}, \\
c_s =2,\,c_r=2,\,c_v=3,\,c_i=5,\,\beta=0.7,\,\hat{\beta}=0.5, \\
\gamma_r = 0.5,\,\text{and }T=20.
\end{cases}
\end{aligned}
\end{equation}
We started with a sparse \emph{uniform grid} with \(\Ninit=50\) steps on which we solved the OCP \eqref{eq:SIRI} via baseline-\(\quito\). The termination tolerance was set to \(\bar{\reftol}=10^{-3}\). The mesh refinement algorithm was successfully terminated after one iteration of localized refinement around the jump points \(\tau \Let \{1.889, 6.14, 7.24\}\); see Figure \ref{fig:SIRI_DMCS} (top left). 
\begin{figure}[htpb]
\includegraphics[scale=0.75]{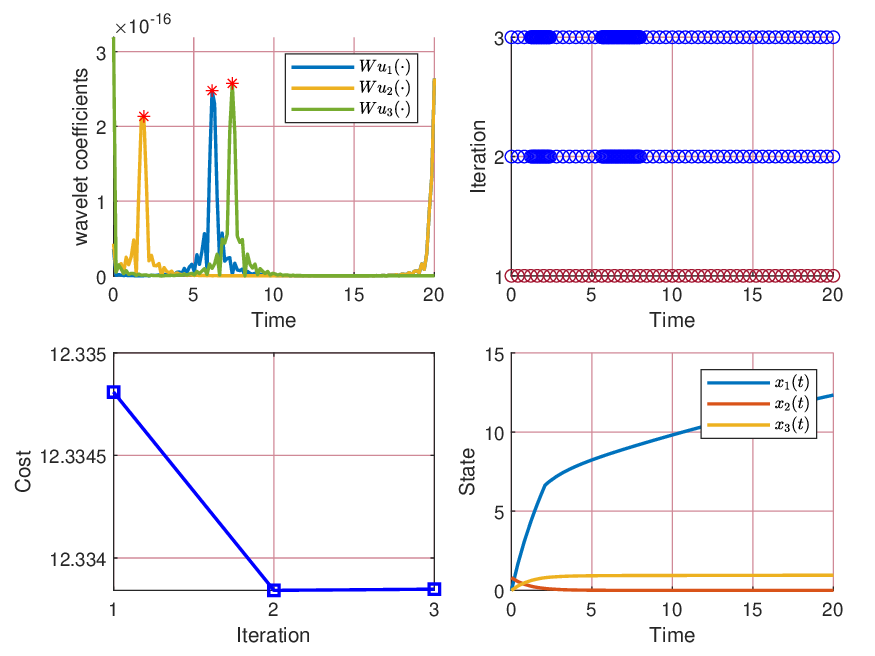}
\caption{ Wavelet coefficients for scale \(b_0 = 10^{-3}\) (top left), mesh history during refinement (top right), cost variation (bottom left), and state trajectories after refinement(bottom right)  for Example \ref{num:SIRIOCP}.}\label{fig:SIRI_DMCS}
\end{figure}
The final mesh with size \(\Nfin = 87\) is shown in Figure \eqref{fig:SIRI_DMCS} (top right). The corresponding control trajectories obtained via the baseline-\(\quito\) with \(N = 50\), \(\quito\) \(\vertwo\), LGR collocation and integrated residue minimization (both using ICLOCS2) are shown in Figure \eqref{fig:SIRIcontrol}. Table \eqref{vdo:Siri-cpu-time} collects the optimization statistics. 
\begin{figure}[htpb]
\centering
\includegraphics[scale=0.6]{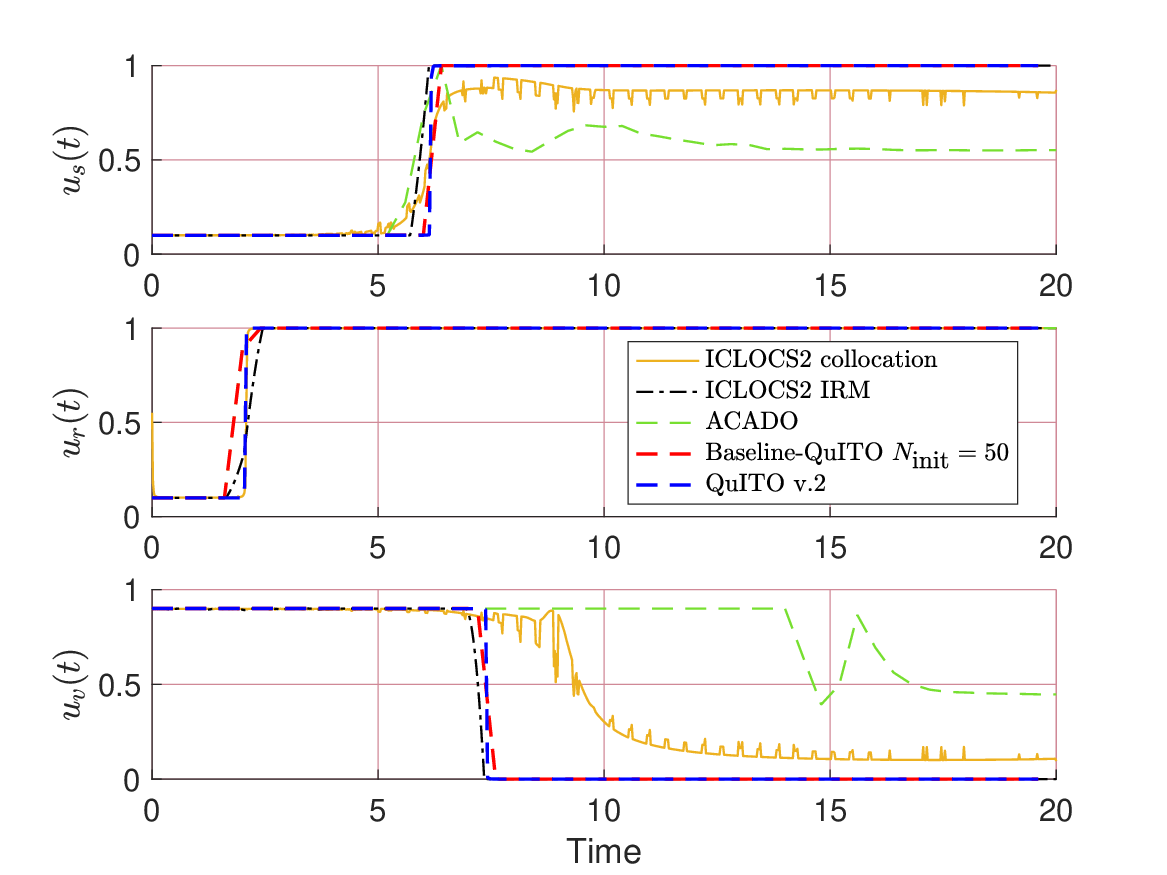}
\caption{Control trajectories obtained using \(\quito\) \(\vertwo\), ICLOCS2's collocation and IRM, and baseline-\(\quito\) for Example \ref{num:SIRIOCP}.}\label{fig:SIRIcontrol}
\end{figure}
We observed that the ICLOCS2 integrated residual method features a lower cost of solution trajectory compared to \(\quito\) \(\vertwo\), but it requires a long computation time and fails to capture and represent the jumps accurately. 

\begin{table}[htbp]
\centering
\begin{tblr}{l c c}
\hline[2pt]	
 \SetRow{azure9}
Method & \(\Ninit\) & \(\Nfin\) & Minimized cost  \\ 
\hline[2pt]
\(\quito\) \(\vertwo\) & 50 & 123 &  12.533 \\
 baseline-\(\quito\) without refinement & 200 & 200 & 12.533  \\
 ICLOCS2 collocation with \\mesh refinement (optimal control rings) & 50 & 18665  &  12.534 \\
 ICLOCS2 Integrated Residue Method  & 50 & 50 &  12.535  \\
 ACADO  & \(50\) & 50  & \(13.591\)    \\
 \hline[2pt]
\end{tblr}
\vspace{2.5mm}
\caption{A comparison of optimization statistics for Example \ref{num:SIRIOCP}.}
\label{vdo:Siri-cpu-time}
\end{table}






\subsection{Multi agent robot motion planning}\label{numexp:auv_pathpanning}
We considered a \(18\)-dimensional multi-agent path-planning and obstacle avoidance problem of an Autonomous Underwater Vehicle (AUV) using \(\quito\). In this experiment we considered three different agents \(\aset[]{\asf,\bsf,\csf}\). Let us specify all the problem data for one such agent first. The dynamical system \cite{ref:weston2024application} is given by the ODEs consisting of six state variables and six control variables: 
\begin{align}\label{eq:auv_dyn}
    \dot{x}_{\asf,1}(t) = M_{\asf}^{-1} \bigl( \widehat{x_{\asf,1}(t)}(x_{\asf,2}(t)) - D_{\asf,s} x_{\asf,1}(t)+ u_{\asf, 1}(t)  \bigr), \nn \\ \dot{x}_{\asf,2}(t) = I_{\asf}^{-1} \bigl( \widehat{I_{\asf}x_{\asf,2}(t)}(x_{\asf,2}(t)) - D_{\asf,r}x_{\asf,2}(t)+u_{\asf,2}(t) \bigr),
\end{align}
where \(t \mapsto x^{\asf}(t) \Let (x_{\asf,1}(t),x_{\asf,2}(t)) \in \Rbb^6\) and \(t \mapsto u^{\asf}(t) \Let (u_{\asf,1}(t),u_{\asf,2}(t)) \in \Rbb^6\). More precisely \(x_{\asf,1}(t) = (x_{\asf}(t),y_{\asf}(t),z_{\asf}(t))\in \Rbb^3\) represents the translation states, \(x_{\asf,2}(t) = (\rho_{\asf}(t),\theta_{\asf}(t),\psi_{\asf}(t)) \in \Rbb^3\) represents the rotational states, and \(u_{\asf,1}(t) \in \Rbb^3\), \(u_{\asf,2}(t) \in \Rbb^3\). \(M_{\asf}=10\) is the mass of the AUV, \(D_{\asf,s} \Let \mathsf{diag}(0.1,5,5)\), \(D_{\asf,r} \Let \mathsf{diag}(5,5,0.1)\) are both \(3 \times 3\) diagonal matrices that represents damping effects. The symmetric, positive definite inertia matrix is given by 
\begin{align}
    I_{\asf} \Let \begin{pmatrix}
        3.45 & 0 & 1.28 \times 10^{-15}\\ 0 & 3.45 & 0 \\ 1.28 \times 10^{-15} & 0 & 2.2 
    \end{pmatrix},
\end{align}
and the map \(\hat{\cdot}: \Rbb^3 \lra \mathfrak{so}(3)\) in \eqref{eq:auv_dyn} is the standard \emph{hat} map given by 
\[\Rbb^3 \ni \omega \Let \begin{pmatrix} \omega_1 \\ \omega_2 \\ \omega_3 \end{pmatrix} \mapsto \hat{\omega} \Let \begin{pmatrix}
    0 & \omega_3 & -\omega_2 \\ -\omega_3 & 0 & \omega_1 \\ \omega_2 & - \omega_1 & 0
\end{pmatrix} \in \mathfrak{so}(3),\]
where \(\mathfrak{so}(3)\) consisting of skew-symmetric matrices is the Lie algebra of the special orthogonal group \(\mathsf{SO}(3)\). The states are constrained in the six-dimensional box i.e., \(x^{\asf}(t) \in \lcrc{0}{15}^3 \times \lcrc{-\pi}{\pi}^3\) and the controls are restricted in the box \(u^{\asf}(t) \in \lcrc{-100}{100}^6\). We additionally imposed the path constraints that correspond to the 3D obstacles that all the agents must avoid:
\begin{equation}\label{eq:path_constraints}
    \mathsf{p}(\st_{\asf,1}(t)) \Let 
\begin{cases} 
\bigl( \frac{x_{\asf}(t)-8.3}{1+ \buffer }\bigr)^2 + \bigl( \frac{y_{\asf}(t)-8.3}{1+  \buffer }\bigr)^2 + \bigl( \frac{z_{\asf}(t)-8}{1+  \buffer}\bigr)^2 \ge 1,\\

\bigl( \frac{x_{\asf}(t)-3}{2+ \buffer}\bigr)^2 + \bigl( \frac{y_{\asf}(t)-6.5}{2+ \buffer}\bigr)^2 + \bigl( \frac{z_{\asf}(t)- 5}{2+ \buffer}\bigr)^2 \ge 1, \\
\bigl( \frac{x_{\asf}(t)-7.3}{2+ \buffer}\bigr)^2 + \bigl( \frac{y_{\asf}(t)-2.7}{2+ \buffer}\bigr)^2 + \bigl( \frac{z_{\asf}(t)-5}{2+ \buffer}\bigr)^2 \ge 1, \\
\bigl( \frac{x_{\asf}(t)-10}{2+ \buffer}\bigr)^2 + \bigl( \frac{y_{\asf}(t)-10}{2+ \buffer}\bigr)^2 + \bigl( \frac{z_{\asf}(t)-2}{2+ \buffer}\bigr)^2 \ge 1, \\

\bigl( \frac{x_{\asf}(t)-1.7}{1+ \buffer}\bigr)^{10} + \bigl( \frac{y_{\asf}(t)-1.7}{1+ \buffer}\bigr)^{10} + \bigl( \frac{z_{\asf}(t)-2}{1+ \buffer}\bigr)^{10}\ge 1, \\

\bigl( \frac{x_{\asf}(t)-2}{1.3+ \buffer}\bigr)^{10} + \bigl( \frac{y_{\asf}(t)-14}{2.3+ \buffer}\bigr)^{10} + \bigl( \frac{z_{\asf}(t)-5}{5+ \buffer}\bigr)^{10}\ge 1, \\

\bigl( \frac{x_{\asf}(t)-14}{1.3+ \buffer}\bigr)^{10}+ \bigl( \frac{y_{\asf}(t)-1}{2.3+ \buffer}\bigr)^{10} + \bigl( \frac{z_{\asf}(t)-5}{5+ \buffer}\bigr)^{10}\ge 1.

\end{cases}
\end{equation}
The first four path constraints in \eqref{eq:path_constraints} are ellipsoids and the subsequent three are cuboidal constraints. To mimic a practical scenario we consider all the AUV bodies to have a radius of \(0.5\) meters with a buffer radius \(0.2\) meters, making \(\buffer \Let 0.7\) meters. We denote the total state and control vector of the three agent system by \(t \mapsto \widecheck{x}(t) \Let (x^{\asf}(t),x^{\bsf}(t),x^{\csf}(t)) \in \Rbb^{18}\) and \(t \mapsto \widecheck{u}(t) \Let (u^{\asf}(t),u^{\bsf}(t),u^{\csf}(t)) \in \Rbb^{18}\). With these ingredients, we considered the multi-agent optimal control problem:
\begin{equation}
\label{eq:multiagent_RobMotion}
\begin{aligned}
& \minimize_{\widecheck{u}(\cdot)}	&&  \int_{0}^{10} \inprod{\widecheck{x}(t)}{Q \widecheck{x}(t)} + \inprod{\widecheck{u}(t)}{R\widecheck{u}(t)}\, \dd t \\
&  \sbjto		&&  \begin{cases}
\text{dynamics }\eqref{eq:auv_dyn} \text{ for all the agents }\asf,\bsf,\csf,\\
x^{\asf}(0)\Let (10_{3\times 1},0_{3\times 1}),  x^{\bsf}(0) \Let (6,10,6,0_{3\times 1}), \\ x^{\csf}(0) \Let (10,6,10,0_{3\times 1}),
x^{\asf}(10)\Let (0_{6\times 1}), \\ x^{\bsf}(10) \Let (2,0_{5\times 1}), x^{\csf}(10) \Let (0,2,0_{4\times 1}), \\
x^{i}(t) \in \lcrc{0}{15}^3 \times \lcrc{-\pi}{\pi}^3, u^{i}(t) \in \lcrc{-100}{100}^6 \,\text{for }i \in \aset[]{\asf,\bsf,\csf},\\
\text{all agents satisfy path constraints }\mathsf{p}(x_{\asf,i}(t)) \,\text{for }i=\asf,\bsf,\csf,\\
\norm{x^{i}(t)-x^{j}(t)} \ge \sqrt{2} \,\,\text{for all }i,j \in \aset{\asf,\bsf,\csf},
\end{cases}
\end{aligned}
\end{equation}
where \(\Rbb^{18 \times 18} \ni Q \Let \mathsf{diag}(Q_{\asf},Q_{\bsf},Q_{\csf})\) with \(Q_{\asf} = Q_{\bsf} = Q_{\csf} = \mathsf{diag}(1_{3\times 1},50_{3 \times 1})\) and \(\Rbb^{18 \times 18} \ni R \Let \mathsf{diag}(R_{\asf},R_{\bsf},R_{\csf})\) with \(R_{\asf} = R_{\bsf} = R_{\csf} = \mathsf{diag}(1,10,10,1,5,10)\). For simplicity, we kept the mass, inertia matrices, and the damping matrices same for all the AUV agents. 
\begin{figure}[h]
\centering
\includegraphics[scale=0.43]{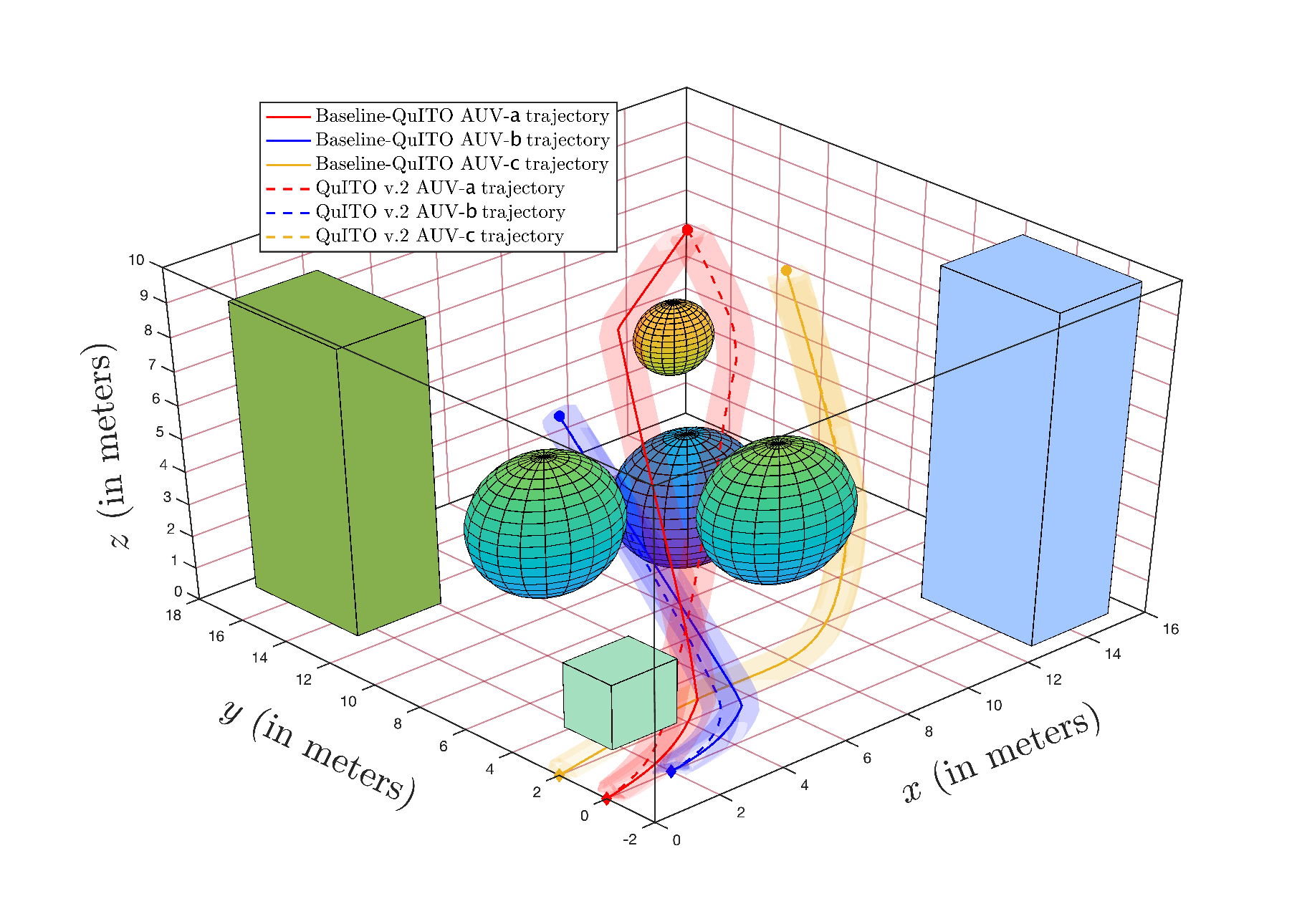}
\caption{Trajectories of the AUVs \(\asf,\bsf,\csf\) for Example \ref{numexp:auv_pathpanning}.}\label{fig:AUVpath}
\end{figure}
We solved the OCP \eqref{eq:multiagent_RobMotion} using the baseline-\(\quito\) with \(\Ninit = 30\) and \(\Dd=1.07\). Next, we employed the localization algorithm \ref{alg:detection_algo} to locate the possible change points in the approximate control trajectory. We found that the wavelet maximas occur at \(\tau = \aset[]{0.157,4.488,9.3}\); see Figure \ref{fig:AUVdistwavelet} (left). To keep the computational burden low we refined only near the first initial jump which has a higher magnitude compared to others and set \(\bar{\reftol}=10\).
\begin{figure}[h!]
\includegraphics[scale=0.5]{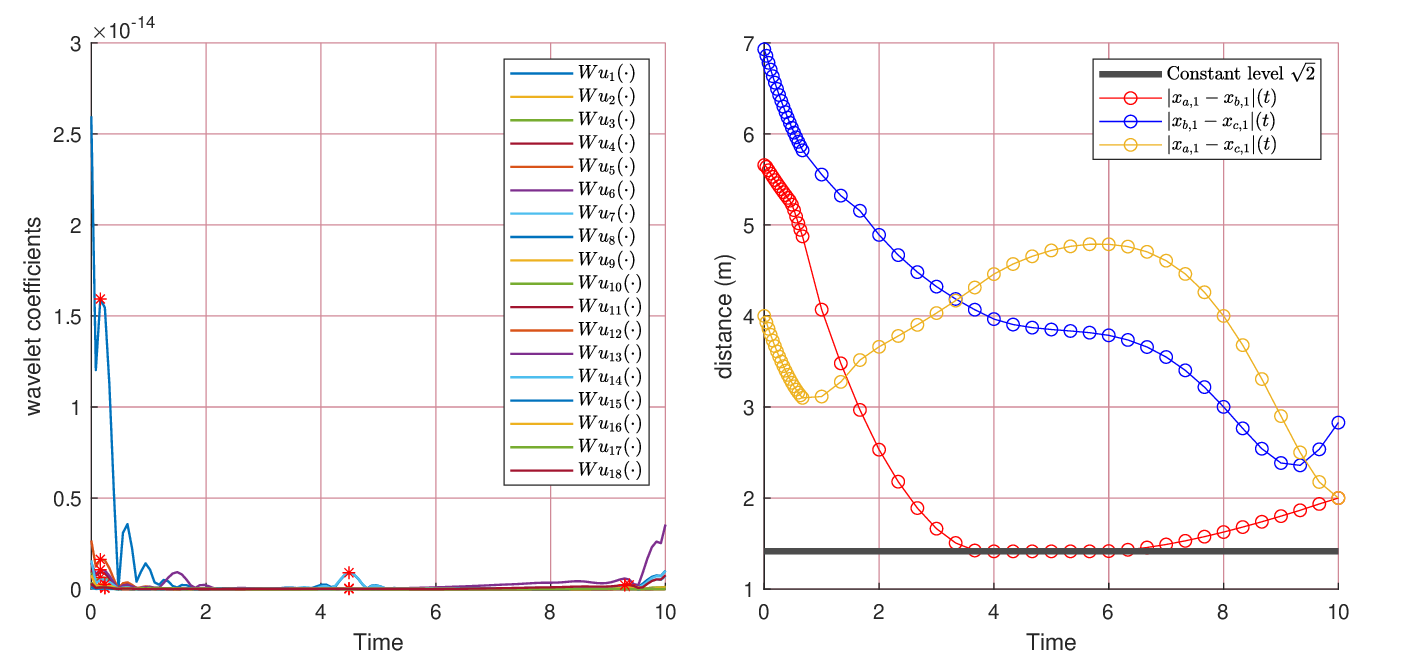}
\caption{Wavelet coefficients for scale \(b_0 = 10^{-3}\) (left) and intermediate distance between the AUVs (right) for Example \ref{numexp:auv_pathpanning}.}\label{fig:AUVdistwavelet}
\end{figure}
We observed the specified error tolerance was met after two iterations and the algorithm terminated successfully, with minimized cost \(6062.4\) for baseline-\(\quito\) and \(5205.3\) with \(\quito\) \(\vertwo\). The total computation time was \(268.32\) seconds for baseline-
\(\quito\) and \(885.04\) seconds with \(\quito\) \(\vertwo\). In baseline-\(\quito\) the AUV-\(\asf\) and AUV-\(\bsf\) both collide with the ellipsoidal constraints while the AUV trajectories in \(\quito\) \(\vertwo\) satisfies all the constraints with a reduced amount of cost, giving a clear indication of better performance with localization and mesh refinement; see Figure \ref{fig:AUVpath}.\footnote{An animation of the AUV trajectories can be found in the YouTube video \url{https://youtu.be/CPDhne9kNHk}.}






\subsection{Discussion on CPU times}\label{subsec:CPUtimediscussion}
This section presents a comparison of computation times of various methods that we employed in the numerical experiments with \(\quito\) \(\vertwo\). From the data recorded in Table \ref{tab:computation_list} we make the following observations: 
\begin{itemize}[leftmargin=*]
    \item \(\quito\) \(\vertwo\) shows improved performance in comparison to baseline-\(\quito\) both in terms of solution accuracy and computation time for all the numerical examples considered above.
    \item For singular optimal control problems, when compared to collocation techniques \(\quito\) \(\vertwo\) shows better performance in terms of solution accuracy but collocation techniques are faster in general (of course, ringing cannot be avoided); see \S\ref{num:bressan} and \S\ref{num:catalystmixOCP}.
    \item For singular problems, while performance with IRM improves when compared to collocation, it is still subpar compared to \(\quito\) \(\vertwo\) both in terms of solution accuracy and computation time;  see \S\ref{num:bressan} and \S\ref{num:catalystmixOCP}.
    \item The multiple shooting algorithm with piecewise constant control parameterization is much faster than \(\quito\) \(\vertwo\). This is natural since the software is optimized, and written in \(\text{C}\)++. But in terms of performance, as demonstrated in the numerical examples, \(\quito\) \(\vertwo\), produces more accurate control trajectories in examples treated in \S\ref{num:bangbangOCP}, \S\ref{num:bressan}, \S\ref{num:catalystmixOCP}, and \S\ref{num:SIRIOCP}. 
\end{itemize}

\begin{table}[h!]
\begin{tblr}{l| c| c| c}
\hline[2pt]
\SetRow{azure9}
Problem & Method & \(\Ninit\) & CPU time (in secs) \\ 
\hline[2pt]
Example \eqref{num:bangbangOCP} &  \(\quito\) \(\vertwo\) & \(50\)  & \(4.42\) \\
(Bang-bang problem) &  ICLOCS2 collocation & \(50\) & \(11.97\) \\ 
 &  ICLOCS2 IRM & \(50\) & \(15.90\) \\
 &  ACADO & \(50\) & \(4.94\) \\
 \hline
Example \eqref{num:bressan} & \(\quito\) \(\vertwo\) & \(20\) & \(1.61\) \\ 
(Bressan problem) &  baseline-\(\quito\) & \(200\) & \(69.20\) \\
 &  ICLOCS2 collocation & \(20\) & \(21.42\) \\ 
 &  ICLOCS2 IRM & \(20\) & \(41.76\) \\
 &  ACADO & \(45\) & \(3.86\) \\
 &  ACADO & \(200\) & \(7.47\) \\
 \hline
Example \eqref{num:catalystmixOCP}  & \(\quito\) \(\vertwo\) & \(100\) & \(16.41\) \\ 
(Catalyst-mixing problem) &  baseline-\(\quito\) & \(200\) & \(95.14\) \\
 &  ICLOCS2 collocation & \(100\) & \(221.21\) \\ 
 &  ICLOCS2 IRM & \(100\) & \(3106.70\) \\
 &  ACADO & \(100\) & \(6.71\) \\
 \hline
Example \eqref{num:SIRIOCP} & \(\quito\) \(\vertwo\) & \(50\) & \(86.19\) \\ 
(SIRI problem) &  baseline-\(\quito\) & \(200\) & \(726.90\) \\
 &  ICLOCS2 collocation & \(50\) & \(146.14\) \\ 
 &  ICLOCS2 IRM & \(50\) & \(1387.20\) \\
 &  ACADO & \(50\) & \(5.76\) \\
 \hline
Example \eqref{numexp:auv_pathpanning}  & \(\quito\) \(\vertwo\) & \(30\) & \(885.04\) \\   
(Multi-agent problem) & & \\
\hline[2pt]
\end{tblr}
\centering
\vspace{2.5mm}
\caption{A comparison of CPU times.}
\label{tab:computation_list}
\end{table}


\section{\(\quito\) \(\vertwo\): the software architecture and functionalities}\label{sec:software}
We encoded the direct transcription, localization, and refinement algorithm developed herein into the previous version of the software \(\quito\) \cite{ref:QuITO:SoftX}, and we call the new package \(\quito\) \(\vertwo\). The toolbox \(\quito\) \(\vertwo\) employs CasADi \cite{ref:CASADI_andersson2019} for the declaration of symbolic variables and to interface optimization solvers. The software package consists of four core blocks:
\begin{enumerate}[leftmargin=*]
    \item The directory \texttt{TemplateProblem} contains the template of the software. Users can solve their optimal control problems by adding relevant problem data, for example, the cost functions, dynamics, constraints, etc.

    \item The directory \texttt{examples} contains a library of pre-coded examples.  

    \item The directory \texttt{src} contains scripts for transcription, along with the change point localization, and mesh refinement algorithms developed in this work. The directory \texttt{src}  must be added to the MATLAB path to run the software.

    \item \(\quito\) \(\vertwo\) also comes in with a Graphical User Interface. The directory \texttt{Graphical Interface} builds the GUI for a one-click solution to the example problems. 
\end{enumerate}
\vspace{2mm}
A schematic diagram of the software is provided in Figure \eqref{fig:QuITO_flowchart}.
\begin{figure}[htbp]
\centerline{\includegraphics[scale=0.6]{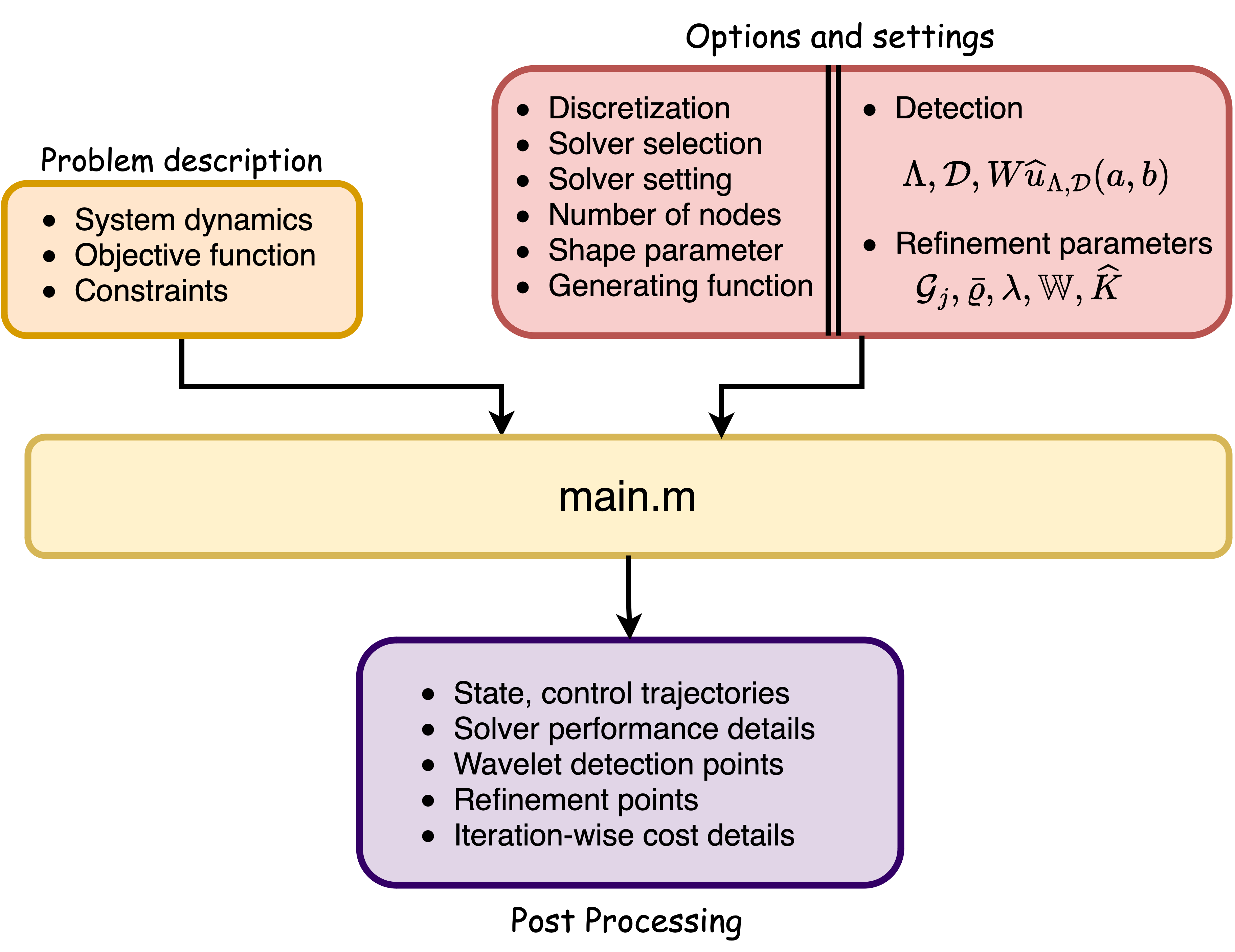}}
\caption{Flowchart of \(\quito\) \(\vertwo\).}
\label{fig:QuITO_flowchart}
\end{figure}
\(\quito\) \(\vertwo\) employs the same architecture as in the previous version \(\quito\) \cite{ref:QuITO:SoftX}. A user-specific optimal control problem can be built using the files  \texttt{TemplateProblem.m}, \texttt{options.m}, \texttt{main.m}, and \texttt{postProcess.m} contained in the \texttt{TemplateProblem} directory. The problem data needs to be defined in \texttt{TemplateProblem.m} and the solver settings such as choice of discretization, NLP solver, etc., need to be specified in the \texttt{options.m} file. 
The meshing strategy can be chosen by specifying the \texttt{options.mesh\_strategy} parameter. For selective mesh refinement, the refinement parameters \(\bar{\reftol}, \width, \lambda\), and \(\widehat{K}\) can be set using the commands :
\begin{lstlisting}[
  style      = Matlab-editor,
  basicstyle = \mlttfamily,
   numbers = none,
]
options.MR_termination_tol= 0.0001;
options.MR_width_factor= 21; % fraction of the time horizon (T)
%Refinement width parameter = T/options.MR_width_factor
options.MR_rate=4; 
options.MR_max_iter=10; 
\end{lstlisting}
To solve the OCP, the user needs to input the arguments \((N,\Dd)\) to the \texttt{options} function in the \texttt{main.m} file. The first argument \(N>0\) is the total number of steps by which the time horizon \(\lcrc{0}{T}\) has been divided (see \S\ref{subsubsec:meshrf_scheme}), which, when using mesh refinement, is naturally the size of the initial uniform grid. The second parameter \(\Dd\in \loro{0}{+\infty}\) is the shape parameter of the generating function \(\genfn(\cdot)\). While our theoretical setup considers only the Gaussian generating function, in the software \(\quito\) \(\vertwo\) we kept several other options for choosing \(\genfn(\cdot)\). These options include Laguerre polynomial-based Gaussian, trigonometric Gaussians, hyper trigonometric functions, etc. The file \texttt{main.m} accesses the problem parameters and optimization specifications from files \texttt{TemplateProblem.m}, and \texttt{options.m}, and computes the control trajectory and the state trajectory using the \texttt{src/solveProblem.m} file (available in the path) and subsequently the localization and refinement algorithms are employed depending on the regularity of the control trajectory. Finally \texttt{postProcess.m} plots the refined control, and state trajectories.
\begin{lstlisting}[
  style      = Matlab-editor,
  basicstyle = \mlttfamily,
   numbers = none,
]
problem = TemplateProblem; 
opts = options(100, 2);   
solution = solveProblem(problem, opts);
postProcess(solution, problem, opts)
\end{lstlisting}
The user can choose to plot the iterative cost variation and mesh history by appropriately setting the \texttt{options.MRplot} parameter in the \texttt{options.m} file:
\begin{lstlisting}[
  style      = Matlab-editor,
  basicstyle = \mlttfamily,
   numbers = none,
]
% Mesh refinement plot options
% 0: Do not plot
% 1: Plot mesh history
% 2: Plot iterative cost variation
% 3: Plot all above
options.MRplot=3;
\end{lstlisting}

\section{Discussion and concluding remarks}
We introduced a complete algorithmic package --- \(\quito\) \(\vertwo\) --- a direct trajectory optimization algorithm for uniform approximations of constrained optimal control trajectories driven by quasi-interpolation and adaptive mesh refinement. We developed a change point localization technique based on wavelets and subsequently established a mesh refinement algorithm driven by this localization technique. Several benchmark constrained OCPs were numerically solved to demonstrate the efficacy of \(\quito\) \(\vertwo\). It was demonstrated that for classes of \emph{singular} optimal control problems, which are well-known hard-to-solve problems in numerical optimal control, \(\quito\) \(\vertwo\) works remarkably well compared to state-of-the-art tools based on pseudospectral collocation, integrated residue minimization technique, and ACADO. Moreover, an updated version of the software \(\quito\), which we call \(\quito\) \(\vertwo\), was developed along with a GUI module for ease of implementation and has been made available at:
\begin{table}[htbp]
\begin{tblr}{l}
	\hline[{red6},3pt]
		\SetRow{red9}\SetCell[c=3]{c} \url{https://github.com/chatterjee-d/QuITOv2.git}
& 3-2
& 3-3
 \\ 
 \hline[{red6},3pt]
\end{tblr}
\centering
\label{gitrepo}
\end{table}

\section{Acknowledgment}
S.\ G.\ and D.\ C.\ thank Ravi N.\ Banavar for his comments and suggestions during the initial stages of this work, and Rohit Gupta, Mayank Baranwal, and Ashish Hota for helpful discussions and encouragement. R.\ A.\ D.\ thanks Mukesh Raj for assistance with the GUI implementation of \(\quito\) \(\vertwo\) and Yuanbo Nie from the Dept of Automatic Control and Systems Engineering, University of Sheffield, for his assistance over email with the implementation of an example using ICLOCS2.

\bibliographystyle{amsalpha}
\bibliography{refs_new}

\end{document}